\documentclass[11pt]{article}
\usepackage[utf8]{inputenc}
\usepackage{amsmath,amsthm,amssymb,amsfonts,amscd}
\usepackage{latexsym}
\usepackage{mathrsfs}
\usepackage{cite}
\usepackage[title]{appendix}
\usepackage{color,enumitem,graphicx}
\usepackage[colorlinks=true,urlcolor=black,
citecolor=red,linkcolor=blue,linktocpage,pdfpagelabels,
bookmarksnumbered,bookmarksopen,
pdftitle={Morse index and spectral asymptotics of least-energy solutions to the Choquard equation in planar domains},
pdfauthor={Jiaoping Chen, Wenjing Chen and Shengbing Deng},
pdfsubject={Morse index and spectral asymptotics for the planar Choquard equation},
pdfkeywords={Choquard equation; Morse index; spectral asymptotics; nondegeneracy; Robin function}]{hyperref}

\makeatletter \@addtoreset{equation}{section} \makeatother

\newtheorem*{theoremA}{Theorem A}
\newtheorem{theorem}{Theorem}[section]

\newtheorem{definition}{Definition}[section]
\newtheorem{proposition}{Proposition}[section]
\newtheorem{lemma}{Lemma}[section]
\newtheorem{remark}{Remark}[section]

\numberwithin{equation}{section}

\allowdisplaybreaks[1]

\begin{document}
\title{Morse index and spectral asymptotics of least-energy solutions to the Choquard equation in planar domains}

\author{Jiaoping Chen, Wenjing Chen and Shengbing Deng\\
\footnotesize School of Mathematics and Statistics, Southwest University,
Chongqing, 400715, P.R. China\\
\footnotesize E-mail: {\tt cjp123888@email.swu.edu.cn} (J. Chen),
{\tt wjchen@swu.edu.cn} (W. Chen),\\
\footnotesize {\tt shbdeng@swu.edu.cn} (S. Deng; corresponding author)}

\date{ }
\maketitle

\begin{abstract}
We investigate the Morse index and spectral asymptotics of least-energy
solutions to the planar Choquard equation
\begin{equation*}
\begin{cases}
 -\Delta u = \displaystyle\left(\int_{\Omega} \frac{u^{p+1}(y)}{|x - y|^\alpha}dy\right) u^p & \text{in } \Omega, \\
u > 0 & \text{in } \Omega, \\
u = 0 & \text{on } \partial \Omega,
\end{cases}
\end{equation*}
where $\Omega\subset\mathbb R^2$ is a smooth bounded domain and
$0<\alpha<1$. Under some geometric assumptions, we obtain sharp asymptotics for the first four eigenpairs
of the linearized problem as $p\to+\infty$.
The first eigenvalue is exactly $(2p+1)^{-1}$. The second and third
eigenfunctions give the translation modes of the limiting bubble, and
their eigenvalue expansions are determined by the Hessian of the Robin
function at the concentration point. The fourth eigenfunction gives the
dilation mode with eigenvalue
$1+3(4-\alpha)/(2p)+o(p^{-1})$. As a consequence, we derive a comparison between
the spectral indices of the solution and the critical-point indices of
the Robin function.
The proofs combine
variational estimates, blow-up analysis and nonlocal Poho\v{z}aev identities.

\vspace{.2cm}

\vspace{.2cm}
\end{abstract}

\section{Introduction}\label{se1}

We are interested in the Morse index and spectral asymptotics of least-energy
solutions to the following Choquard equation in a smooth bounded domain
$\Omega\subset\mathbb R^2$
\begin{equation}\label{maineq}
\begin{cases}
-\Delta u = \left(\displaystyle \int_{\Omega} \frac{u^{p+1}(y)}{|x - y|^\alpha}dy\right) u^p & \text{in } \Omega, \\
u > 0 & \text{in } \Omega, \\
u = 0 & \text{on } \partial \Omega,
\end{cases}
\end{equation}
where $0<\alpha<1$ and $p>1$.

The well-known Lane--Emden problem
\begin{equation}\label{Lane}
\begin{cases}
-\Delta u = u^q & \text{in } \Omega, \\
u > 0 & \text{in } \Omega, \\
u = 0 & \text{on } \partial \Omega
\end{cases}
\end{equation}
serves as a fundamental model in astrophysics for describing self-gravitating spheres and stellar structures;
for further physical background, we refer to \cite{Cha,Hor}.
For any smooth bounded domain $\Omega\subset\mathbb{R}^N$ with $N\ge 2$,
problem \eqref{Lane} admits at least one solution for every $q>1$
(and $q<\frac{N+2}{N-2}$ if $N\ge3$).
Such solutions can be constructed variationally as minimizers of the energy functional on the Nehari manifold,
known as \emph{least-energy solutions}.

The uniqueness and multiplicity of solutions to \eqref{Lane} depend delicately on the geometry of $\Omega$ and the exponent $q$;
see, e.g., \cite{Bat,Dam1,Gla,Li,Lin,Mck}.
The two-dimensional case is particularly distinctive, since there is no Sobolev critical exponent,
and the asymptotic behavior of solutions as $q\to+\infty$ requires independent analysis.
In \cite{DeM2}, De~Marchis et al. established a complete characterization of the asymptotic behavior of solutions $u_q$ to \eqref{Lane}
under the uniform energy bound
\begin{equation*}
\sup_{q} q \|\nabla u_q\|_{L^2(\Omega)}^2 \leq C.
\end{equation*}
Sign-changing solutions were investigated in \cite{DeM1,DeM3}.
Subsequently, the authors in \cite{DeM4} refined these results and derived precise asymptotics for the $L^\infty$-norm of solutions.

Building on \cite{DeM2,DeM4}, De~Marchis et al. \cite{De} analyzed the linearized eigenvalue problem
\begin{equation*}
\begin{cases}
-\Delta \phi = \lambda q_{n} u_{q_n}^{q_n - 1} \phi  &\text{in } \Omega, \\
\phi = 0  &\text{on } \partial\Omega, \\
\|\phi\|_{L^\infty(\Omega)} = 1,
\end{cases}
\end{equation*}
and combining the Morse index analysis of single-spike solutions with
a priori estimate, proved uniqueness in convex planar domains for large $q$.
Using ODE methods and local Poho\v{z}aev identities, Grossi et al. \cite{Gro} established the nondegeneracy of multi-spike solutions
and local uniqueness of single-spike solutions for general domains.
Recently, Ianni et al. \cite{Ian1} studied the Morse index and Leray-Schauder degree
of multi-spike solutions to \eqref{Lane}, and derived a new local uniqueness result.
For more results on problem \eqref{Lane}, see \cite{DeM6,DeM5,Kam} and the references therein.

In this work, we focus on the nonlocal problem \eqref{maineq}, which is closely related to the  Choquard equation
\begin{equation}\label{wen1.7}
-\Delta u +V(x) u = \left( \int_{\mathbb{R}^N} \frac{|u(y)|^l}{|x - y|^{\alpha}} dy \right) |u|^{l-2}u \quad \text{in } \mathbb{R}^N.
\end{equation}
For $l=2$, $\alpha=1$, $V(x)\equiv1$ and $N=3$, this equation originates from Pekar's quantum theory of the polaron at rest \cite{Pek}.
In recent decades, extensive research has been devoted to the existence and qualitative properties of solutions to \eqref{wen1.7}.
By variational methods, the existence and uniqueness of the ground states were obtained by Lieb \cite{Lieb1}. Subsequently,
Lions \cite{Lio} proved the existence of infinitely many radially symmetric solutions.
Regularity, positivity, radial symmetry, and decay properties of ground states were established in \cite{Ma,Mor}.
Moroz and Van Schaftingen \cite{Mor} identified the lower and upper critical exponents
$$2_{\alpha,*} = \frac{2N-\alpha}{N}, \quad 2_\alpha^* = \frac{2N-\alpha}{N-2},$$
which determine existence and nonexistence for $N\ge 3$ and $\alpha\in(0,N)$.
For the upper critical case, Gao and Yang \cite{Gao2} studied the existence, multiplicity
and nonexistence of solutions to the following Br\'ezis--Nirenberg type problem
\begin{equation*}
\begin{cases}
-\Delta u = \left( \displaystyle\int_{\Omega} \frac{u^{2^*_\alpha}(y)}{|x-y|^\alpha} dy \right) u^{2^*_\alpha - 1} + \varepsilon u &\text{in } \Omega, \\
u>0 &\text{in }\Omega,\\
u = 0 &\text{on } \partial\Omega,
\end{cases}
\end{equation*}
where $\Omega\subset\mathbb{R}^N$  ($N\ge3$) is a smooth bounded domain.
As $\varepsilon\to0$, solutions blow up at a critical point of the Robin function \cite{Yang31}.
Via Lyapunov--Schmidt reduction, Yang et al. \cite{Yang30} constructed solutions concentrating at critical points of the Robin function.
In dimensions $N\ge5$, Pan et al. \cite{Pan} related the Morse index of blow-up solutions to the Hessian of the Robin function.
Cannone et al. \cite{Can} studied the corresponding slightly subcritical
problem for the Newtonian kernel $\alpha=N-2$ in dimensions $N=3,4,5$.
For various choices of $V(x)$ and $l$, many results are available for $N\ge3$; see \cite{Alv,Chen,Chen1,Chen2,Guol} and the references therein.

The two-dimensional case exhibits entirely different behavior.
The Sobolev embedding $H_0^1(\Omega)\hookrightarrow L^q(\Omega)$ holds for all $q\ge1$,
but $H_0^1(\Omega)$ does not embed into $L^\infty(\Omega)$.
For $N=2$, $\alpha\in(0,2)$, $l=p+1$ and $V(x)\equiv0$,
problem \eqref{maineq} is the bounded-domain Dirichlet counterpart of
\eqref{wen1.7}.
Gao et al. \cite{Gao} proved that least-energy solutions of problem \eqref{maineq} satisfy
$$\lim_{p \to +\infty} p \int_\Omega |\nabla u_{p}|^2 dx = 2(4-\alpha)\pi e,$$
and do not exhibit complete blow-up or vanishing.
The classification theorem of Gluck \cite{Glu} identifies the normalized
limiting profile as
\begin{equation}\label{eq:2.17}
U(x):=-\frac{4-\alpha}{2}
\log\left(1+C_\alpha^{-2}|x|^2\right),
\end{equation}
where $ C_\alpha = \left( \frac{(2 - \alpha)(4 - \alpha)}{\pi} \right)^{\frac{1}{4 - \alpha}}$ and the function $U$ is the unique solution to the limiting equation associated with \eqref{maineq} as $p\to+\infty$:
\begin{equation}\label{pce}
-\Delta u = \left( \int_{\mathbb{R}^2} \frac{e^{u(y)}}{|x - y|^\alpha} dy \right) e^{u(x)} \quad \text{in } \mathbb{R}^2.
\end{equation}
Gao et  al.  \cite{Gao1}  also established the non-degeneracy of solutions to  \eqref{pce} by combining the integral
representation of solutions with spherical harmonic decomposition.

Motivated by the local theory in \cite{De} and the higher-dimensional
Hartree theory in \cite{Can,Pan}, we address the following question:
\begin{itemize}
\item[$(\mathcal P)$] What are the sharp asymptotics of the eigenpairs of
the linearized problem at $u_p$, and how do they relate its Morse and
augmented Morse indices to the Robin function at the concentration point?
\end{itemize}

We now recall the geometric assumptions on $\Omega$ from \cite{Gao}:
\begin{itemize}
    \item[$(\mathcal{H}_1)$:] Let $R_\Omega = \sup\{ R : B_R(x) \subset \Omega \text{ for some } x \in \Omega \}$. Then
    $$R_\Omega \ge \left( \frac{2(4-\alpha)\pi}{\tilde{C}_\alpha} \right)^{\frac{1}{4-\alpha}}, \quad
    \tilde{C}_\alpha = \int_{B_1(0)} \int_{B_1(0)} \frac{1}{|x-y|^\alpha} dx dy.$$
    \item[$(\mathcal{H}_2)$:] There exists $y\in\Omega$ such that $\langle x-y, \nu(x) \rangle > 0$ for all $x\in\partial\Omega$ and
    $$\int_{\partial\Omega} \frac{1}{\langle x-y, \nu(x) \rangle} d\sigma_x < 2\pi e.$$
\end{itemize}

The following statement collects
\cite{Gao} and the immediate
consequences used below.

\begin{theoremA}\phantomsection
\makeatletter
\def\@currentlabel{A}
\makeatother\label{TA1}
Let $\alpha\in(0,1)$ and let $\Omega\subset\mathbb{R}^2$ be a smooth bounded domain satisfying $(\mathcal{H}_1)$--$(\mathcal{H}_2)$.
Let $(u_p)$ be a family of least-energy solutions to \eqref{maineq}.
Then there exists one point $x_\infty \in \Omega$ and a sequence $p_n \to +\infty$ as $n \to +\infty$ such that setting
 $$\mathcal{S}:=\{x_\infty\},$$ one has
\begin{itemize}
 \item[$(1)$] The asymptotic profile satisfies
\begin{align}
 \lim_{n \to +\infty} \sqrt{p_n} u_{p_n} &= 0 \quad \text{in } C^1_{\text{loc}}
(\overline{\Omega} \setminus \mathcal{S}), \label{eq:2.7} \\
\lim_{n \to +\infty} p_n u_{p_n}(x) &= 2(4-\alpha)\pi\sqrt{e}
 G(x, x_{\infty}) \quad \text{in } C^2_{\text{loc}}(\overline{\Omega} \setminus \mathcal{S}). \label{eq:2.8}
\end{align}
    Here $G$ is the Dirichlet Green function of $-\Delta$ in $\Omega$.

    \item[$(2)$] The energy satisfies
   \begin{equation*}
\lim_{n \to +\infty} p_n \int_\Omega |\nabla u_{p_n}(x)|^2 dx = 2(4-\alpha)\pi e,
\end{equation*}
\begin{equation}\label{eq:2.91}
1 \leq \liminf_{n \to +\infty} \| u_{p_n} \|_{L^\infty(\Omega)} \leq \limsup_{n \to +\infty} \| u_{p_n} \|_{L^\infty(\Omega)} \leq \sqrt{e}.
\end{equation}

    \item[$(3)$] The point $x_\infty$ is a critical point of the Robin
function $R(x)=H(x,x)$, i.e.,
\begin{equation}\label{CritialR}
\nabla R(x_\infty) = 0.
\end{equation}

\item[$(4)$] The maximum points fixed above satisfy
\begin{equation}\label{eq:2.12}
\lim_{n \to +\infty} x_{n} = x_{\infty},
\end{equation}
\begin{equation}\label{eq:2.13}
\lim_{n \to +\infty} u_{p_n}(x_{n})
=\lim_{n\to+\infty}\| u_{p_n} \|_{L^\infty(\Omega)}=\sqrt{e}.
\end{equation}

    \item[$(5)$] Define the rescaled sequence
   \begin{equation}\label{eq:2.14}
 \varepsilon_{n}:= \left( p_n u_{p_n}^{2p_n}(x_n) \right)^{-\frac{1}{4-\alpha}}
\rightarrow0\quad \mathrm{as}~n\rightarrow+\infty,
\end{equation}
and set
\begin{equation}\label{eq:2.15}
v_n(x):=\frac{p_n}{u_{p_n}(x_n)}
\left(u_{p_n}(\varepsilon_nx+x_n)-u_{p_n}(x_n)\right),
\qquad x\in\Omega_n:=\frac{\Omega-x_n}{\varepsilon_n}.
\end{equation}
Then
\begin{equation}\label{eq:2.16}
\lim_{n \to +\infty} v_{n} = U \quad \text{in } C^2_{\mathrm{loc}}(\mathbb{R}^2),
\end{equation}
where $U$ is defined by \eqref{eq:2.17} and satisfies
\begin{equation*}
\displaystyle
\int_{\mathbb{R}^2} \displaystyle\int_{\mathbb{R}^2}
\frac{e^{U(y)} e^{U(x)}}{|x - y|^\alpha} dx dy = 2(4 - \alpha)\pi.
\end{equation*}
\end{itemize}
\end{theoremA}

Choose $r>0$ sufficiently small such that $\overline{B_{4r}(x_\infty)}\subset\Omega$. Since $x_n\to x_\infty$, we have $\overline{B_{2r}(x_n)}\subset\Omega$ for all sufficiently large $n$. The radius $r$ is independent of $n$ and may be reduced if necessary.

\begin{remark}\label{remark:eps-small}
{\rm The bubble extraction in \cite[Proposition 4.2]{Gao} starts at the
global maximum $x_n$ and produces $k$ families of mutually separated
bubbles. In fact, $k=1$ here. To see this, passing to a subsequence, let $M_j$ denote the limiting peak heights, with $M_1=\sqrt e$ and $M_j\ge1$. Rescaling on disjoint peak neighborhoods and using positivity of the kernel, we obtain
\[
2(4-\alpha)\pi e
=\lim_{n\to\infty}p_n\int_\Omega|\nabla u_{p_n}|^2
\geq2(4-\alpha)\pi\sum_{j=1}^kM_j^2
\geq2(4-\alpha)\pi(e+k-1).
\]
Thus $k=1$. Now \cite[Proposition 4.3]{Gao} gives a constant $C>0$
such that
\begin{equation}\label{eq:global-gradient}
p_n\,|x-x_n|\,|\nabla u_{p_n}(x)| \leq C
\qquad\text{for every }x\in\Omega.
\end{equation}
Moreover, by  \eqref{eq:2.13} and \eqref{eq:2.14}, for every $m,k>0$, one has
$$
\varepsilon_n^m=o(p_n^{-k})\quad \mathrm{as}\quad n\to+\infty.
$$}
\end{remark}

Let $\lambda_{i,n}$ and $\phi_{i,n}$ ($i\in\mathbb{N}$) denote the eigenvalues (counted with multiplicity) and eigenfunctions of the linearized problem at $u_{p_n}$:
\begin{equation}\label{tezhengzhi1}
\begin{cases}
-\Delta \phi = \lambda \left(\displaystyle p_n u_{p_n}^{p_n-1} \phi \int_{\Omega} \frac{u_{p_n}^{p_n+1}(y)}{|x - y|^\alpha}dy + (p_n+1) u_{p_n}^{p_n} \int_{\Omega} \frac{u_{p_n}^{p_n}(y) \phi(y)}{|x - y|^\alpha} dy \right) \quad &\text{in } \Omega, \\
\phi = 0 \quad &\text{on } \partial \Omega, \\
\| \phi \|_{L^\infty(\Omega)} = 1.
\end{cases}
\end{equation}

For every $i\ge1$, we use the inner rescaling
\[
\widetilde\phi_{i,n}(\xi):=\phi_{i,n}(x_n+\varepsilon_n\xi),
\qquad \xi\in\Omega_n:=\frac{\Omega-x_n}{\varepsilon_n}.
\]
We now state our first result, which describes the asymptotics of the first eigenpair.

\begin{theorem}\label{Th1}
Under the hypotheses of Theorem \ref{TA1}, choose the principal eigenfunction
$\phi_{1,n}$ to be positive.
Define
$$\varphi_n(\xi):=\frac{\phi_{1,n}(x_n+\varepsilon_n \xi)}{\phi_{1,n}(x_n)},\quad \xi\in\Omega_n:=\frac{\Omega-x_n}{\varepsilon_n}.$$
Then, as $n\to\infty$,
$$\lambda_{1,n}= \frac{1}{2p_n+1},\qquad
\phi_{1,n}=\frac{u_{p_n}}{\|u_{p_n}\|_{L^\infty(\Omega)}},\qquad
\varphi_n \to 1 \quad \text{in } C_{\text{loc}}^1(\mathbb{R}^2).$$
\end{theorem}

Next, we establish sharp asymptotics for the \emph{second and third eigenpairs}, which connect directly to the Hessian of the Robin function.

\begin{theorem}\label{th2}
Under the hypotheses of Theorem \ref{TA1}, the eigenvalues satisfy
\begin{equation}\label{jm1.8}
\lambda_{i,n} =
 1 + 3\pi C_\alpha^2\varepsilon_n^2\mu_{i-1}
 + o(\varepsilon_n^2),
\quad i=2,3.
\end{equation}
Moreover, 
\begin{equation*}
\widetilde\phi_{i,n}(\xi)=\sum_{j=1}^2a_j^i
\frac{\xi_j}{C_\alpha^2+|\xi|^2}+o(1)
\quad\text{in }C^1_{\mathrm{loc}}(\mathbb R^2),\quad i=2,3,
\end{equation*}
\begin{equation}\label{jm1.7}
\frac{\phi_{i,n}}{\varepsilon_n} =
2\pi\sum_{j=1}^2 a_j^i \frac{\partial G}{\partial y_j}\bigl(\,\cdot\,,x_\infty\bigr) + o(1)
\quad \text{in } C^1_{\text{loc}}(\overline{\Omega}\setminus\{x_\infty\}),\quad i=2,3,
\end{equation}
where $G$ denotes the Dirichlet Green function of $-\Delta$, and $\mu_1\leq\mu_2$ are the eigenvalues of
$D^2R(x_\infty)$. Moreover, $a^i=(a_1^i,a_2^i)\ne0$ is an eigenvector of $D^2R(x_\infty)$
for $\mu_{i-1}$, and $a^2\cdot a^3=0$.
\end{theorem}

We also characterize the \emph{fourth eigenpair}.

\begin{theorem}\label{th4}
Under the hypotheses of Theorem \ref{TA1}, the fourth eigenvalue satisfies
\begin{equation}\label{jm1.10}
\lambda_{4,n} = 1 + \frac{\frac{3}{2}(4-\alpha)}{p_n} + o\!\left(\frac{1}{p_n}\right).
\end{equation}
Moreover, 
for some $b>0$,
\begin{equation*}
\widetilde\phi_{4,n}(\xi)=b\frac{C_\alpha^2-|\xi|^2}{C_\alpha^2+|\xi|^2}+o(1)
\quad\text{in }C^1_{\mathrm{loc}}(\mathbb R^2),
\end{equation*}
\begin{equation}\label{jm1.9}
p_n \phi_{4,n} = -2(4-\alpha)\pi b \,G\bigl(\,\cdot\,,x_\infty\bigr) + o(1)
\quad \text{in } \, C^1_{\mathrm{loc}}(\overline{\Omega}\setminus\{x_\infty\}).
\end{equation}
\end{theorem}

Combining the spectral expansions with the variational characterization,
we obtain the relation with the Morse indices. Here $m(u_p)$ and $m_0(u_p)$
count, with multiplicity, the eigenvalues below one and those below or
equal to one, respectively; see Definition \ref{DA1}. For the critical
point $x_\infty$ of the Robin function, set
\[
m(x_\infty):=\#\{j\in\{1,2\}:\mu_j<0\},\qquad
m_0(x_\infty):=\#\{j\in\{1,2\}:\mu_j\leq0\}.
\]

\begin{theorem}\label{th3}
Under the hypotheses of Theorem \ref{TA1},
 $D^2R(x_\infty)$ is
positive semidefinite and, for all sufficiently large $n$,
\begin{equation}\label{jm1.6}
1+m(x_\infty)\leq m(u_{p_n})=1
\leq m_0(u_{p_n})\leq1+m_0(x_\infty)\leq2.
\end{equation}
If $x_\infty$ is a nondegenerate critical point of $R$, then it is a strict
local minimum, $u_{p_n}$ is nondegenerate for all sufficiently large $n$,
and
\[
m(u_{p_n})=m_0(u_{p_n})=1+m(x_\infty)=1.
\]
\end{theorem}
In \eqref{jm1.6}, $m(u_{p_n})$ takes the value either $1$ or $2$. Indeed, this follows directly from the estimate $m_0(x_\infty)\le 1$, which was verified using the properties of the Robin function in planar domains (see Lemma \ref{Rgdengshi}). For $N\ge 3$, the corresponding estimate is weaker; see \cite{Pan}.
\begin{remark}\normalfont
The restriction $0<\alpha<1$ ensures that the first derivatives of the Riesz kernel are locally integrable in two dimensions, as required for the integration by parts used below.
\end{remark}
\begin{remark}\normalfont\label{wen1.5}
The proof follows the variational and blow-up approach used for the local
problem, with the following additional ingredients for the nonlocal terms.
\begin{itemize}
\item[(1)]Unlike the local case, the Poho\v{z}aev identities involve double integrals with the Riesz kernel. Combining these identities with the boundary identities for the Green and Robin functions, we obtain sharp expansions for the second and third eigenvalues.
\item[(2)]The cutoff argument for the second and third eigenvalues introduces nonlocal error terms involving derivatives of the Riesz kernel. We estimate these terms by symmetrization and integration by parts, using differences of cutoff functions to control the kernel singularities.
\end{itemize}
\end{remark}

The paper is organized as follows.
Section \ref{se2} contains the preliminary variational and analytic estimates.
In Section \ref{se3}, we analyze the principal eigenpair and prove Theorem
\ref{Th1}. Section \ref{se4} is devoted to the second and third eigenpairs
and proves Theorem \ref{th2}. In Section \ref{se5}, we analyze the fourth
eigenpair, prove Theorem \ref{th4}, and derive Theorem \ref{th3}.

Throughout the paper, $C$ denotes a positive constant that may change from
line to line. We use the standard Landau symbols $O(\cdot)$ and $o(\cdot)$,
with the relevant limiting variable specified in each occurrence.

\section{Preliminary results}\label{se2}

\subsection{Variational framework and Morse index}

The energy functional associated with \eqref{maineq} is
\[
\mathcal E_p(v)=\frac12\int_\Omega|\nabla v|^2\,dx
-\frac{1}{2(p+1)}\int_\Omega\int_\Omega
\frac{|v(x)|^{p+1}|v(y)|^{p+1}}{|x-y|^\alpha}\,dxdy.
\]
Its Nehari manifold is
\[
\mathcal N_p=\left\{v\in H_0^1(\Omega)\setminus\{0\}:
\langle\mathcal E_p'(v),v\rangle=0\right\}.
\]
Throughout the paper, a least-energy solution means a minimizer of
$\mathcal E_p$ on $\mathcal N_p$.
This agrees with the homogeneous quotient characterization used in
\cite{Gao}. Indeed, if
\[
A(v)=\int_\Omega|\nabla v|^2,\qquad
D_p(v)=\int_\Omega\int_\Omega
\frac{|v(x)|^{p+1}|v(y)|^{p+1}}{|x-y|^\alpha}\,dxdy,
\]
then $t(v)=(A(v)/D_p(v))^{1/(2p)}$ is the unique positive number
such that $t(v)v\in\mathcal N_p$ and
\[
\mathcal E_p(t(v)v)
=\frac{p}{2(p+1)}
\left(\frac{A(v)}{D_p(v)^{1/(p+1)}}\right)^{(p+1)/p}.
\]
For $\phi,\psi\in H_0^1(\Omega)$, set
\begin{align*}
B_p(\phi,\psi):=p\int_\Omega u_p^{p-1}\phi\psi
\left(\int_\Omega \frac{u_p^{p+1}(y)}{|x-y|^\alpha}dy\right)dx+(p+1)\int_\Omega\int_\Omega
\frac{u_p^p(x)\phi(x)u_p^p(y)\psi(y)}{|x-y|^\alpha}dxdy
\end{align*}
and
\[
Q_p(\phi):=\int_\Omega |\nabla\phi|^2dx-B_p(\phi,\phi).
\]
The generalized eigenvalue problem is
\[
\int_\Omega \nabla\phi\cdot\nabla\psi\,dx=\lambda B_p(\phi,\psi),\qquad \psi\in H_0^1(\Omega).
\]
The form $B_p$ is symmetric and positive definite, and its compactness follows from the HLS inequality in Lemma \ref{A1} and the compact Sobolev embeddings. Thus, the eigenvalues satisfy $0<\lambda_{1,p}\leq\lambda_{2,p}\leq\cdots\to+\infty$, counted with multiplicity.

\begin{definition}\label{DA1}
The Morse index and augmented Morse index of a solution $u_p$ to \eqref{maineq} are defined by
\[
\begin{cases}
m(u_p) := \# \left\{ k \in \mathbb{N} : \lambda_{k,p} < 1 \right\}, \\
m_0(u_p) := \# \left\{ k \in \mathbb{N} : \lambda_{k,p} \leq 1 \right\},
\end{cases}
\]
where $\lambda_{1,p} \leq \lambda_{2,p} \leq \lambda_{3,p} \leq \cdots$ denote the eigenvalues, counted with multiplicity, of the linearized problem
\begin{equation}\label{tezhengzhi}
\begin{cases}
-\Delta \phi_{i,p} = \lambda_{i,p} \left( \displaystyle p u_p^{p-1} \phi_{i,p} \int_{\Omega} \frac{u_p^{p+1}(y)}{|x - y|^\alpha}dy + (p+1) u_p^{p} \int_{\Omega} \frac{u_p^{p}(y) \phi_{i,p}(y)}{|x - y|^\alpha} dy \right)\quad &\text{in } \Omega, \\
\phi_{i,p} = 0 \quad &\text{on } \partial \Omega, \\
\| \phi_{i,p} \|_{L^\infty(\Omega)} = 1,
\end{cases}
\end{equation}
Here $\phi_{i,p}$ is the $i$-th eigenfunction corresponding to the eigenvalue
$\lambda_{i,p}.$
\end{definition}

The solution $u_p$ is nondegenerate if and only if $1$ is not an eigenvalue of \eqref{tezhengzhi}, in which case $m(u_p)=m_0(u_p)$. In particular, for $i\ne j$,
\[
\int_\Omega\nabla\phi_{i,p}\cdot\nabla\phi_{j,p}\,dx
=B_p(\phi_{i,p},\phi_{j,p})=0.
\]
The Morse index and related
spectral information have applications to symmetry, uniqueness and
bifurcation; see, e.g., \cite{Bir,Dam,Gui,Guo,GuoZ,Pac}.
Computing the Morse index generally requires precise spectral information of the linearized operator, which is nontrivial;
see \cite{Gla1,Gla2,Luo1} and references therein.

\begin{proposition}\label{prop:least-energy-morse}
Every least-energy solution $u_p$ of \eqref{maineq} has Morse index one. In
particular,
\[
m(u_p)=1,\qquad \lambda_{1,p}<1\leq\lambda_{2,p}.
\]
\end{proposition}
\begin{proof}
Set
\[
\mathcal F_p(v)=\langle\mathcal E_p'(v),v\rangle.
\]
Since $\mathcal E_p'(u_p)=0$, differentiation gives
\[
\mathcal F_p'(u_p)[\phi]
=\mathcal E_p''(u_p)[u_p,\phi]
=-2p\int_\Omega u_p^p\phi
\left(\int_\Omega\frac{u_p^{p+1}(y)}{|x-y|^\alpha}\,dy\right)dx.
\]
In particular,
\[
\mathcal F_p'(u_p)[u_p]
=-2p\int_\Omega|\nabla u_p|^2\,dx<0.
\]
Thus $\mathcal N_p$ is a $C^1$ manifold near $u_p$, and
$T_{u_p}\mathcal N_p=\ker\mathcal F_p'(u_p)$ has codimension one in
$H_0^1(\Omega)$. Since $u_p$ minimizes $\mathcal E_p$ on $\mathcal N_p$,
\[
Q_p(\phi)=\mathcal E_p''(u_p)[\phi,\phi]\geq0
\qquad\text{for every }\phi\in T_{u_p}\mathcal N_p.
\]
Every two-dimensional subspace of $H_0^1(\Omega)$ intersects
$T_{u_p}\mathcal N_p$ nontrivially. Hence $Q_p$ cannot be negative definite
on a subspace of dimension two, and $m(u_p)\leq1$. On the other hand, the
equation for $u_p$ gives
\[
Q_p(u_p)=-2p\int_\Omega|\nabla u_p|^2\,dx<0.
\]
Therefore $m(u_p)=1$. Since the first eigenvalue is below one, this is
equivalent to $\lambda_{1,p}<1\leq\lambda_{2,p}$.
\end{proof}

\subsection{Green and Robin functions}

Let $G(x, y)$ be the Green function of $-\Delta$ in $\Omega$ with Dirichlet boundary conditions, which satisfies
\begin{equation*}
\begin{cases}
-\Delta G(\cdot, y) = \delta_y & \text{in } \Omega, \\
G(\cdot, y) = 0 & \text{on } \partial\Omega,
\end{cases}
\end{equation*}
where $\delta_y$ denotes the Dirac delta function centered at $y$. The function
 $H(x, y)$ is the regular part of the Green function, i.e.,
\begin{equation*}
H(x, y) = -\frac{1}{2\pi} \log|x - y| - G(x, y).
\end{equation*}
We denote by $R(x): = H(x, x)$ the Robin function of $\Omega$.
\begin{lemma}\label{Rdayu}\cite[Lemma 2.1]{De}
If $\Omega\subset\mathbb R^2$ is a bounded domain, then its Robin function
satisfies
\[
\Delta R > 0 \quad \text{in } \Omega.
\]
\end{lemma}
\begin{lemma}\label{Rgdengshi}\cite[Lemma 2.3]{De}
For any $y \in \Omega$,
\begin{align}
\int_{\partial\Omega} (x-y)\cdot \nu(x) \left( \frac{\partial G}{\partial\nu}(x,y) \right)^2
d\sigma_x &= \frac{1}{2\pi}, \label{eq:2.3} \\
\int_{\partial\Omega} \nu_j(x) \left( \frac{\partial G}{\partial\nu}(x,y) \right)^2
d\sigma_x &= \frac{\partial R}{\partial y_j}(y), \label{eq:2.5} \\
\int_{\partial\Omega} \frac{\partial G}{\partial x_j}(x,y) \frac{\partial^2 G}{\partial y_k
\partial\nu}(x,y) d\sigma_x &= \frac{1}{2} \frac{\partial^2 R}
{\partial y_j \partial y_k}(y), \label{eq:2.6}
\end{align}
 where $\nu(x)$ denotes the unit outer normal of $\partial\Omega$ at $x$.
\end{lemma}
\subsection{Asymptotic behavior of least-energy solutions}
Under the assumptions and notation of Theorem \ref{TA1}, we establish the estimates needed in the subsequent analysis.

\begin{lemma}\label{lem:offpeak}
Let $K\subset\overline\Omega\setminus\{x_\infty\}$ be compact. Then
\begin{equation}\label{eq223}
 p_n\left\|u_{p_n}^{p_n-1}(x)
 \int_\Omega\frac{u_{p_n}^{p_n+1}(y)}{|x-y|^\alpha}\,dy
 \right\|_{L^\infty(K)}\longrightarrow0
\end{equation}
and
\begin{equation}\label{eq224}
 p_n\left\|u_{p_n}^{p_n}(x)
 \int_\Omega\frac{u_{p_n}^{p_n}(y)}{|x-y|^\alpha}\,dy
 \right\|_{L^\infty(K)}\longrightarrow0.
\end{equation}
More generally, let $a,b\in\mathbb R$ and $m\geq0$ be fixed. If
$0\leq\gamma_n\leq(p_n+1)/2$, then
\begin{equation}\label{eq225}
\varepsilon_n^{-m}p_n^{\gamma_n}
\int_K\int_\Omega
\frac{u_{p_n}^{p_n+a}(x)u_{p_n}^{p_n+b}(y)}
{|x-y|^\alpha}\,dy\,dx\longrightarrow0.
\end{equation}
Moreover,
\[
\varepsilon_n^{-m}p_n^{\gamma_n}
\left\|u_{p_n}^{p_n+a}(x)
\int_\Omega\frac{u_{p_n}^{p_n+b}(y)}{|x-y|^\alpha}\,dy
\right\|_{L^\infty(K)}\longrightarrow0.
\]
\end{lemma}
\begin{proof}
Choose a compact set
$K_1\subset\overline\Omega\setminus\{x_\infty\}$ whose relative interior
contains $K$, and put
\[
m_n=\|u_{p_n}\|_{L^\infty(K_1)},\qquad
M_n=\|u_{p_n}\|_{L^\infty(\Omega)}.
\]
By \eqref{eq:2.7}, $\delta_n:=p_n^{1/2}m_n\to0$, whereas
$M_n\to\sqrt e$. Since $\alpha<2$,
\[
\sup_{x\in\Omega}\int_\Omega |x-y|^{-\alpha}\,dy\leq C.
\]
Consequently, for $x\in K$,
\[
u_{p_n}^{p_n+a}(x)
\int_\Omega\frac{u_{p_n}^{p_n+b}(y)}{|x-y|^\alpha}\,dy
\leq C m_n^{p_n+a}M_n^{p_n+b}.
\]
The same bound, up to a fixed factor, holds for the double integral in
\eqref{eq225}. Recalling that
\[
\varepsilon_n^{-m}
=p_n^{\frac{m}{4-\alpha}}
M_n^{\frac{2mp_n}{4-\alpha}},
\]
and using $\gamma_n\leq(p_n+1)/2$, we obtain
\[
\varepsilon_n^{-m}p_n^{\gamma_n}
 m_n^{p_n+a}M_n^{p_n+b}
\leq C p_n^{C_{a,m}}(C_m\delta_n)^{p_n+a}.
\]
The right-hand side tends to zero. This proves the two general assertions.
Taking $(a,b)=(-1,1)$ and $(a,b)=(0,0)$, with $m=0$ and
$\gamma_n=1$, gives \eqref{eq223} and \eqref{eq224}.
\end{proof}

Next, we establish pointwise decay estimates for $v_n$. They yield the strong
Riesz-form convergence in Lemma \ref{RE2.9} below.

\begin{lemma}\label{JMPA2.8}
For every $\gamma\in(0,4-\alpha)$, there exist $R_\gamma>0$,
$\widetilde C_\gamma>0$, $C_\gamma>0$, and $n_\gamma\in\mathbb N$ such that,
for every $n\geq n_\gamma$,
\begin{equation*}
v_{n}(z) \leq \left(4-\alpha - \frac{\gamma}{2}\right) \log \frac{1}{|z|} +
 \widetilde C_\gamma \quad \mathrm{for}  \ R_\gamma \leq |z| \leq \frac{r}{\varepsilon_{n}}
\end{equation*}
and
\begin{equation*}
0 \leq \left(1 + \frac{v_{n}(z)}{p_n}\right)^{p_n - 1} \leq \frac{C_\gamma}{1 + |z|^{4-\alpha-\gamma}} \quad \mathrm{for } \  |z| \leq \frac{r}{\varepsilon_{n}}.
\end{equation*}
Moreover,
\begin{equation*}
0 \leq \left(1 + \frac{v_{n}(z)}{p_n}\right)^{p_n } \leq \frac{C_\gamma}{1 + |z|^{4-\alpha-\gamma}} \quad \mathrm{for } \  |z| \leq \frac{r}{\varepsilon_{n}}.
\end{equation*}
\end{lemma}
\begin{proof}
Fix $\gamma\in(0,4-\alpha)$. By \eqref{MA4.141},
$\beta_n/(2\pi)\ge4-\alpha-\gamma/4$ for all large $n$.
Apply Lemma \ref{lemmaA4} with $\gamma/4$. Since
$\log(1/|z|)<0$ for $|z|>1$, we obtain
\[
v_n(z)\le \left(4-\alpha-\frac\gamma2\right)\log\frac1{|z|}+C_\gamma
\quad\text{for }R_\gamma\le|z|\le r/\varepsilon_n.
\]
Since $0<1+v_n/p_n\le1$ and $\log(1+t)\le t$ for $t>-1$,
\[
\left(1+\frac{v_n(z)}{p_n}\right)^{p_n}
\le e^{v_n(z)}\le C_\gamma |z|^{-(4-\alpha-\gamma/2)}.
\]
The same argument with the factor $(p_n-1)/p_n$ gives, for all large $n$,
\[
\left(1+\frac{v_n(z)}{p_n}\right)^{p_n-1}
\le C_\gamma |z|^{-(4-\alpha-\gamma)}.
\]
After increasing $C_\gamma$ to cover $|z|\le R_\gamma$, the two asserted bounds follow.
\end{proof}
\begin{lemma}\label{lem:log-growth}
There exists $C>0$ independent of $n$ such that
\[
|\nabla v_n(\xi)|\leq\frac{C}{1+|\xi|},\qquad
|v_n(\xi)|\leq C\bigl(1+\log(2+|\xi|)\bigr)
\quad\text{for every }\xi\in\Omega_n.
\]
\end{lemma}
\begin{proof}
The gradient bound follows from \eqref{eq:global-gradient},
$u_{p_n}(x_n)\to\sqrt e$ and the local $C^2$ convergence in
\eqref{eq:2.16}. For $1\leq|\xi|\leq r/\varepsilon_n$, integrate
this bound along the radial segment from $\xi/|\xi|$ to $\xi$,
which lies in $B_{r/\varepsilon_n}(0)\subset\Omega_n$.
The local bound on $B_1(0)$ then gives the asserted logarithmic bound
on this ball. For $|\xi|\geq r/\varepsilon_n$, use
$-p_n\leq v_n\leq0$ and
\[
\frac{\log(1/\varepsilon_n)}{p_n}
=\frac{\log p_n}{(4-\alpha)p_n}
+\frac{2\log u_{p_n}(x_n)}{4-\alpha}
\longrightarrow\frac{1}{4-\alpha}>0.
\]
These imply $p_n\leq C\log(2+|\xi|)$ in the remaining region.
\end{proof}
\begin{lemma}\label{lem:shifted-potentials}
Put
\[
w_n(\xi)=1+\frac{v_n(\xi)}{p_n}
\]
in $\Omega_n$, and extend $w_n$ by zero outside $\Omega_n$. Fix
$\gamma\in(0,2-\alpha)$. For every $s\in\{-1,0,1\}$, there exists
$C_\gamma>0$ such that
\[
0\leq w_n^{p_n+s}(\xi)
\leq\frac{C_\gamma}{1+|\xi|^{4-\alpha-\gamma}}
\quad\text{in }\mathbb R^2
\]
and
\[
\int_{\mathbb R^2}
\frac{w_n^{p_n+s}(\eta)}{|\xi-\eta|^\alpha}\,d\eta
\leq\frac{C_\gamma}{1+|\xi|^\alpha}
\quad\text{for every }\xi\in\mathbb R^2.
\]
\end{lemma}
\begin{proof}
Set $
\beta:=4-\alpha-\gamma>2.$
By Lemma~\ref{JMPA2.8}, for every $s\in\{-1,0,1\}$,
\[
0\leq w_n^{p_n+s}(\xi)
\leq \frac{C_\gamma}{1+|\xi|^\beta},
\qquad
|\xi|\leq \frac{r}{\varepsilon_n}.
\]
We next consider $|\xi|\geq r/\varepsilon_n$. By \eqref{eq:global-gradient},
the boundary condition $u_{p_n}=0$ on $\partial\Omega$, and
$B_{4r}(x_\infty)\subset\subset\Omega$, we obtain
\[
\sup_{\Omega\setminus B_r(x_n)}u_{p_n}
\leq \frac{C}{p_n}.
\]
Since $u_{p_n}(x_n)\to\sqrt e$, it follows that, for every fixed
$N>0$,
\[
\sup_{\Omega_n\setminus B_{r/\varepsilon_n}(0)}
w_n^{p_n-1}
\leq
\left(\frac{C}{p_n}\right)^{p_n-1}
=o(\varepsilon_n^N).
\]
Using $\operatorname{diam}(\Omega_n)=O(\varepsilon_n^{-1})$ and
choosing $N>\beta$, we conclude that
\[
0\leq w_n^{p_n+s}(\xi)
\leq \frac{C_\gamma}{1+|\xi|^\beta}
\qquad\text{in }\mathbb R^2.
\]
It remains to estimate the Riesz potential. Since $\beta>2$, the
first estimate gives uniform $L^1(\mathbb R^2)$ and
$L^\infty(\mathbb R^2)$ bounds for $w_n^{p_n+s}$. Hence the
potential is uniformly bounded for $|\xi|\leq2$. For $|\xi|>2$,
splitting the integral into
\[
\{|\eta|\leq |\xi|/2\},\qquad
\{|\eta|>|\xi|/2,\qquad |\xi-\eta|\leq|\xi|/2\},
\]
and the remaining region, we obtain
\[
\int_{\mathbb R^2}
\frac{w_n^{p_n+s}(\eta)}{|\xi-\eta|^\alpha}\,d\eta
\leq
C|\xi|^{-\alpha}
+C|\xi|^{2-\alpha-\beta}.
\]
Since $2-\alpha-\beta=-2+\gamma<-\alpha$,
 the desired estimate follows.
\end{proof}
\begin{lemma}\label{RE2.9}
Let $D_n\subset\mathbb R^2$ satisfy $\boldsymbol{1}_{D_n}\to1$ almost everywhere.
Suppose that $f_n\to f$ and $g_n\to g$ pointwise and that
$|f_n(\xi)|+|g_n(\xi)|\le C(1+|\xi|)$ uniformly in $n$.
Then
\begin{align*}
\lim_{n \to +\infty} \displaystyle\int_{D_n \cap B_{\frac{r}{\varepsilon_{n}}}(0)}\left(\displaystyle\int_{\Omega_n}
\frac{\left(1+\frac{v_n(\eta)}{p_n}\right)^{p_n+1}}
{|\xi-\eta|^{\alpha}}d\eta \right)\left(1+\frac{v_n(\xi)}{p_n}\right)^{p_n-1}g_n(\xi)d\xi
\\= \int_{\mathbb{R}^2} \left( \int_{\mathbb{R}^2} \frac{e^{U(\eta)}}{|\xi- \eta|^\alpha} d\eta \right) e^{U(\xi)} g(\xi) d\xi
\end{align*}
and
\begin{align*}
\lim_{n \to +\infty} \displaystyle\int_{D_n \cap B_{\frac{r}{\varepsilon_{n}}}(0)}\left(\displaystyle\int_{\Omega_n}
\frac{\left(1+\frac{v_n(\eta)}{p_n}\right)^{p_n}g_n(\eta)}
{|\xi-\eta|^{\alpha}}d\eta \right)\left(1+\frac{v_n(\xi)}{p_n}\right)^{p_n}f_n(\xi)d\xi
\\= \int_{\mathbb{R}^2} \left( \int_{\mathbb{R}^2} \frac{e^{U(\eta)}g(\eta)}{|\xi- \eta|^\alpha} d\eta \right) e^{U(\xi)}f(\xi)  d\xi.
\end{align*}
More precisely, for $s\in\{-1,0,1\}$,
\begin{equation}\label{eq:strong-weighted-HLS}
\boldsymbol1_{\Omega_n}w_n^{p_n+s}f_n\longrightarrow e^Uf
\quad\text{strongly in }L^{4/(4-\alpha)}(\mathbb R^2),
\end{equation}
and similarly for $g_n$.
\end{lemma}
\begin{proof}
Put $q=4/(4-\alpha)$ and extend all functions on $\Omega_n$ by zero.
Choose $\gamma>0$ so small that
$\gamma<(2-\alpha)/2$. Lemma \ref{lem:shifted-potentials} gives, for
$s\in\{-1,0,1\}$ and every $\xi\in\mathbb R^2$,
\[
w_n^{p_n+s}(|f_n|+|g_n|)
\le \frac{C(1+|\xi|)}{1+|\xi|^{4-\alpha-\gamma}}.
\]
The right-hand side belongs to $L^q(\mathbb R^2)$, since
$q(3-\alpha-\gamma)>2$. The local convergence $v_n\to U$
implies $w_n^{p_n+s}\to e^U$ pointwise. Dominated convergence
therefore proves \eqref{eq:strong-weighted-HLS}; the same argument
applies to $g_n$ and to either outer cutoff, whose indicator converges
pointwise almost everywhere to one.
The bilinear HLS estimate
\[
\left|\int_{\mathbb R^2}\int_{\mathbb R^2}
\frac{F(\xi)G(\eta)}{|\xi-\eta|^\alpha}\,d\eta d\xi\right|
\le C\|F\|_{L^q}\|G\|_{L^q}
\]
now gives both limits.
\end{proof}

For this rescaling of $u_{p_n}$, we denote
\begin{equation}\label{bianhuan1}
\tilde{u}_{p_n}(\xi) = u_{p_n}(x_n + \varepsilon_n \xi), \quad \text{for } \xi \in \Omega_n.
\end{equation}
By definition,
\begin{equation}\label{J2.34}
\| \tilde{u}_{p_n} \|_{L^\infty(\Omega_n)} = \| u_{p_n} \|_{L^\infty(\Omega)}
 \overset{\eqref{eq:2.91}}{\leq} {C}.
\end{equation}
Moreover, by \eqref{eq:2.13}, \eqref{eq:2.15}, and \eqref{eq:2.16},
\begin{equation*}
\tilde{u}_{p_n} = \| u_{p_n} \|_{L^\infty(\Omega)} \left(1 + \frac{v_n}{p_n}\right) \to \sqrt{e}
\quad \text{in } C_{\text{loc}}^0(\mathbb{R}^2) \, \mathrm{ as } \ n \to +\infty,
\end{equation*}
\begin{equation}\label{J2.36}
p_n \nabla \tilde{u}_{p_n} = \| u_{p_n} \|_{L^\infty(\Omega)} \nabla v_n \to \sqrt{e} \nabla U
\quad \text{in } C_{\text{loc}}^0(\mathbb{R}^2) \, \mathrm{ as } \ n \to +\infty.
\end{equation}
In addition, the gradient estimate becomes
\begin{equation}\label{J2.37}
p_n |\xi| |\nabla \tilde{u}_{p_n}(\xi)| \leq C \quad \text{for all }
\xi \in \Omega_n \text{ and all sufficiently large } n.
\end{equation}
\subsection{Asymptotic properties of the linearized problem}
For the sequence $(u_{p_n})$ considered above, we have the following result.
\begin{lemma}\label{lem:kernel-differentiation}
Let $\alpha\in(0,1)$ and set
\[
A_p(x)=\int_\Omega\frac{u_p^{p+1}(y)}{|x-y|^\alpha}\,dy.
\]
Then $A_p\in C^1(\overline\Omega)$ and
\[
\frac{\partial A_p}{\partial x_j}(x)
=(p+1)\int_\Omega
\frac{u_p^p(y)\frac{\partial u_p}{\partial y_j}(y)}{|x-y|^\alpha}\,dy.
\]
Moreover, for every $z\in\mathbb R^2$,
\[
(x-z)\cdot\nabla A_p(x)
=(2-\alpha)A_p(x)
+(p+1)\int_\Omega
\frac{u_p^p(y)(y-z)\cdot\nabla u_p(y)}{|x-y|^\alpha}\,dy.
\]
\end{lemma}
\begin{proof}
Since $\alpha<1$, the kernel and its first derivatives are locally integrable, so integration by parts in $y$ is justified. The boundary terms vanish since $u_p=0$ on $\partial\Omega$. The first identity follows from
\[
\partial_{x_j}|x-y|^{-\alpha}
=-\partial_{y_j}|x-y|^{-\alpha}.
\]
For the second, write $x-z=(x-y)+(y-z)$ and integrate by parts, using
\[
(x-y)\cdot\nabla_x|x-y|^{-\alpha}
=-\alpha|x-y|^{-\alpha}.
\]
Continuity up to the boundary follows from dominated convergence and the regularity of $u_p$.
\end{proof}
\begin{lemma}\label{poid}
Let $i\ge2$, $z\in\mathbb R^2$ and
$\omega_{p_n}(x)=(x-z)\cdot\nabla u_{p_n}(x)$. Then
\begin{align}\label{wen2.6}
\int_{\partial\Omega} \frac{\partial u_{p_n}}{\partial \nu} \frac{\partial \phi_{i,n}}{\partial \nu} (x-z) \cdot \nu \, d\sigma_x
= (1-\lambda_{i,n}) \Bigg\{ &p_n \int_\Omega u_{p_n}^{p_n-1} \phi_{i,n} \omega_{p_n} \left( \int_\Omega \frac{u_{p_n}^{{p_n}+1}(y)}{|x-y|^\alpha} dy \right) dx \nonumber \\
& + (p_n+1) \int_\Omega\int_\Omega \frac{u_{p_n}^{p_n}(x) \phi_{i,n}(x) u_{p_n}^{p_n}(y) \omega_{p_n}(y)}{|x-y|^\alpha} dx dy \Bigg\},
\end{align}
\begin{align}\label{wen48}
\int_{\partial\Omega} \frac{\partial \phi_{i,n}}{\partial \nu_x} \frac{\partial u_{p_n}}{\partial x_j} d\sigma_x
= (1-\lambda_{i,n}) \Bigg\{&
{p_n} \int_\Omega \phi_{i,n} u_{p_n}^{{p_n}-1} \frac{\partial u_{p_n}}{\partial x_j} \left( \int_\Omega \frac{u_{p_n}^{{p_n}+1}(y)}{|x-y|^\alpha} dy \right) dx \nonumber \\
& + ({p_n}+1) \int_\Omega u_{p_n}^{{p_n}} \frac{\partial u_{p_n}}{\partial x_j} \left( \int_\Omega \frac{u_{p_n}^{{p_n}}(y) \phi_{i,n}(y)}{|x-y|^\alpha} dy \right) dx \Bigg\}.
\end{align}
\end{lemma}
\begin{proof}
We suppress the index $n$ and write $p=p_n$, $u=u_{p_n}$, $\phi=\phi_{i,n}$
and $\lambda=\lambda_{i,n}$. Set
\[
A(x)=\int_\Omega\frac{u^{p+1}(y)}{|x-y|^\alpha}\,dy,
\qquad
\mathcal M\psi=pAu^{p-1}\psi+(p+1)u^p
\int_\Omega\frac{u^p(y)\psi(y)}{|x-y|^\alpha}\,dy.
\]
The kernel is symmetric. Hence
\[
\int_\Omega\varphi\mathcal M\psi= B_p(\varphi,\psi)
=B_p(\psi,\varphi)=\int_\Omega\psi\mathcal M\varphi.
\]
Testing the equation for $u$ with $\phi$, and the eigenvalue equation
with $u$, gives
\[
\int_\Omega Au^p\phi
=\int_\Omega\nabla u\cdot\nabla\phi
=\lambda B_p(u,\phi)
=\lambda(2p+1)\int_\Omega Au^p\phi.
\]
For $i\geq2$, Proposition \ref{prop:least-energy-morse} gives
$\lambda\geq1$, so $\lambda(2p+1)\ne1$. Therefore
\begin{equation}\label{eq:sourceorth}
\int_\Omega Au^p\phi=0.
\end{equation}
By Lemma \ref{lem:kernel-differentiation}, differentiation with respect to
$x_j$ and integration by parts in the convolution variable give
\[
-\Delta u_{x_j}=\mathcal M u_{x_j}.
\]
Green's identity, $\phi=0$ on $\partial\Omega$, and
$u_{x_j}=\nu_j\partial_\nu u$ on $\partial\Omega$ yield
\[
\int_{\partial\Omega}u_{x_j}\partial_\nu\phi
=\int_\Omega\{\phi(-\Delta u_{x_j})-u_{x_j}(-\Delta\phi)\}
=(1-\lambda)B_p(\phi,u_{x_j}).
\]
Expanding the last bilinear form proves \eqref{wen48}.
For the dilation field $\omega=(x-z)\cdot\nabla u$, homogeneity of the Riesz kernel gives
\[
(x-z)\cdot\nabla A(x)=(2-\alpha)A(x)
+(p+1)\int_\Omega\frac{u^p(y)(y-z)\cdot\nabla u(y)}{|x-y|^\alpha}\,dy.
\]
Consequently,
\[
-\Delta\omega=\mathcal M\omega+(4-\alpha)Au^p.
\]
Another application of Green's identity gives
\begin{align*}
\int_{\partial\Omega}(x-z)\cdot\nu\,\partial_\nu u\,\partial_\nu\phi
=\int_\Omega\{\phi(-\Delta\omega)-\omega(-\Delta\phi)\}=(1-\lambda)B_p(\phi,\omega)
+(4-\alpha)\int_\Omega Au^p\phi.
\end{align*}
The last integral vanishes by \eqref{eq:sourceorth}. Expanding $B_p(\phi,\omega)$
proves \eqref{wen2.6}.
\end{proof}
We use the rescaling
\begin{equation}\label{JM2.42}
\widetilde{\phi}_{i,n}(x):=\phi_{i,n}(x_n+\varepsilon_nx),
\qquad x\in\Omega_n.
\end{equation}
Then $(\lambda_{i,n},\widetilde{\phi}_{i,n})$ solves the following rescaled
eigenvalue problem:
\begin{equation}\label{JMPA2.43}
\left\{\begin{array}{ll}
-\Delta \phi = \lambda\Bigg\{W_n\left(1+\frac{v_n}{p_n}\right) ^{p_n-1}\phi +
(1+\frac{1}{p_n})\left(\displaystyle\int_{\Omega_n}\frac{\left(1+\frac{v_n(y)}{p_n}\right) ^{p_n}\phi(y)}
{|x-y|^{\alpha}} dy \right) \left(1+\frac{v_n}{p_n}\right) ^{p_n}\Bigg\} & \mbox{in} \  \Omega_n, \\[0.05cm]
\phi = 0 & \mbox{on} \ \partial \Omega_n, \\[0.05cm]
\| \phi \|_{L^\infty(\Omega_n)} = 1,
\end{array}
\right.
\end{equation}
where
\[
W_n(x):=\int_{\Omega_n}
\frac{\left(1+\frac{v_n(y)}{p_n}\right)^{p_n+1}}{|x-y|^\alpha}\,dy.
\]
Lemma \ref{fl2} shows that $W_n$ is uniformly bounded. The formal limiting
problem is therefore
\begin{equation*}
-\Delta \widetilde\phi=\lambda e^{U}\left[\left( \displaystyle\int_{\mathbb{R}^2} \frac{e^{U(y)}}{|x - y|^\alpha} dy \right)\widetilde\phi+\displaystyle\int_{\mathbb{R}^2} \frac{e^{U(y)}\widetilde\phi(y)}{|x - y|^\alpha} dy \right]   \ \mathrm{in} \  {\mathbb{R}^2}.
\end{equation*}

In the remainder of this section, we will prove some key intermediate asymptotic results regarding
 eigenvalues and eigenfunctions.
\begin{lemma}\label{lem:2.11}
Let $\widetilde{\phi}_{i,n}$ be defined by \eqref{JM2.42}. Suppose that
$\widetilde{\phi}_{i,n}\to\widetilde{\phi}$ in
$C_{\mathrm{loc}}^0(\mathbb R^2)$ and that
$\lambda_{i,n}\to\Lambda\in[0,+\infty)$. Then
$\widetilde{\phi}\not\equiv0$.
\end{lemma}
\begin{proof}
We first consider the eigenfunction $\phi_{i,n}$, which solves \eqref{tezhengzhi1} with $\lambda = \lambda_{i,n}$.
From \eqref{eq223} and \eqref{eq224}, we get
$$\lambda_{i,n}\Bigg\{p_n u_{p_n}^{{p_n}-1} \phi_{i,n} \left( \displaystyle\int_{\Omega} \frac{u_{p_n}^{{p_n}+1}(y)}{|x - y|^\alpha} dy \right)+({p_n}+1) u_{p_n}^{{p_n}} \left( \displaystyle \int_{\Omega} \frac{u_{p_n}^{{p_n}}(y) \phi_{i,n}(y)}{|x - y|^\alpha} dy \right) \Bigg\}$$
converges to $0$ locally uniformly in $\Omega \setminus \{x_\infty\}$ as $n \to +\infty$.
Local boundary estimates show that every subsequence has a further subsequence converging in
$C^1_{\mathrm{loc}}(\overline\Omega\setminus\{x_\infty\})$ to a harmonic function $\phi_0$.
Since $|\phi_0|\le1$, its isolated singularity is removable. The extension is harmonic in
$\Omega$ and vanishes on $\partial\Omega$, so the maximum principle gives $\phi_0\equiv0$.
Therefore the whole sequence satisfies
\begin{equation}\label{jm2.45}
\phi_{i,n} \to 0 \quad \text{locally uniformly in } \overline{\Omega} \setminus \{x_\infty\} \text{ as } n \to +\infty.
\end{equation}
We now consider $\widetilde{\phi}_{i,n}$. After changing its sign if
necessary, choose $s_n\in\Omega_n$ such that
\begin{equation}\label{jm2.46}
\widetilde{\phi}_{i,n}(s_n) = 1.
\end{equation}
Observe that $B_{\frac{r}{\varepsilon_n}}(0) \subset \Omega_n$ by the
choice of $r$ fixed after Theorem \ref{TA1}, and
\begin{equation}\label{jm2.47}
|s_n| < \frac{r}{\varepsilon_n} \quad \text{for } n \text{ large},
\end{equation}
Indeed, \eqref{JM2.42}, \eqref{jm2.46}, \eqref{eq:2.12} and
\eqref{jm2.45} imply that
$x_n+\varepsilon_ns_n\in B_r(x_n)$ for all large $n$, which proves
\eqref{jm2.47}.
Assume to the contrary that $\widetilde{\phi} \equiv 0$, i.e.,
\begin{equation}\label{jm2.48}
\widetilde{\phi}_{i,n} \to 0 \quad \text{locally uniformly in } \mathbb{R}^2 \text{ as } n \to +\infty.
\end{equation}
It then follows that $|s_n|\to+\infty$, and as a consequence,
\begin{equation}\label{jm2.49}
|s_n| > 1 \quad   \ \mathrm{for} \ n \text{ large}.
\end{equation}
Let $z_n$ be the Kelvin transform of $\widetilde{\phi}_{i,n}$, namely
\begin{equation*}
z_n(x) := \widetilde{\phi}_{i,n}\left(\frac{x}{|x|^2}\right).
\end{equation*}
Note that $z_n$ is well defined in $\mathbb{R}^2 \setminus B_{\frac{\varepsilon_n}{r}}(0)$ (as $B_{\frac{r}{\varepsilon_n}}(0) \subset \Omega_n$ and $\widetilde{\phi}_{i,n}$ is defined in $\Omega_n$), and by \eqref{JMPA2.43}, it satisfies
\begin{equation}\label{jm2.50}
-\Delta z_n = \left[ \frac{\lambda_{i,n}}{|x|^4} a_n(z) \widetilde{\phi}_{i,n}(z) + \frac{\lambda_{i,n}}{|x|^4} \left(1+\frac{1}{p_n}\right)\left(1+\frac{v_n(z)}{p_n}\right)^{p_n} (\mathcal{K}_n \widetilde{\phi}_{i,n})(z) \right]_{z=\frac{x}{|x|^2}} \ \mathrm{in}  \ \mathbb{R}^2 \setminus B_{\frac{\varepsilon_n}{r}}(0),
\end{equation}
where $a_n(x)=\left(1+\frac{v_n(x)}{p_n}\right)^{p_n-1} \displaystyle\!\!\int_{\Omega_n} \!\!\frac{\left(1+\frac{v_n(y)}{ p_n}\right)^{p_n+1}}{|x-y|^{\alpha}}dy$ and $(\mathcal{K}_n\widetilde{\phi}_{i,n})(x)=\displaystyle\int_{\Omega_n} \!\!\frac{\left(1+\frac{v_n(y)}{p_n}\right)^{p_n} \widetilde{\phi}_{i,n}(y)}{|x-y|^{\alpha}}  dy.$
Moreover, by \eqref{jm2.48}, we have
\begin{equation*}
z_n(x) \to 0 \text{ as } n \to +\infty, \text{ pointwise for any } x \neq 0.
\end{equation*}
Denote
\begin{equation*}
g_n(x) :=
\begin{cases}
 \frac{\lambda_{i,n}}{|x|^4} a_n(z) \widetilde{\phi}_{i,n}(z) + \frac{\lambda_{i,n}}{|x|^4}
 \left(1+\frac{1}{p_n}\right)\left(1+\frac{v_n(z)}{p_n}\right)^{p_n} (\mathcal{K}_n
 \widetilde{\phi}_{i,n})(z)
& \text{for}\, x \in B_1(0) \setminus \overline{B_{\frac{\varepsilon_n}{r}}(0)},\\
0
& \text{for} \,x \in B_{\frac{\varepsilon_n}{r}}(0),
\end{cases}
\end{equation*}
where ${z=\frac{x}{|x|^2}}$ and $g_n(x)=g_{1,n}(x)+g_{2,n}(x)$
with $g_{1,n}(x)=\frac{\lambda_{i,n}}{|x|^4} a_n(z) \widetilde{\phi}_{i,n}(z)$
and
$$g_{2,n}(x)=\frac{\lambda_{i,n}}{|x|^4} \left(1+\frac{1}{p_n}\right)\left(1+\frac{v_n(z)}{p_n}\right)^{p_n} (\mathcal{K}_n \widetilde{\phi}_{i,n})(z).
$$
Choose $\gamma>0$ so small that $\alpha+\gamma<1$.
Lemma \ref{lem:shifted-potentials} gives
\[
a_n(z)\le \frac{C}{1+|z|^{4-\alpha-\gamma}},
\qquad
\left(1+\frac{v_n(z)}{p_n}\right)^{p_n}
\le \frac{C}{1+|z|^{4-\alpha-\gamma}}.
\]
Consequently,
\[
|g_{1,n}(x)|\le C|x|^{-\alpha-\gamma}\quad\text{in }B_1(0).
\]
For every fixed $x\ne0$, \eqref{jm2.48} implies $g_{1,n}(x)\to0$.
Since $|x|^{-\alpha-\gamma}\in L^2(B_1)$, dominated convergence yields
\begin{equation}\label{g1}
\|g_{1,n}\|_{L^2(B_1)}\to0.
\end{equation}
For $g_{2,n}$, write
\begin{equation*}
g_{2,n}(x) = |x|^{-4} \lambda_{i,n} \left(1+\frac{1}{p_n}\right) \underbrace{\left(1+\frac{v_n(z)}{p_n}\right)^{p_n}}_{:= A_n(z)}
\underbrace{\int_{\Omega_n} \frac{\left(1+\frac{v_n(y)}{p_n}\right)^{p_n} \widetilde{\phi}_{i,n}(y)}{|z-y|^\alpha} dy}_{:= B_n(z)}.
\end{equation*}
The preceding decay estimate implies
\[
\sup_n\int_{\Omega_n}\left(1+\frac{v_n(y)}{p_n}\right)^{p_n}dy<\infty.
\]
Splitting the integral into $|y-z|<1$ and $|y-z|\ge1$ therefore gives
\begin{equation*}
\sup_{n,z}|B_n(z)|<\infty.
\end{equation*}
For fixed $z$, split further into $B_R(z)$ and its complement.
On $B_R(z)$, \eqref{jm2.48} and local integrability of $|z-y|^{-\alpha}$ give convergence to zero.
On the complement, the uniform $L^1$ bound gives a bound $CR^{-\alpha}$.
Letting first $n\to\infty$ and then $R\to\infty$, we obtain $B_n(z)\to0$.
Thus $g_{2,n}(x)\to0$ for every fixed $x\ne0$, while
\[
|g_{2,n}(x)|\le C|x|^{-\alpha-\gamma}\quad\text{in }B_1(0).
\]
Dominated convergence gives
\begin{equation}\label{g2}
\|g_{2,n}\|_{L^2(B_1)}\to0.
\end{equation}
As a consequence, by \eqref{g1} and \eqref{g2}, we have
\begin{equation}\label{g}
\lim_{n\to\infty} \|g_n\|_{L^2(B_1(0))} \leq \lim_{n\to\infty} \|g_{1,n}\|_{L^2(B_1(0))} + \lim_{n\to\infty} \|g_{2,n}\|_{L^2(B_1(0))} = 0.
\end{equation}
Next, let $f_n \in H_0^1(B_1(0))$ such that
\begin{equation}\label{fangchengzu}
\begin{cases}
-\Delta f_n = g_n & \text{in } B_1(0), \\
f_n = 0 & \text{on } \partial B_1(0).
\end{cases}
\end{equation}
Using \eqref{g} and the elliptic regularity, we obtain
\begin{equation}\label{jm2.54}
f_n \to 0 \quad \text{uniformly in } B_1(0) \text{ as } n \to +\infty.
\end{equation}
Consider the difference $z_n - f_n$, which is harmonic on $B_1(0) \setminus \overline{B_{\frac{\varepsilon_n}{r}}(0)}$  (see \eqref{jm2.50} and \eqref{fangchengzu}),
and the maximum principle for harmonic functions gives
\begin{equation*}
\begin{aligned}
\|z_n - f_n\|_{L^\infty(B_1(0) \setminus \overline{B_{\frac{\varepsilon_n}{r}}(0)})}
&\leq \|z_n - f_n\|_{L^\infty(\partial B_1(0))} + \|z_n - f_n\|_{L^\infty(\partial B_{\frac{\varepsilon_n}{r}}(0))} \\
&\leq \|z_n\|_{L^\infty(\partial B_1(0))} + \|z_n\|_{L^\infty(\partial B_{\frac{\varepsilon_n}{r}}(0))} + \|f_n\|_{L^\infty(\partial B_{\frac{\varepsilon_n}{r}}(0))} \\
&\stackrel{\eqref{jm2.54}}{=} \|\widetilde{\phi}_{i,n}\|_{L^\infty(\partial B_1(0))} + \|\widetilde{\phi}_{i,n}\|_{L^\infty(\partial B_{\frac{\varepsilon_n}{r}}(0))} + o(1) \\
&\stackrel{\eqref{jm2.48}}{=} \|{\phi}_{i,n}\|_{L^\infty(\partial B_r(x_n))} + o(1) \\
&\stackrel{\eqref{eq:2.12}}{\leq} \|{\phi}_{i,n}\|_{L^\infty(\Omega \setminus B_{\frac{r}{2}}(x_\infty))} + o(1)
\stackrel{\eqref{jm2.45}}{=} o(1)
\end{aligned}
\end{equation*}
as $n \to +\infty$. To conclude, recalling again \eqref{jm2.54}, we obtain for $n \to +\infty$,
\[
\|z_n\|_{L^\infty(B_1(0) \setminus \overline{B_{\frac{\varepsilon_n}{r}}(0)})}
\leq \|f_n\|_{L^\infty(B_1(0))} + \|z_n - f_n\|_{L^\infty(B_1(0) \setminus \overline{B_{\frac{\varepsilon_n}{r}}(0)})}
= o(1).
\]
However, this is impossible, since \eqref{jm2.47} and \eqref{jm2.49} imply that for $n$ sufficiently
large,
\[
\frac{s_n}{|s_n|^2} \in B_1(0) \setminus \overline{B_{\frac{\varepsilon_n}{r}}(0)},
\]
and by definition
\[
z_n\left(\frac{s_n}{|s_n|^2}\right) = \widetilde{\phi}_{i,n}(s_n) \stackrel{\eqref{jm2.46}}{=} 1.
\]
\end{proof}

Next, we recall the kernel classification for the linearized nonlocal
Liouville equation proved in \cite[Theorem 1.1]{Gao1}.

\begin{lemma}\label{lem:2.12}
Assume that $\alpha\in(0,2)$ and that
$\phi\in C^2(\mathbb R^2)\cap L^\infty(\mathbb R^2)$ solves
\begin{equation*}
-\Delta \phi = \left( \displaystyle\int_{\mathbb{R}^2} \frac{e^{U(y)}
\phi(y)}{|x-y|^\alpha} \, dy \right) e^{U(x)} + \left(\displaystyle \int_{\mathbb{R}^2}
\frac{e^{U(y)}}{|x-y|^\alpha} \, dy \right) e^{U(x)} \phi(x)
\quad\text{in }\mathbb R^2,
\end{equation*}
where $U$ is defined in \eqref{eq:2.17}. Then
\begin{equation}\label{jm2.551}
\phi(x)
= \sum_{j=1}^2 a_j \frac{(4-\alpha) x_j}{C_\alpha^2+|x|^2}
+ b \frac{(4-\alpha)}{2}
\frac{C_\alpha^2-|x|^2}{C_\alpha^2+|x|^2}
\end{equation}
for some $a_j, b\in \mathbb{R}. $
\end{lemma}
To simplify notation, we write \eqref{jm2.551} as
\begin{equation*}
\phi(x)
= \sum_{j=1}^2  \frac{ a_j x_j}{C_\alpha^2+|x|^2}
+ b \frac{C_\alpha^2-|x|^2}{C_\alpha^2+|x|^2}.
\end{equation*}

\begin{lemma}\label{jmle2.13}
Fix $i\in\mathbb N$ and suppose that $\lambda_{i,n}\to1$. Then, after
passing to a subsequence, there exists
$(a_1^i,a_2^i,b^i)\in\mathbb R^3\setminus\{(0,0,0)\}$ such that
\begin{equation}\label{jm2.56}
\widetilde\phi_{i,n}(x)
= \sum_{j=1}^2  \frac{a_j^i x_j}{C_\alpha^2+|x|^2}
+ b^i \frac{C_\alpha^2-|x|^2}{C_\alpha^2+|x|^2}+ o(1)
 \quad \mathrm{in}\,C^1_{\mathrm{loc}}(\mathbb{R}^2)\,
 \mathrm{as } \, n \to +\infty,
\end{equation}
where $\widetilde{\phi}_{i,n}$ is the rescaled eigenfunction defined in \eqref{JM2.42}.
\end{lemma}
\begin{proof}
 By \eqref{JMPA2.43}, Lemma \ref{lem:shifted-potentials} and $\|\widetilde\phi_{i,n}\|_{L^\infty(\Omega_n)}=1$,
\[
\|\Delta\widetilde\phi_{i,n}\|_{L^\infty(\Omega_n)}\le C.
\]
Interior elliptic estimates yield
\[
\|\widetilde\phi_{i,n}\|_{C^{1,\beta}(B_R(0))}\le C_R
\qquad (R>0)
\]
for some $\beta\in(0,1)$. Thus, up to a subsequence, Lemma \ref{lem:2.11} gives
\[
\widetilde\phi_{i,n}\to\phi
\quad\text{in }C^1_{\mathrm{loc}}(\mathbb R^2),
\quad
\phi\not\equiv0,\quad \|\phi\|_{L^\infty(\mathbb R^2)}\le1.
\]
Splitting the convolution integrals into $B_R(0)$ and its complement, together with Lemma \ref{lem:shifted-potentials}, allows us to pass to the limit in \eqref{JMPA2.43}:
\[
-\Delta\phi
=
\left(\int_{\mathbb R^2}
\frac{e^{U(y)}\phi(y)}{|x-y|^\alpha}\,dy\right)e^{U(x)}
+
\left(\int_{\mathbb R^2}
\frac{e^{U(y)}}{|x-y|^\alpha}\,dy\right)e^{U(x)}\phi(x)
\quad\text{in }\mathbb R^2.
\]
By Lemma \ref{lem:2.12},
\[
\phi(x)=
\sum_{j=1}^{2}\frac{a_j^i x_j}{C_\alpha^2+|x|^2}
+b^i\frac{C_\alpha^2-|x|^2}{C_\alpha^2+|x|^2},
\qquad
(a_1^i,a_2^i,b^i)\ne(0,0,0).
\]
\end{proof}

\begin{proposition}\label{jmp2.14}
Suppose that, along a subsequence, $\lambda_{i,n}\to1$ and
\eqref{jm2.56} holds with $b^i\ne0$, where $i\in\mathbb N$. Then,
\begin{equation}\label{jm2.57}
p_n \phi_{i,n} = -2\pi(4 - \alpha) b^iG(x, x_\infty) + o(1) \quad \mathrm{in }\,
C^1_{\mathrm{loc}}\left(\overline{\Omega} \setminus \{x_\infty\}\right)
\end{equation}
and
\begin{equation}\label{jm2.58}
\lambda_{i,n} =1 + \frac{\frac{3}{2}(4-\alpha)}{p_n} \left(1 + o(1)\right) \quad \mathrm{as} \,
n \to +\infty.
\end{equation}
\end{proposition}
\begin{proof}
{\bf Step 1: proof of \eqref{jm2.57}.}
Since $\lambda_{i,n}\to1$, we have
$\lambda_{i,n}(2p_n+1)\ne1$ for all large $n$. The testing argument
used to prove \eqref{eq:sourceorth} therefore gives
\begin{equation}\label{app1}
\int_{\Omega}\int_{\Omega} \frac{u_{p_n}^{p_n+1}(y) u_{p_n}^{p_n}(x) \phi_{i,n}(x)}{|x-y|^\alpha}  dx dy = 0.
\end{equation}
Then
\begin{align}\label{jm2.59}
0&=\int_{\Omega}\int_{\Omega}  \frac{u_{p_n}^{p_n+1}(y) u_{p_n}^{p_n}(x) \phi_{i,n}(x)}{|x-y|^\alpha}  dx dy, \nonumber\\
&= \int_{\Omega}\int_{\Omega}  \frac{u_{p_n}^{p_n+1}(y) u_{p_n}^{p_n-1}(x) \left( u_{p_n}(x) - \|u_{p_n}\|_{L^\infty(\Omega)} + \|u_{p_n}\|_{L^\infty(\Omega)} \right) \phi_{i,n}(x)}{|x-y|^\alpha}dx dy \nonumber\\
&=\int_{\Omega}\int_{\Omega}  \frac{u_{p_n}^{p_n+1}(y) u_{p_n}^{p_n-1}(x) \left( u_{p_n}(x) - \|u_{p_n}\|_{L^\infty(\Omega)} \right) \phi_{i,n}(x)}{|x-y|^\alpha} dx dy\nonumber\\
&\quad+\int_{\Omega}\int_{\Omega}  \frac{u_{p_n}^{p_n+1}(y) u_{p_n}^{p_n-1}(x)  \|u_{p_n}\|_{L^\infty(\Omega)} \phi_{i,n}(x)}{|x-y|^\alpha} dx dy.
\end{align}
For the fixed radius chosen after Theorem \ref{TA1}, \eqref{eq:2.12} gives
\begin{align}\label{jm2.60}
&\int_{\Omega \setminus B_r(x_n)}\int_{\Omega}  \frac{u_{p_n}^{p_n+1}(y) u_{p_n}^{p_n-1}(x) \left(\|u_{p_n}\|_{L^\infty(\Omega)}-u_{p_n}(x) \right) |\phi_{i,n}(x)|}{|x-y|^\alpha} dxdy\nonumber\\
\leq&\|u_{p_n}\|_{L^\infty(\Omega)}\int_{\Omega \setminus B_r(x_n)}\int_{\Omega}  \frac{u_{p_n}^{p_n+1}(y) u_{p_n}^{p_n-1}(x)}{|x-y|^\alpha} dxdy\nonumber\\
\overset{(\star)}{\leq} &\|u_{p_n}\|_{L^\infty(\Omega)}\int_{\Omega \setminus B_{\frac{r}{2}}(x_{\infty})}\int_{\Omega}  \frac{u_{p_n}^{p_n+1}(y) u_{p_n}^{p_n-1}(x)}{|x-y|^\alpha} dxdy\nonumber\\
\overset{\eqref{eq225}}{=}&o\left(\frac{1}{p_n^2}\right),
\end{align}
where in $(\star)$ we have used that $B_{\frac{r}{2}}(x_\infty) \subseteq B_r(x_n)$ for $n$ large.
So from \eqref{jm2.59} and \eqref{jm2.60} we get
\begin{align}\label{jm2.61}
&-\int_{\Omega}\int_{\Omega}  \frac{u_{p_n}^{p_n+1}(y) u_{p_n}^{p_n-1}(x)  \|u_{p_n}\|_{L^\infty(\Omega)} \phi_{i,n}(x)}{|x-y|^\alpha} dx dy\nonumber\\=&\int_{ B_r(x_n)}\int_{\Omega}  \frac{u_{p_n}^{p_n+1}(y) u_{p_n}^{p_n-1}(x) \left( u_{p_n}(x) - \|u_{p_n}\|_{L^\infty(\Omega)} \right) \phi_{i,n}(x)}{|x-y|^\alpha} dx dy+o\left(\frac{1}{p_n^2}\right).
\end{align}
Next, recalling the definitions of $\widetilde{\phi}_{i,n}$ (see \eqref{JM2.42})
and $v_n $ (see \eqref{eq:2.15}),  by rescaling, we have
\begin{align}\label{jm2.62}
&\int_{ B_r(x_n)}\int_{\Omega}  \frac{u_{p_n}^{p_n+1}(y) u_{p_n}^{p_n-1}(x) \left( u_{p_n}(x) - \|u_{p_n}\|_{L^\infty(\Omega)} \right) \phi_{i,n}(x)}{|x-y|^\alpha} dx dy\nonumber\\
\overset{\eqref{JM2.42}}{=}&\varepsilon_n^{4-\alpha} \!\int_{B_{\frac{r}{\varepsilon_n}}(0)}\!\!\int_{\Omega_n} \frac{u_{p_n}^{p_n+1}(x_n+\varepsilon_n y) u_{p_n}^{p_n-1}(x_n+\varepsilon_n x) \left( u_{p_n}(x_n+\varepsilon_n x) - \|u_{p_n}\|_{L^\infty(\Omega)} \right) \widetilde{\phi}_{i,n}(x)}{|x-y|^\alpha}  dxdy \nonumber\\
\overset{\genfrac{}{}{0pt}{}{\eqref{eq:2.14}}{\eqref{eq:2.15}}}{=}&\frac{\|u_{p_n}\|_{L^\infty(\Omega)}}{p_n^2} \int_{B_{\frac{r}{\varepsilon_n}}(0)}\int_{\Omega_n} \frac{\left( 1 + \frac{v_n(y)}{p_n} \right)^{p_n+1} \left( 1 + \frac{v_n(x)}{p_n} \right)^{p_n-1} v_n(x) \widetilde{\phi}_{i,n}(x)}{|x-y|^\alpha}  dx dy\nonumber\\
\overset{(\star\star)}{=}&\frac{\|u_{p_n}\|_{L^\infty(\Omega)}}{p_n^2} \left(\displaystyle\int_{\mathbb{R}^2}\displaystyle\int_{\mathbb{R}^2} \frac{e^{U(y)}e^{U(x)}U(x)\left(\sum_{j=1}^2  \frac{a_j^i x_j}{C_\alpha^2+|x|^2} + b^i \frac{C_\alpha^2-|x|^2}{C_\alpha^2+|x|^2}\right)}{|x-y|^\alpha}dx dy+o(1)\right)\nonumber\\
\overset{\eqref{pce}}{=}&\frac{\|u_{p_n}\|_{L^\infty(\Omega)}}{p_n^2} \left(b^i \int_{\mathbb{R}^2} U(x) \frac{C_\alpha^2-|x|^2}{C_\alpha^2+|x|^2} \left( -\Delta U \right) dx+o(1)\right)\nonumber\\
\overset{\eqref{eq:2.17}}{=}&\frac{\|u_{p_n}\|_{L^\infty(\Omega)}}{p_n^2}
\left(\frac{\pi}{2} (4 - \alpha)^2 b^i+o(1)\right)
\end{align}
as $n\to+\infty$. The passage to the limit in $(\star\star)$ follows from
Lemmas \ref{lem:log-growth} and \ref{RE2.9}, applied to
$v_n\widetilde\phi_{i,n}$.
Substituting \eqref{jm2.62} into \eqref{jm2.61} yields
\begin{equation}\label{eq:2.65}
p_n^2 \int_{\Omega}\int_{\Omega}  \frac{u_{p_n}^{p_n+1}(y) u_{p_n}^{p_n-1}(x)  \phi_{i,n}(x)}{|x-y|^\alpha} dx dy = -\frac{\pi}{2} (4 - \alpha)^2 b^i + o(1).
\end{equation}
Fix $\delta\in(0,r)$ and
$x\in\Omega\setminus B_{2\delta}(x_\infty)$. Green's representation formula
gives
\begin{align*}
p_n \phi_{i,n} &= p_n \lambda_{i,n} \left( \int_\Omega G(x,y) (p_n+1) u_{p_n}^{p_n}(y) \left( \int_\Omega \frac{u_{p_n}^{p_n}(z) \phi_{i,n}(z)}{|z-y|^\alpha} dz \right) dy \right. \\
&\quad \left. + \int_\Omega p_n G(x,y) u_{p_n}^{p_n-1}(y) \phi_{i,n}(y) \left( \int_\Omega \frac{u_{p_n}^{p_n+1}(z)}{|z-y|^\alpha} dz \right) dy \right)\\
&= \lambda_{i,n} \left( G(x,x_n) p_n(p_n+1) \int_\Omega u_{p_n}^{p_n}(y) \left( \int_\Omega \frac{u_{p_n}^{p_n}(z) \phi_{i,n}(z)}{|z-y|^\alpha} dz \right) dy \right. \\
&\quad \left. + G(x,x_n) p_n^2 \int_\Omega u_{p_n}^{p_n-1}(y) \phi_{i,n}(y) \left( \int_\Omega \frac{u_{p_n}^{p_n+1}(z)}{|z-y|^\alpha} dz \right) dy \right) \\
&\quad +\underbrace{\lambda_{i,n} (p_n+1) p_n \int_\Omega \left( G(x,y) - G(x,x_n) \right) u_{p_n}^{p_n}(y) \left( \int_\Omega \frac{u_{p_n}^{p_n}(z) \phi_{i,n}(z)}{|z-y|^\alpha} dz \right) dy}_{=: I_{1,n}(x)} \\
&\quad  + \underbrace{\lambda_{i,n} p_n^2 \int_\Omega \left( G(x,y) - G(x,x_n) \right) u_{p_n}^{p_n-1}(y) \phi_{i,n}(y) \left( \int_\Omega \frac{u_{p_n}^{p_n+1}(z)}{|z-y|^\alpha} dz \right) dy}_{=: I_{2,n}(x)}\\
&\overset{{(\blacksquare)}}{=} -2\pi(4 - \alpha) b^iG(x, x_n)( 1+ o(1))+I_{1,n}(x)+I_{2,n}(x)\\
&\overset{\eqref{eq:2.12}}{=} -2\pi(4 - \alpha) b^iG(x, x_\infty)( 1+ o(1))+I_{1,n}(x)+I_{2,n}(x)
\end{align*}
uniformly as $n\to+\infty$. Indeed, the two source masses multiplying
$G(x,x_n)$ satisfy, by Lemma \ref{FULU} and \eqref{eq:2.65},
\[
-\frac{\pi}{2}\alpha(4-\alpha)b^i+o(1)
\quad\text{and}\quad
-\frac{\pi}{2}(4-\alpha)^2b^i+o(1),
\]
respectively. Their sum is
$-2\pi(4-\alpha)b^i+o(1)$, which proves the step
labelled $(\blacksquare)$.
We next prove that
\[
I_{1,n}=o(1),\qquad I_{2,n}=o(1)
\quad\text{in }C^1_{\mathrm{loc}}
(\overline\Omega\setminus\{x_\infty\}).
\]
Let $K\subset\subset\overline\Omega\setminus\{x_\infty\}$. Choose
$\rho\in(0,r)$ such that $K\cap B_{2\rho}(x_\infty)=\varnothing$.
For all large $n$, $B_\rho(x_n)\subset B_{2\rho}(x_\infty)$.
Denote by $I_{j,n}^{\mathrm{near}}$ and $I_{j,n}^{\mathrm{far}}$ the
contributions to $I_{j,n}$ from $B_\rho(x_n)$ and
$\Omega\setminus B_\rho(x_n)$, respectively.
The separated-variable Green estimates give, for $|\beta|\leq1$,
\[
\sup_{x\in K,\,y\in B_\rho(x_n)}
\left|\partial_x^\beta\bigl(G(x,y)-G(x,x_n)\bigr)\right|
\leq C_K|y-x_n|.
\]
Write $w_n=1+v_n/p_n$. Using $\|\phi_{i,n}\|_\infty=1$,
the change of variables $y=x_n+\varepsilon_n\eta$ and
$z=x_n+\varepsilon_n\xi$, and
$\varepsilon_n^{4-\alpha}u_{p_n}^{2p_n}(x_n)=p_n^{-1}$,
we obtain
\begin{align*}
\|I_{1,n}^{\mathrm{near}}\|_{C^1(K)}
&\leq C_K(p_n+1)\varepsilon_n
\int_{B_{\rho/\varepsilon_n}(0)}\int_{\Omega_n}
\frac{|\eta|w_n^{p_n}(\eta)w_n^{p_n}(\xi)}
{|\eta-\xi|^\alpha}\,d\xi d\eta,\\
\|I_{2,n}^{\mathrm{near}}\|_{C^1(K)}
&\leq C_Kp_n\varepsilon_n
\int_{B_{\rho/\varepsilon_n}(0)}\int_{\Omega_n}
\frac{|\eta|w_n^{p_n-1}(\eta)w_n^{p_n+1}(\xi)}
{|\eta-\xi|^\alpha}\,d\xi d\eta.
\end{align*}
Both double integrals are uniformly bounded by Lemma \ref{RE2.9}, applied
with test functions of at most linear growth. Since
$p_n\varepsilon_n=o(1)$, the two near contributions tend to zero in
$C^1(K)$.
For the remaining part, $\Omega\setminus B_\rho(x_n)$ is contained in a
fixed compact subset of $\overline\Omega\setminus\{x_\infty\}$. Lemma
\ref{lem:offpeak}, with $\gamma_n=2$ and $m=0$, gives
\[
\sup_{y\in\Omega\setminus B_\rho(x_n)}p_n^2
\left\{u_{p_n}^{p_n}(y)
\int_\Omega\frac{u_{p_n}^{p_n}(z)}{|z-y|^\alpha}\,dz
+u_{p_n}^{p_n-1}(y)
\int_\Omega\frac{u_{p_n}^{p_n+1}(z)}{|z-y|^\alpha}\,dz
\right\}=o(1).
\]
The uniform $L^1$ bounds for $G(x,\cdot)$ and $\nabla_xG(x,\cdot)$,
together with the separated-variable bounds for $G(x,x_n)$ and
$\nabla_xG(x,x_n)$, therefore imply
\[
\|I_{1,n}^{\mathrm{far}}\|_{C^1(K)}
+\|I_{2,n}^{\mathrm{far}}\|_{C^1(K)}=o(1).
\]
Since $K$ was arbitrary, this proves the asserted $C^1_{\mathrm{loc}}$
convergence and completes the proof of \eqref{jm2.57}.

{\bf Step 2: proof of \eqref{jm2.58}.}
The Rayleigh quotient of $u_{p_n}$ gives
$\lambda_{1,n}\leq(2p_n+1)^{-1}$, so the present index satisfies
$i\geq2$. Substituting $z=x_n$ into \eqref{wen2.6}, we obtain
\begin{align}\label{jm2.73}
\int_{\partial\Omega} \frac{\partial u_{p_n}}{\partial \nu} \frac{\partial \phi_{i,n}}{\partial \nu} (x-x_n) \cdot \nu \, d\sigma_x
= &(1-\lambda_{i,n}) \Bigg\{ p_n \int_\Omega u_{p_n}^{p_n-1} \phi_{i,n} \omega_{p_n} \left( \int_\Omega \frac{u_{p_n}^{{p_n}+1}(y)}{|x-y|^\alpha} dy \right) dx \nonumber \\
& + (p_n+1) \int_\Omega\int_\Omega \frac{u_{p_n}^{p_n}(x) \phi_{i,n}(x) u_{p_n}^{p_n}(y) \omega_{p_n}(y)}{|x-y|^\alpha} dx dy \Bigg\}.
\end{align}
Equations \eqref{eq:2.12}, \eqref{eq:2.8}, and \eqref{jm2.57} give the
following expansion of the left-hand side of \eqref{jm2.73}:
\begin{align}\label{jm2.74}
\int_{\partial\Omega} \frac{\partial u_{p_n}}{\partial \nu} \frac{\partial \phi_{i,n}}{\partial \nu} (x-x_n) \cdot \nu \, d\sigma_x
=& -\frac{4(4-\alpha)^2{\pi}^2\sqrt{e}b^i}{p_n^2} \int_{\partial\Omega} (x - x_\infty) \cdot \nu(x) \left( \frac{\partial G}{\partial \nu}(x, x_\infty) \right)^2 d\sigma_x+o\left(\frac{1}{p_n^2}\right) \nonumber\\
\stackrel{\eqref{eq:2.3}}{=}& -\frac{2(4-\alpha)^2\pi\sqrt{e}b^i}{p_n^2}
+o\left(\frac{1}{p_n^2}\right)
\end{align}
as $n\to+\infty.$ On the other hand, by the definition of $\tilde{u}_{p_n}$ (see \eqref{bianhuan1}), we multiply the bracket on the right-hand side of \eqref{jm2.73} by $p_n$ and obtain
\begin{align}\label{jm2.75}
&p_n^2 \int_\Omega u_{p_n}^{p_n-1}(x) \phi_{i,n}(x)\cdot(x-x_n)\cdot\nabla u_{p_n}(x) \left( \int_\Omega \frac{u_{p_n}^{p_n+1}(y)}{|x-y|^\alpha}dy \right) dx\nonumber\\
&\quad+{p_n}({p_n}+1) \int_{\Omega}\int_{\Omega} \frac{u_{p_n}^{p_n}(y)\cdot(y-x_n)\cdot\nabla u_{p_n}(y)u_{p_n}^{p_n}(x)\phi_{i,n}(x)}{|x-y|^\alpha}dxdy \nonumber\\
\overset{(\spadesuit)}{=}&p_n^2 \int_\Omega u_{p_n}^{p_n-1}(x) \phi_{i,n}(x)\cdot(x-x_n)\cdot\nabla u_{p_n}(x) \left( \int_\Omega \frac{u_{p_n}^{p_n+1}(y)}{|x-y|^\alpha}dy \right) dx\nonumber\\
&\quad+{p_n}({p_n}+1) \int_{\Omega}\int_{\Omega} \frac{u_{p_n}^{p_n}(x)\cdot(x-x_n)\cdot\nabla u_{p_n}(x)u_{p_n}^{p_n}(y)\phi_{i,n}(y)}{|x-y|^\alpha}dxdy \nonumber\\
\overset{(\star)}{=} &p_n^2 \int_{B_r(x_n)} u_{p_n}^{p_n-1}(x) \phi_{i,n}(x)\cdot(x-x_n)\cdot\nabla u_{p_n}(x) \left( \int_\Omega \frac{u_{p_n}^{{p_n}+1}(y)}{|x-y|^\alpha}dy \right) dx \nonumber\\
&\quad+ {p_n}({p_n}+1) \int_{B_r(x_n)} u_{p_n}^{p_n}(x)\cdot(x-x_n)\cdot\nabla u_{p_n}(x) \left( \int_\Omega \frac{u_{p_n}^{p_n}(y)\phi_{i,n}(y)}{|x-y|^\alpha}dy \right) dx + o(1)\nonumber \\
= &\int_{B_{\frac{r}{\varepsilon_n}}(0)} p_n\left(1+\frac{v_n(\eta)}{p_n}\right)^{p_n-1} \widetilde{\phi}_{i,n}(\eta)\cdot\eta\cdot\nabla\widetilde{u}_{p_n}(\eta) \left( \int_{\Omega_n} \frac{\left(1+\frac{v_n(w)}{p_n}\right)^{p_n+1}}{|\eta-w|^\alpha}dw \right) d\eta \nonumber \\
&\quad+ \int_{B_{\frac{r}{\varepsilon_n}}(0)} (p_n+1)\left(1+\frac{v_n(\eta)}
{p_n}\right)^{p_n} \cdot\eta\cdot\nabla\widetilde{u}_{p_n}(\eta) \left( \int_{\Omega_n}
\frac{\left(1+\frac{v_n(w)}{p_n}\right)^{p_n} \widetilde{\phi}_{i,n}(w)}{|\eta-w|^\alpha}dw
\right)d\eta+o(1)
\end{align}
as $n\to+\infty$, where $r$ is the fixed radius chosen after Theorem
\ref{TA1}, and we used symmetry in $(\spadesuit)$.
Then, as $n \to +\infty$, one gets
\begin{align*}
&\Bigg|p_n^2 \int_{\Omega\setminus B_r(x_n)} u_{p_n}^{p_n-1}(x) \phi_{i,n}(x)\cdot(x-x_n)\cdot\nabla u_{p_n}(x) \left( \int_\Omega \frac{u_{p_n}^{p_n+1}(y)}{|x-y|^\alpha}dy \right) dx\nonumber\\
&\quad\quad\quad\quad+{p_n}({p_n}+1) \int_{\Omega\setminus B_r(x_n)}\int_{\Omega} \frac{u_{p_n}^{p_n}(x)\cdot(x-x_n)\cdot\nabla u_{p_n}(x)u_{p_n}^{p_n}(y)\phi_{i,n}(y)}{|x-y|^\alpha}dxdy \Bigg| \nonumber\\
\le&\Bigg|p_n^2 \int_{\Omega\setminus B_r(x_n)} u_{p_n}^{p_n-1}(x) \phi_{i,n}(x)\cdot(x-x_n)\cdot\nabla u_{p_n}(x) \left( \int_\Omega \frac{u_{p_n}^{p_n+1}(y)}{|x-y|^\alpha}dy \right) dx \Bigg|\nonumber\\
&\quad\quad\quad\quad+\Bigg|{p_n}({p_n}+1) \int_{\Omega\setminus B_r(x_n)}\int_{\Omega} \frac{u_{p_n}^{p_n}(x)\cdot(x-x_n)\cdot\nabla u_{p_n}(x)u_{p_n}^{p_n}(y)\phi_{i,n}(y)}{|x-y|^\alpha}dxdy \Bigg|\nonumber\\
\overset{\eqref{J2.37}}{\le}&C|\Omega|p_n \left(\left\| \left( \int_\Omega \frac{u_{p_n}^{p_{n}+1}(y)}{|x - y|^\alpha} dy \right) u_{p_n}^{p_n-1}(x) \right\|_{L^\infty({\Omega\setminus B_{\frac{r}{2}}(x_{\infty})})}+ \left\| \left( \int_\Omega \frac{u_{p_n}^{p_{n}}(y)}{|x - y|^\alpha} dy \right) u_{p_n}^{p_n}(x) \right\|_{L^\infty({\Omega\setminus B_{\frac{r}{2}}(x_{\infty})})}\right)\nonumber\\
\overset{\genfrac{}{}{0pt}{}{\eqref{eq223}}{\eqref{eq224}}}{=}&o(1).
\end{align*}
By Lemma \ref{RE2.9}, passing the limit in \eqref{jm2.75}, we have
\begin{align}\label{jm2.76}
&\int_{B_{\frac{r}{\varepsilon_n}}(0)} p_n\left(1+\frac{v_n(\eta)}{p_n}\right)^{p_n-1} \widetilde{\phi}_{i,n}(\eta)\cdot\eta\cdot\nabla\widetilde{u}_{p_n}(\eta) \left( \int_{\Omega_n} \frac{\left(1+\frac{v_n(w)}{p_n}\right)^{p_n+1}}{|\eta-w|^\alpha}dw \right) d\eta \nonumber \\
&\quad\quad\quad\quad+ \int_{B_{\frac{r}{\varepsilon_n}}(0)} (p_n+1)\left(1+\frac{v_n(\eta)}{p_n}\right)^{p_n} \cdot\eta\cdot\nabla\widetilde{u}_{p_n}(\eta) \left( \int_{\Omega_n} \frac{\left(1+\frac{v_n(w)}{p_n}\right)^{p_n} \widetilde{\phi}_{i,n}(w)}{|\eta-w|^\alpha}dw \right)d\eta\nonumber \\
= &\sqrt{e} \int_{\mathbb{R}^2} e^{U(\eta)} \left( \sum_{j=1}^2 \frac{a_j^i \eta_j}{C_\alpha^2+|\eta|^2} + b^i \frac{C_\alpha^2-|\eta|^2}{C_\alpha^2+|\eta|^2} \right)\cdot \eta\cdot\nabla U(\eta) \left( \int_{\mathbb{R}^2} \frac{e^{U(w)}}{|\eta-w|^\alpha}\,dw \right) d\eta \nonumber\\
&\quad\quad\quad\quad+\sqrt{e} \int_{\mathbb{R}^2} e^{U(\eta)} \cdot\eta\cdot\nabla U(\eta) \left( \displaystyle\int_{\mathbb{R}^2} \frac{e^{U(w)} \left( \sum_{j=1}^2 \frac{a_j^i w_j}{C_\alpha^2+|w|^2} + b^i \frac{C_\alpha^2-|w|^2}{C_\alpha^2+|w|^2} \right)}{|\eta-w|^\alpha}\,dw \right) d\eta + o(1)\nonumber\\
=&\sqrt{e}\, b^i \int_{\mathbb{R}^2} \left(\eta \cdot \nabla U(\eta)\right) \left(-\Delta \left(\frac{C_\alpha^2 - |\eta|^2}{C_\alpha^2 + |\eta|^2}\right)\right) d\eta+o(1)
\overset{\eqref{eq:2.17}}{=}\frac{4\pi}{3} (4 - \alpha) \sqrt{e} b^i+o(1).
\end{align}
Here the translation terms vanish by oddness. The last constant follows
from $t=|\eta|^2/C_\alpha^2$ and
\[
\int_{\mathbb R^2}(\eta\cdot\nabla U)
\left(-\Delta\frac{C_\alpha^2-|\eta|^2}{C_\alpha^2+|\eta|^2}\right)d\eta
=-8\pi(4-\alpha)\int_0^\infty\frac{t(1-t)}{(1+t)^4}\,dt
=\frac{4\pi}{3}(4-\alpha).
\]
Substituting \eqref{jm2.74}--\eqref{jm2.76} into \eqref{jm2.73} yields
\begin{equation*}
-\frac{2(4-\alpha)^2\pi\sqrt{e}b^i}{p_n^2}(1+o(1))=(1-\lambda_{i,n})\frac{1}{p_n}\frac{4\pi}{3} (4 - \alpha)
\sqrt{e} b^i(1+o(1))  \quad \mathrm{as}\, \, n \to +\infty.
\end{equation*}
 Since  $b^i\not=0,$ this implies \eqref{jm2.58}.
\end{proof}
\section{Estimates for the first eigenpair}\label{se3}

We prove Theorem \ref{Th1} by the variational characterization of the first
eigenvalue, keeping track of equality in the Rayleigh quotient.

\begin{proof}[Proof of Theorem \ref{Th1}]
Fix $n$. By the variational framework in Section \ref{se2},
\[
\lambda_{1,n}=\inf_{0\ne\phi\in H_0^1(\Omega)}
\frac{\int_\Omega|\nabla\phi|^2\,dx}{B_{p_n}(\phi,\phi)}.
\]
Testing with $u_{p_n}$ gives
\begin{equation}\label{lelamda}
\lambda_{1,n}\leq
\frac{\int_\Omega|\nabla u_{p_n}|^2\,dx}
{(2p_n+1)\int_\Omega|\nabla u_{p_n}|^2\,dx}
=\frac1{2p_n+1}.
\end{equation}
To prove the reverse inequality, put
\[
A_n(\phi):=\int_\Omega u_{p_n}^{p_n-1}(x)\phi^2(x)
\left(\int_\Omega\frac{u_{p_n}^{p_n+1}(y)}{|x-y|^\alpha}\,dy\right)dx.
\]
For $\phi\in C_c^\infty(\Omega)$, the equation for $u_{p_n}>0$ gives
\[
\int_\Omega|\nabla\phi|^2\,dx-A_n(\phi)
=\int_\Omega u_{p_n}^2
\left|\nabla\left(\frac{\phi}{u_{p_n}}\right)\right|^2dx\geq0.
\]
The resulting inequality extends to $H_0^1(\Omega)$ by density, since
the coefficient in $A_n$ is bounded for fixed $n$. Symmetrization also
gives, for every $\phi\in H_0^1(\Omega)$,
\begin{align*}
A_n(\phi)-\int_\Omega\int_\Omega
\frac{u_{p_n}^{p_n}(x)\phi(x)u_{p_n}^{p_n}(y)\phi(y)}{|x-y|^\alpha}\,dxdy=\frac12\int_\Omega\int_\Omega
\frac{u_{p_n}^{p_n+1}(x)u_{p_n}^{p_n+1}(y)}{|x-y|^\alpha}
\left(\frac{\phi(x)}{u_{p_n}(x)}
-\frac{\phi(y)}{u_{p_n}(y)}\right)^2dxdy\geq0.
\end{align*}
Consequently,
\[
B_{p_n}(\phi,\phi)\leq(2p_n+1)A_n(\phi)
\leq(2p_n+1)\int_\Omega|\nabla\phi|^2\,dx.
\]
Together with \eqref{lelamda}, this proves
$\lambda_{1,n}=(2p_n+1)^{-1}$. Equality in the double-integral inequality
forces $\phi/u_{p_n}$ to be constant almost everywhere. Hence the first
eigenvalue is simple and, with the positive $L^\infty$ normalization,
\[
\phi_{1,n}=\frac{u_{p_n}}{\|u_{p_n}\|_{L^\infty(\Omega)}}.
\]
Since $x_n$ is a maximum point, the rescaled function in Theorem
\ref{Th1} satisfies the exact identity
\[
\varphi_n(\xi)=\frac{u_{p_n}(x_n+\varepsilon_n\xi)}{u_{p_n}(x_n)}
=1+\frac{v_n(\xi)}{p_n}.
\]
The convergence \eqref{eq:2.16} therefore yields
$\varphi_n\to1$ in $C^1_{\mathrm{loc}}(\mathbb R^2)$.
\end{proof}

\section{Estimates for the second and third eigenpairs}\label{se4}
In this section, we prove Theorem \ref{th2}. We first estimate the
Rayleigh quotient on the space generated by $u_{p_n}$ and the cutoff
derivatives of $u_{p_n}$. We then identify the translation profiles and
use Green's representation and the Poho\v{z}aev identities to obtain the
outer expansion and the sharp eigenvalue asymptotics.
\begin{proposition}\label{jmp3.1}
For $i=2,3$, there exists a constant $C>0$ such that
\begin{equation}\label{jm3.1}
\lambda_{i,n} \leq 1 + C\varepsilon_n^2\quad
\mathrm{for} \,\  n \,\ \mathrm{large}.
\end{equation}
In addition,
\begin{equation}\label{jm3.2}
\lambda_{i,n} = 1 + o(1)
\quad \text{as } n\to+\infty.
\end{equation}
\end{proposition}
\begin{proof}
Using the variational characterization of eigenvalues associated with \eqref{tezhengzhi1}, we have

\begin{equation}\label{jm3.3}
\lambda_{i,n}= \inf_{\substack{W\subset H_0^1(\Omega)\\ \dim W=i}} \max_{\phi\in W\setminus\{0\}} \frac{\displaystyle\int_\Omega |\nabla \phi|^2 dx}{B_{p_n}(\phi,\phi)}.
\end{equation}
Since $L_{p_n}(\partial_{x_j}u_{p_n})=0$ in $\Omega$ for $j=1,2$, we use their cutoff versions as test functions, where
\[
L_{p_n}(\phi):=-\Delta\phi
-p_nu_{p_n}^{p_n-1}\phi
\int_\Omega\frac{u_{p_n}^{p_n+1}(y)}{|x-y|^\alpha}\,dy
-(p_n+1)u_{p_n}^{p_n}
\int_\Omega\frac{u_{p_n}^{p_n}(y)\phi(y)}{|x-y|^\alpha}\,dy.
\]
Let $r>0$ be the fixed radius chosen after Theorem \ref{TA1}, and let
$\widetilde{\chi}\in C_0^\infty(B_r(0))$ be such that
 $\widetilde{\chi} \equiv 1$ in $B_{\frac{r}{2}}(0)$, $0 \leq \widetilde{\chi} \leq 1$ in $B_r(0)$.
We then define the functions in $H_0^1(\Omega)$ by
\[
\psi_{i,n} := \frac{\partial u_{p_n}}{\partial x_i} \chi_n,\quad i = 1, 2,
\]
where $\chi_n(x) := \widetilde{\chi}(x - x_n)$ and $x_n$ is the global
maximum point selected in item~(4) of Theorem \ref{TA1}. Denote
\[
W_{i,n} := \operatorname{Span}\{u_{p_n},\, \psi_{j,n},\, j = 1,\dots,i-1\}, \quad i = 2, 3.
\]
The functions $u_{p_n},\psi_{1,n},\psi_{2,n}$ are linearly independent
for all large $n$. Indeed, otherwise write a nontrivial relation as
\[
a_{0,n}u_{p_n}+a_{1,n}\psi_{1,n}+a_{2,n}\psi_{2,n}=0
\]
and set
\[
A_{0,n}=a_{0,n},\qquad
A_{j,n}=\frac{a_{j,n}}{p_n\varepsilon_n}\quad (j=1,2),
\]
after multiplying all coefficients so that
$\max_{0\leq j\leq2}|A_{j,n}|=1$. Rescaling in
$B_{r/(2\varepsilon_n)}(0)$ and passing to a subsequence would give the
nontrivial relation
\[
A_0+A_1\frac{\partial U}{\partial x_1}
+A_2\frac{\partial U}{\partial x_2}=0.
\]
This is impossible because $1,\partial_{x_1}U,\partial_{x_2}U$ are linearly
independent. Hence $\dim W_{i,n}=i$. Every $\phi\in W_{i,n}$ can be written as
\begin{equation}\label{jm3.4}
\phi = a_{0,n} u_{p_n} + \sum_{j=1}^{i-1} a_{j,n} \psi_{j,n} = a_{0,n} u_{p_n}
+ \chi_n z_n,
\end{equation}
where
\begin{equation*}
z_n := \sum_{j=1}^{i-1} a_{j,n} \frac{\partial u_{p_n}}{\partial x_j}.
\end{equation*}
Thus, we have
\begin{align*}
\int_\Omega |\nabla\phi|^2 \, dx
&= \int_\Omega |\nabla(a_{0,n} u_{p_n}
+ \chi_n z_n)|^2 \, dx \\
&= a_{0,n}^2 \int_\Omega |\nabla u_{p_n}|^2 \, dx + 2a_{0,n} \int_\Omega \nabla u_{p_n}
 \cdot \nabla(\chi_n z_{n}) \, dx +
\int_\Omega |\nabla(\chi_n z_n)|^2 \, dx.
\end{align*}
A direct computation gives
\begin{align*}
\int_\Omega \nabla u_{p_n} \cdot \nabla(\chi_n z_n) \, dx
= &\int_\Omega (-\Delta u_{p_n}) \chi_n z_n \, dx
\\=& \int_\Omega u_{p_n}^{p_n}(x) \chi_n(x) z_n(x)
 \left( \int_\Omega \frac{u_{p_n}^{p_n+1}(y)}{|x - y|^\alpha} \, dy \right) dx.
\end{align*}
Clearly, $z_n$ satisfies
\begin{equation}\label{wen4.13}
-\Delta z_n = p_n u_{p_n}^{{p_n}-1} z_n \left( \displaystyle\int_{\Omega}
\frac{u_{p_n}^{{p_n}+1}(y)}{|x - y|^\alpha} dy \right)+({p_n}+1) u_{p_n}^{{p_n}}
\left( \displaystyle \int_{\Omega} \frac{u_{p_n}^{{p_n}}(y) z_n(y)}{|x - y|^\alpha} dy \right)
  \quad \mbox{in} \  \Omega.
\end{equation}
Multiplying \eqref{wen4.13} by $\chi_n^2 z_n$ and integrating, we have
\begin{align*}
& \int_\Omega \chi_n^2 |\nabla z_n|^2 dx + 2 \int_\Omega \chi_n z_n
\nabla\chi_n \cdot \nabla z_n dx \nonumber\\
=& p_n \int_{\Omega}\int_{\Omega} \frac{u_{p_n}^{p_n-1}(x) \chi_n^2(x)
 z_n^2(x) u_{p_n}^{p_n+1}(y)}{|x - y|^\alpha} dx dy
 +( p_n+1) \int_{\Omega}\int_{\Omega} \frac{u_{p_n}^{p_n}(x)
 \chi_n^2(x) z_n(x) u_{p_n}^{p_n}(y) z_n(y)}{|x - y|^\alpha} dx dy,
\end{align*}
then
\begin{align*}
\int_\Omega |\nabla(\chi_n z_n)|^2 dx
&= \int_\Omega |\nabla \chi_n|^2 |z_n|^2 dx +
\int_\Omega |\chi_n|^2 |\nabla z_n|^2 dx +
 2 \int_\Omega \chi_n z_n \nabla\chi_n \cdot \nabla z_n dx \\
&= \int_\Omega |\nabla \chi_n|^2 |z_n|^2 dx +
 p_n \int_{\Omega}\int_{\Omega}
\frac{u_{p_n}^{p_n-1}(x) \chi_n^2(x) z_n^2(x)
u_{p_n}^{p_n+1}(y)}{|x - y|^\alpha} dx dy \\
& \quad + (p_n+1) \int_{\Omega}\int_{\Omega}
\frac{u_{p_n}^{p_n}(x) \chi_n^2(x) z_n(x)
 u_{p_n}^{p_n}(y) z_n(y)}{|x - y|^\alpha} dx dy.
\end{align*}
Therefore, by \eqref{jm3.3} and \eqref{jm3.4}, we have
\begin{align}\label{wen4.16}
\lambda_{i,n}
&\leq \max_{\substack{\phi \in W_{i,n} \\ \phi \neq 0}}
\frac{\int_\Omega |\nabla\phi|^2 \, dx}{{
p_n \displaystyle\int_\Omega u_{p_n}^{p_n-1} \phi^2 \left( \int_\Omega
\frac{u_{p_n}^{p_n+1}(y)}{|x-y|^\alpha} dy \right) dx
+ (p_n+1) \int_{\Omega}\int_{\Omega}\frac{u_{p_n}^{p_n}(x) \phi(x)
 u_{p_n}^{p_n}(y) \phi(y)}{|x-y|^\alpha} dx dy
}} \nonumber\\
&\leq \max_{\substack{(a_{0,n},\dots,a_{i-1,n}) \in \mathbb{R}^i \\ \sum_{j=0}^{i-1}
a_{j,n}^2 = 1}} \left\{1+
\frac{N(a_{0,n},\dots,a_{i-1,n})}{D(a_{0,n},\dots,a_{i-1,n})}\right\},
\end{align}
where
\begin{equation}\label{wen4.17}
N(a_{0,n},\dots,a_{i-1,n}):=N_{1,n}+N_{2,n}+N_{3,n},
\end{equation}
with
\begin{equation}\label{wen4.18}
N_{1,n} := -2p_n a_{0,n}^2  \int_{\Omega}\int_{\Omega}
\frac{u_{p_n}^{p_n+1}(x) u_{p_n}^{p_n+1}(y)}{|x - y|^\alpha} \, dx dy,
\end{equation}
\begin{equation}\label{wen4.19}
N_{2,n}:= -4p_na_{0,n} \int_{\Omega} \int_{\Omega}
\frac{u_{p_n}^{p_n}(x) \chi_n(x)
\left( \sum_{j=1}^{i-1} a_{j,n} \frac{\partial u_{p_n}}{\partial x_j}(x) \right)
u_{p_n}^{p_n+1}(y)}{|x - y|^\alpha} \, dx dy,
\end{equation}
\begin{align}\label{wen4.20}
N_{3,n}
&:= \underbrace{\int_\Omega |\nabla \chi_n|^2 \left( \sum_{j=1}^{i-1} a_{j,n}
\frac{\partial u_{p_n}}{\partial x_j}(x) \right) \left( \sum_{l=1}^{i-1} a_{l,n}
 \frac{\partial u_{p_n}}{\partial x_l}(x) \right) dx}_{:= N_{3,n}^{(1)}}\nonumber \\
&+\,\underbrace{({p_n+1}) \int_\Omega u_{p_n}^{p_n} \chi_n
 \left( \sum_{j=1}^{i-1} a_{j,n} \frac{\partial u_{p_n}}{\partial x_j}(x) \right)
 \left( \int_\Omega \frac{u_{p_n}^{p_n}(y) \big(\chi_n(x) - \chi_n(y)\big)
  \left( \sum_{l=1}^{i-1} a_{l,n} \frac{\partial u_{p_n}}{\partial y_l}(y) \right)}
  {|x - y|^\alpha} dy \right) dx}_{:= N_{3,n}^{(2)}},
\end{align}
and
\begin{equation}\label{wen4.21}
D(a_{0,n},\dots,a_{i-1,n}):=D_{1,n}+D_{2,n}+D_{3,n},
\end{equation}
with
\begin{equation}\label{wen4.22}
D_{1,n}:= (2p_n+1) a_{0,n}^2  \int_{\Omega}\int_{\Omega}
\frac{u_{p_n}^{p_n+1}(x) u_{p_n}^{p_n+1}(y)}{|x - y|^\alpha} \, dx dy,
\end{equation}
\begin{equation}\label{wen4.23}
D_{2,n}:= 2(2p_n+1)a_{0,n} \int_{\Omega} \int_{\Omega}
\frac{u_{p_n}^{p_n}(x) \chi_n(x)
\left( \sum_{j=1}^{i-1} a_{j,n} \frac{\partial u_{p_n}}{\partial x_j}(x) \right)
u_{p_n}^{p_n+1}(y)}{|x - y|^\alpha} \, dx dy,
\end{equation}
\begin{align}\label{wen4.24}
D_{3,n}
&:= p_n \int_{\Omega}\int_{\Omega} \frac{u_{p_n}^{p_n-1}(x)
 u_{p_n}^{p_n+1}(y) \chi_n^2(x) \left( \sum_{j=1}^{i-1} a_{j,n} \frac{\partial u_{p_n}}
 {\partial x_j}(x) \right) \left( \sum_{l=1}^{i-1} a_{l,n} \frac{\partial u_{p_n}}
 {\partial x_l}(x) \right)}{|x - y|^\alpha} dx dy \nonumber\\
&\quad + ({p_n+1})\int_\Omega u_{p_n}^{p_n}(x) \chi_n(x)
\left( \sum_{j=1}^{i-1} a_{j,n} \frac{\partial u_{p_n}}{\partial x_j}(x) \right)
 \left( \int_\Omega \frac{u_{p_n}^{p_n}(y) \chi_n(y)
 \left( \sum_{l=1}^{i-1} a_{l,n} \frac{\partial u_{p_n}}{\partial y_l}(y) \right)}
 {|x - y|^\alpha} dy \right) dx.
\end{align}
The rest of the proof is divided into three steps.

 {\bf Step 1.}
We first establish the estimates needed for \eqref{wen4.16}.
For \eqref{wen4.23}, applying integration by parts,  we obtain
\begin{align}\label{wen4.26}
& (2p_n+1)\int_{\Omega}\int_{\Omega } \frac{u_{p_n}^{p_n}(x) \chi_n(x)
 \left( \sum_{j=1}^{i-1} a_{j,n} \frac{\partial u_{p_n}}{\partial x_j}(x) \right)
  u_{p_n}^{p_n+1}(y)}{|x - y|^\alpha} \, dx dy \nonumber\\
= &\frac{2p_n+1}{p_n+1} \int_{\Omega}\int_{\Omega } \frac{\chi_n(x)
\left( \sum_{j=1}^{i-1} a_{j,n} \frac{\partial}{\partial x_j} u_{p_n}^{p_n+1}(x) \right)
 u_{p_n}^{p_n+1}(y)}{|x - y|^\alpha} \, dx dy\nonumber \\
= &\underbrace{\frac{2p_n+1}{p_n+1} \int_{\Omega}\int_{\Omega }
\frac{\left( -\sum_{j=1}^{i-1} a_{j,n} \frac{\partial}{\partial x_j} \chi_n(x) \right)
 u_{p_n}^{p_n+1}(x) u_{p_n}^{p_n+1}(y)}{|x - y|^\alpha} \, dx dy}_{:= \widehat{H_1}}\nonumber \\
& \quad + \underbrace{\frac{2p_n+1}{p_n+1} \int_{\Omega}\int_{\Omega }
\left( -\sum_{j=1}^{i-1} a_{j,n} \frac{\partial}{\partial x_{j}}
\left( \frac{1}{|x - y|^\alpha} \right) \right) \chi_n(x) u_{p_n}^{p_n+1}(x)
 u_{p_n}^{p_n+1}(y) \, dx dy}_{:= \widehat{H_2}}.
\end{align}
From \eqref{wenC4.5}, we have
\begin{equation}\label{wen4.27}
\widehat{H}_1
=\frac{2p_n+1}{p_n+1} \int_{\Omega \cap \{|x-x_n| \geq \frac{r}{2}\}}\int_{\Omega}
\frac{\left( -\sum_{j=1}^{i-1} a_{j,n} \frac{\partial}{\partial x_j} \chi_n(x) \right)
 u_{p_n}^{p_n+1}(x) u_{p_n}^{p_n+1}(y)}{|x - y|^\alpha} \, dx dy
 =o\left(\frac{1}{p_n^2}\right).
\end{equation}
The derivative of the Riesz kernel must be symmetrized before it is estimated.
Interchanging $x$ and $y$ in $\widehat H_2$ and using
$\nabla_y|x-y|^{-\alpha}=-\nabla_x|x-y|^{-\alpha}$, we obtain
\begin{align*}
\widehat H_2
=\frac{2p_n+1}{2(p_n+1)}\int_\Omega\int_\Omega
\left(\sum_{j=1}^{i-1}a_{j,n}\frac{\partial}{\partial x_j}|x-y|^{-\alpha}\right)
\bigl(\chi_n(y)-\chi_n(x)\bigr)
u_{p_n}^{p_n+1}(x)u_{p_n}^{p_n+1}(y)\,dxdy.
\end{align*}
This identity is justified first for
$(|x-y|^2+\delta^2)^{-\alpha/2}$ and then by letting $\delta\to 0^+$.
Indeed,
\[
|\nabla_x|x-y|^{-\alpha}|\,|\chi_n(x)-\chi_n(y)|
\le C|x-y|^{-\alpha}.
\]
The difference of the cutoffs vanishes on
$B_{r/2}(x_n)\times B_{r/2}(x_n)$. Hence \eqref{wenC4.5} gives
\[
|\widehat H_2|
\le C\int_{\Omega\setminus B_{r/2}(x_n)}\int_\Omega
\frac{u_{p_n}^{p_n+1}(x)u_{p_n}^{p_n+1}(y)}{|x-y|^\alpha}\,dydx
=o(p_n^{-2}).
\]
Combining \eqref{wen4.23}, \eqref{wen4.26} and \eqref{wen4.27}, we conclude that
\begin{equation}\label{wen4.29}
D_{2,n} = 2a_{0,n}(\widehat{H}_1 + \widehat{H}_2) = o\left(\frac{1}{p_n^2}\right).
\end{equation}
For $D_{3,n}$ given in \eqref{wen4.24}, denote $D_{3,n}=a^T(M_{n}+K_{n})a$
 with $a^T=(a_{1,n},\ldots,a_{i-1,n})\in\mathbb R^{i-1}$, where
\begin{align*}
    &( M_{n})_{jl} =  p_n \int_{\Omega}\int_{\Omega} \frac{u_{p_n}^{p_n-1}(x)
 u_{p_n}^{p_n+1}(y) \chi_n^2(x)   \frac{\partial u_{p_n}}
 {\partial x_j}(x)   \frac{\partial u_{p_n}}
 {\partial x_l}(x) }{|x - y|^\alpha} dx dy, \\
    &( K_{n} )_{jl}= ({p_n+1})\int_\Omega
  \int_\Omega \frac{u_{p_n}^{p_n}(x) \chi_n(x)
 \frac{\partial u_{p_n}}{\partial x_j}(x)u_{p_n}^{p_n}(y) \chi_n(y)
  \frac{\partial u_{p_n}}{\partial y_l}(y) }
 {|x - y|^\alpha} dx  dy.
\end{align*}
By the change of variables
$x=x_n+\varepsilon_n\xi$ and $y=x_n+\varepsilon_n\eta$, we get
\[
\frac{\partial u_{p_n}}{\partial x_j}
(x_n+\varepsilon_n\xi)
=\frac{u_{p_n}(x_n)}{p_n\varepsilon_n}
\frac{\partial v_n}{\partial \xi_j}(\xi).
\]
Set
\[
h_{j,n}(\xi):= \chi_n(x_n + \varepsilon_n\xi)\partial_{\xi_j}v_n(\xi),
\]
extended by zero outside $\Omega_n$. By Lemma \ref{lem:log-growth} and the local convergence of $v_n$,
\[
|h_{j,n}(\xi)| \le \frac{C}{1+|\xi|},\qquad h_{j,n} \to \partial_{\xi_j}U \quad \text{locally uniformly in }\mathbb R^2.
\]
Since $h_{j,n}$ is supported in $B_{r/\varepsilon_n}(0)$, we apply Lemma \ref{RE2.9} with $D_n=\mathbb R^2$: the first limit with $g_n = h_{j,n}h_{l,n}$ and the second with $f_n = h_{j,n}$ and $g_n = h_{l,n}$ yield
\[
(M_n)_{jl} = \frac{u_{p_n}^2(x_n)}{p_n^2\varepsilon_n^2}
\left(
\int_{\mathbb R^2}\int_{\mathbb R^2}
\frac{e^{U(\xi)+U(\eta)}\partial_{\xi_j}U(\xi)\partial_{\xi_l}U(\xi)}{|\xi-\eta|^\alpha}
\,d\eta\,d\xi + o(1)
\right),
\]
\[
(K_n)_{jl} = \frac{u_{p_n}^2(x_n)}{p_n^2\varepsilon_n^2}
\left(
\int_{\mathbb R^2}\int_{\mathbb R^2}
\frac{e^{U(\xi)+U(\eta)}\partial_{\xi_j}U(\xi)\partial_{\eta_l}U(\eta)}{|\xi-\eta|^\alpha}
\,d\eta\,d\xi + o(1)
\right),
\]
where we have used $(p_n+1)/p_n\to 1$ for the second expansion.
Moreover, by symmetry,
it follows that
\[
(M_n)_{jl} = \frac{u_{p_n}^2(x_n)}{p_n^2 \varepsilon_n^2}(C_{1,\alpha}\delta_{jl}+o(1)),\,
(K_n)_{jl} = \frac{u_{p_n}^2(x_n)}{p_n^2 \varepsilon_n^2}(C_{2,\alpha}\delta_{jl}+o(1)),
\]
where
\begin{equation*}
\delta_{jl}=
\begin{cases}
1,& j=l,\\
0,& j\neq l,
\end{cases}
\end{equation*}
and $C_{1,\alpha}$, $C_{2,\alpha}$ are given by
\begin{align*}
C_{1,\alpha}
&= \int_{\mathbb{R}^2} \left( \int_{\mathbb{R}^2}
\frac{e^{U(\eta)}}{|\xi - \eta|^\alpha} d\eta \right) e^{U(\xi)}
\left( \frac{\partial U}{\partial \xi_1}(\xi) \right)^2 d\xi\\
&= \int_{\mathbb{R}^2}
\frac{2(4 - \alpha)C_\alpha^2}{(C_\alpha^2 + |\xi|^2)^2}
(4 - \alpha)^2 \frac{\xi_1^2}{(C_\alpha^2 + |\xi|^2)^2} d\xi
=\frac{\pi(4 - \alpha)^3}{6C_\alpha^2},
\end{align*}
\begin{equation*}
C_{2,\alpha} = \int_{\mathbb{R}^2} \int_{\mathbb{R}^2} \frac{e^{U(\xi)}
 e^{U(\eta)}\frac{\partial U}{\partial \xi_1}(\xi)
 \frac{\partial U}{\partial \eta_1}(\eta)}{|\xi - \eta|^\alpha}  d\xi d\eta=
 \frac{\alpha}{4 - \alpha} C_{1,\alpha}
 =\frac{\pi\alpha(4 - \alpha)^2}{6C_\alpha^2}.
\end{equation*}
Hence, we have
\begin{equation}\label{mn}
M_n = \frac{u_{p_n}^2(x_n)}{p_n^2 \varepsilon_n^2}(C_{1,\alpha}I_{i-1}+o(1)),
\end{equation}
\begin{equation}\label{kn}
K_n = \frac{u_{p_n}^2(x_n)}{p_n^2 \varepsilon_n^2}(C_{2,\alpha}I_{i-1}+o(1)),
\end{equation}
where $I_{i-1}$ denotes the identity matrix of order $i-1.$
Substituting \eqref{mn} and \eqref{kn} into the definition of $D_{3,n}$, we find that
\begin{equation}\label{D3}
D_{3,n}=\frac{u_{p_n}^2(x_n)}{p_n^2 \varepsilon_n^2}
\bigl(C_{1,\alpha}+C_{2,\alpha}+o(1)\bigr)
\sum_{j=1}^{i-1}a_{j,n}^2.
\end{equation}
Finally, combining \eqref{wen4.21}, \eqref{wen4.22}, \eqref{wen4.29} and \eqref{D3}, we  derive
\begin{align*}
D(a_{0,n},\dots,a_{i-1,n})=(2p_n+1) a_{0,n}^2  &\int_{\Omega}\int_{\Omega}
\frac{u_{p_n}^{p_n+1}(x) u_{p_n}^{p_n+1}(y)}{|x - y|^\alpha}  dx dy+o\left(\frac{1}{p_n^2}\right)\\
&+\frac{u_{p_n}^2(x_n)}{p_n^2 \varepsilon_n^2}
\bigl(C_{1,\alpha}+C_{2,\alpha}+o(1)\bigr)
\sum_{j=1}^{i-1}a_{j,n}^2.
\end{align*}
The same integral occurs in $N_{2,n}$. More precisely, \eqref{wen4.19} and
\eqref{wen4.26} give the exact relation
\[
N_{2,n}=-\frac{4p_na_{0,n}}{2p_n+1}(\widehat H_1+\widehat H_2).
\]
Consequently, \eqref{wen4.27} and the estimate for $\widehat H_2$ imply
\begin{equation}\label{wen4.32}
N_{2,n} = o\left(\frac{1}{p_n^2}\right).
\end{equation}
For $N_{3,n}^{(1)}$, the support of $\nabla\chi_n$ stays a fixed
positive distance from $x_\infty$. Thus \eqref{eq:2.8} gives
\begin{align}\label{wen4.33}
N_{3,n}^{(1)}
=\frac{1}{p_n^2}\int_\Omega |\nabla\chi_n|^2
\left(\sum_{j=1}^{i-1}a_{j,n}\partial_{x_j}(p_nu_{p_n})\right)^2dx
\leq\frac{C}{p_n^2}\sum_{j=1}^{i-1}a_{j,n}^2.
\end{align}
Here $C$ is independent of the coefficients and of $n$.
For $N_{3,n}^{(2)}$, by symmetry and integration by parts,
\begin{align*}
N_{3,n}^{(2)} &= \frac{p_n+1}{2} \int_{\Omega} \int_{\Omega} \frac{(\chi_n(x)-\chi_n(y))^2
 u_{p_n}^{p_n}(x) \left(\sum_{j=1}^{i-1} a_{j,n} \frac{\partial u_{p_n}}{\partial x_j}(x)\right)
 u_{p_n}^{p_n}(y) \left(\sum_{l=1}^{i-1} a_{l,n} \frac{\partial u_{p_n}}{\partial y_l}(y)\right)}
 {|x-y|^\alpha} dx dy \\
&= \frac{1}{2(p_n+1)} \int_{\Omega} \int_{\Omega}
 \left( \sum_{j,l=1}^{i-1} a_{j,n} a_{l,n}
\frac{\partial^2}{\partial x_j \partial y_l} \left( \frac{(\chi_n(x)-\chi_n(y))^2}
{|x-y|^\alpha} \right) \right) u_{p_n}^{p_n+1}(x) u_{p_n}^{p_n+1}(y) dx dy\\
&=: \frac{1}{2(p_n+1)}\int_\Omega\int_\Omega
\mathcal C_n(x,y)u_{p_n}^{p_n+1}(x)u_{p_n}^{p_n+1}(y)\,dxdy.
\end{align*}
Here $\mathcal C_n$ is the second directional derivative of the cutoff
commutator kernel.  More precisely, after replacing $|x-y|^{-\alpha}$ by
$(|x-y|^2+\delta^2)^{-\alpha/2}$, differentiating, and then letting
$\delta\to 0^+$, the Lipschitz bounds for $\chi_n$ give, uniformly for
$\sum_{j=1}^{i-1}a_{j,n}^2\le1$,
\begin{equation}\label{eq:commutator-second}
|\mathcal C_n(x,y)|\le C|x-y|^{-\alpha}.
\end{equation}
  In addition,
$\mathcal C_n(x,y)=0$ on
$B_{r/2}(x_n)\times B_{r/2}(x_n)$.  Consequently, \eqref{wenC4.5} and
\eqref{eq:commutator-second} yield
\begin{equation}\label{wen4.35}
|N_{3,n}^{(2)}|
\le \frac{C}{p_n+1}\int_{\Omega\setminus B_{r/2}(x_n)}\int_\Omega
\frac{u_{p_n}^{p_n+1}(x)u_{p_n}^{p_n+1}(y)}{|x-y|^\alpha}\,dydx
=o\left(\frac{1}{p_n^2}\right).
\end{equation}
Therefore, combining \eqref{wen4.20}, \eqref{wen4.33} and \eqref{wen4.35}, we find that for
 any \(\alpha \in (0, 1)\),
\begin{equation*}
N_{3,n}= N_{3,n}^{(1)} + N_{3,n}^{(2)} =  O\left(\frac{1}{p_n^2} \right).
\end{equation*}
All estimates obtained in this step are uniform with respect to $(a_{0,n},\ldots,a_{i-1,n})$ on the unit sphere $\sum_{j=0}^{i-1}a_{j,n}^2=1$.
Based on this result, and further combining \eqref{wen4.17}, \eqref{wen4.18} and \eqref{wen4.32}, we obtain
\begin{align*}
N(a_{0,n},\dots,a_{i-1,n})&=N_{1,n}+N_{2,n}+N_{3,n} \\
&=-2p_n a_{0,n}^2  \int_{\Omega}\int_{\Omega}
\frac{u_{p_n}^{p_n+1}(x) u_{p_n}^{p_n+1}(y)}{|x - y|^\alpha} \, dx dy+
o\left(\frac{1}{p_n^2} \right)+O\left(\frac{1}{p_n^2} \right).
\end{align*}

{\bf Step 2.} We complete the proof of \eqref{jm3.1}.
Let
\[
I_n:=\int_\Omega\int_\Omega
\frac{u_{p_n}^{p_n+1}(x)u_{p_n}^{p_n+1}(y)}{|x-y|^\alpha}dxdy.
\]
The estimates above give, uniformly for $\sum_{j=0}^{i-1}a_{j,n}^2=1$,
\begin{equation}\label{upperD}
D(a_{0,n},\ldots,a_{i-1,n})
=(2p_n+1)a_{0,n}^2 I_n
+\frac{u_{p_n}^2(x_n)}{p_n^2\varepsilon_n^2}
\bigl(C_{1,\alpha}+C_{2,\alpha}+o(1)\bigr)
\sum_{j=1}^{i-1}a_{j,n}^2+o\left(\frac{1}{p_n^2}\right),
\end{equation}
and
\begin{equation}\label{upperN}
N(a_{0,n},\ldots,a_{i-1,n})
=-2p_na_{0,n}^2 I_n+O\left(\frac{1}{p_n^2}\right).
\end{equation}
Since $p_nI_n$ is bounded away from zero and infinity and $u_{p_n}(x_n)\to\sqrt e$, it follows from \eqref{upperD} that
\begin{equation}\label{lowerD}
D(a_{0,n},\ldots,a_{i-1,n})\geq c a_{0,n}^2+\frac{c}{p_n^2\varepsilon_n^2}
\sum_{j=1}^{i-1}a_{j,n}^2
\end{equation}
for some $c>0$ and all sufficiently large $n$. We split the unit sphere into two cases.
If $|a_{0,n}|\geq M/p_n$, with $M$ fixed and sufficiently large, then the negative
term in \eqref{upperN} dominates the remainder. Thus
\[
N(a_{0,n},\ldots,a_{i-1,n})
\leq C\varepsilon_n^2D(a_{0,n},\ldots,a_{i-1,n}).
\]
If $|a_{0,n}|<M/p_n$, then
$\sum_{j=1}^{i-1}a_{j,n}^2\geq1/2$ for large $n$. Equations \eqref{lowerD}
and \eqref{upperN} give, respectively,
\[
D(a_{0,n},\ldots,a_{i-1,n})\geq \frac{c}{p_n^2\varepsilon_n^2},
\qquad
N(a_{0,n},\ldots,a_{i-1,n})=O(p_n^{-2}).
\]
Thus in this case also
\[
\frac{N(a_{0,n},\ldots,a_{i-1,n})}{D(a_{0,n},\ldots,a_{i-1,n})}
\leq C\varepsilon_n^2.
\]
Combining the two cases with \eqref{wen4.16}, we obtain
\[
\lambda_{i,n}\leq 1+C\varepsilon_n^2,
\]
which proves \eqref{jm3.1}.

{\bf Step 3.}
By Proposition \ref{prop:least-energy-morse} and the upper bound above,
\[
1\leq\lambda_{2,n}\leq\lambda_{3,n}
\leq1+C\varepsilon_n^2.
\]
Therefore,
\[
\lambda_{2,n}\to1,\qquad\lambda_{3,n}\to1.
\]
\end{proof}
\begin{proposition}\label{jmp3.2}
For $i=2,3$, let $\widetilde{\phi}_{i,n}$ be the rescaled eigenfunctions
defined in \eqref{JM2.42}. Then
\begin{equation}\label{jm3.27}
\widetilde{\phi}_{i,n}(y) = \sum_{j=1}^2
\frac{a_j^i y_j}{C_\alpha^2 + |y|^2} + o(1)
\quad \mathrm{in }\,\, C^1_{\mathrm{loc}}(\mathbb{R}^2) \, \mathrm{ as } \,\, n \to +\infty,
\end{equation}
where $a^i=(a_1^i,a_2^i)\in\mathbb R^2\setminus\{0\}$, and $a^2$ is
orthogonal to $a^3$ in $\mathbb{R}^2$.
\end{proposition}
\begin{proof}
Using \eqref{jm3.2} and Lemma \ref{jmle2.13},  there exist
 $a_1^i,a_2^i,b^i\in\mathbb{R}$, $(a_1^i,a_2^i,b^i)\not=(0,0,0),$ such that
\begin{equation*}
\widetilde\phi_{i,n}(y)
= \sum_{j=1}^2  \frac{a_j^i y_j}{C_\alpha^2+|y|^2}
+ b^i \frac{C_\alpha^2-|y|^2}{C_\alpha^2+|y|^2}+ o(1) \quad
\mathrm{in}\, C^1_{\mathrm{loc}}(\mathbb{R}^2) \, \mathrm{as } \, n \to +\infty.
\end{equation*}
Suppose by contradiction that $b^i \not= 0$.
Since $0<\alpha<1$, one has
$\frac32(4-\alpha)>3$. Thus \eqref{jm2.58} gives, for all sufficiently
large $n$,
\[
\lambda_{i,n} \ge 1 + \frac{3}{p_n},
\]
which contradicts \eqref{jm3.1}, since Remark \ref{remark:eps-small} gives $\varepsilon_n^2=o(p_n^{-1})$.
Hence, $b^i = 0$ and \eqref{jm3.27} holds.

To establish the orthogonality of the vectors $a^2$ and $a^3$,
we recall the assumption $\int_{\Omega} \nabla \phi_{2,n} \cdot \nabla \phi_{3,n} \,dx = 0,$
which implies
\begin{align}\label{jm3.28}
0& = p_n \displaystyle\int_\Omega u_{p_n}^{p_n-1} \phi_{2,n}\phi_{3,n} \left( \int_\Omega \frac{u_{p_n}^{p_n+1}(y)}{|x-y|^\alpha} dy \right) dx
+ (p_n+1) \int_{\Omega}\int_{\Omega}\frac{u_{p_n}^{p_n}(x)\phi_{3,n}(x) u_{p_n}^{p_n}(y)
 \phi_{2,n}(y)}{|x-y|^\alpha} dx dy\nonumber\\
&= p_n\int_{B_r(x_n)}   u_{p_n}^{p_n-1} \phi_{2,n}\phi_{3,n} \left( \int_\Omega \frac{u_{p_n}^{p_n+1}(y)}{|x-y|^\alpha} dy \right) dx
\nonumber\\
&\quad \quad + (p_n+1) \int_{B_r(x_n)} \int_{\Omega}\frac{u_{p_n}^{p_n}(x)\phi_{3,n}(x) u_{p_n}^{p_n}(y)
 \phi_{2,n}(y)}{|x-y|^\alpha} dx dy+ o(1)
\end{align}
as $n \to +\infty$, where we have used the fact that
\begin{align*}
&p_n \displaystyle\int_{\Omega \setminus B_r(x_n)} u_{p_n}^{p_n-1} |\phi_{2,n}\phi_{3,n}|
 \left( \int_\Omega \frac{u_{p_n}^{p_n+1}(y)}{|x-y|^\alpha} dy \right) dx\nonumber\\
&\quad +(p_n+1) \int_{\Omega \setminus B_r(x_n)}\int_{\Omega}\frac{u_{p_n}^{p_n}(x)|\phi_{3,n}(x)|
u_{p_n}^{p_n}(y)
|\phi_{2,n}(y)|}{|x-y|^\alpha}dxdy\\
\leq& C p_n \Bigg\|\left( \int_{\Omega}\frac{  u_{p_n}^{p_n+1}(y)}{|x-y|^\alpha}
dy\right) u_{p_n}^{p_n-1}(x) \Bigg\|_{L^\infty(\Omega \setminus B_{\frac{r}{2}}
(x_\infty))} \nonumber\\
&\quad+C(p_n+1)  \Bigg\|\left( \int_{\Omega}\frac{  u_{p_n}^{p_n}(y)}{|x-y|^\alpha}
dy\right) u_{p_n}^{p_n}(x) \Bigg\|_{L^\infty(\Omega \setminus B_{\frac{r}{2}}
(x_\infty))}\\
\overset{\substack{\eqref{eq223}\\\eqref{eq224}}}{=}& o(1)\quad  \mathrm{as} \,\, n \to +\infty.
\end{align*}
From \eqref{jm3.28}, it follows that
\begin{align*}
&\int_{B_{\frac{r}{\varepsilon_n}}(0)}\left(1 + \frac{v_n(\xi)}{p_n}\right)^{p_n-1}
\widetilde{\phi}_{2,n}\widetilde{\phi}_{3,n}
\left(\int_{\Omega_n}\frac{\left(1 + \frac{v_n(\eta)}{p_n}\right)^{p_n+1}}
{{|\xi-\eta|}^{\alpha}} d\eta\right) d\xi\\
 +&\int_{B_{\frac{r}{\varepsilon_n}}(0)}\left(1 + \frac{v_n(\xi)}{p_n}\right)^{p_n}
 \widetilde{\phi}_{3,n}
 \left(\int_{\Omega_n}
\frac{\left(1 + \frac{v_n(\eta)}{p_n}\right)^{p_n}}
{{|\xi-\eta|}^{\alpha}}\widetilde{\phi}_{2,n}(\eta) d\eta \right)d\xi+o(1)=0
\end{align*}
as $n \to +\infty$. Using \eqref{eq:2.16}, \eqref{jm3.27} and Lemma \ref{RE2.9},
passing to the limit, we obtain
\begin{align*}
&\sum_{h,j=1}^2 a_h^2 a_j^3 \int_{\mathbb{R}^2} e^{U(\xi)}
\frac{\xi_h}{C_\alpha^2 + |\xi|^2}
\frac{\xi_j}{C_\alpha^2 + |\xi|^2}
\left(\int_{\mathbb{R}^2}\frac{ e^{U(\eta)}}{|\xi-\eta|^{\alpha}}d\eta\right)d\xi \\
&\quad +\sum_{h,j=1}^2 a_h^2 a_j^3 \int_{\mathbb{R}^2} e^{U(\xi)}
\frac{\xi_j}{C_\alpha^2 + |\xi|^2}
\left(\int_{\mathbb{R}^2}\frac{ e^{U(\eta)}\frac{\eta_h}
{C_\alpha^2 + |\eta|^2}}{|\xi-\eta|^{\alpha}}d\eta\right)d\xi= 0.
\end{align*}
Put $Z_j(\xi)=\xi_j/(C_\alpha^2+|\xi|^2)$. Since $Z_j$ solves
the limiting linearized equation, integration by parts and direct
integration give
\[
\int_{\mathbb R^2}Z_h(-\Delta Z_j)
=\int_{\mathbb R^2}\nabla Z_h\cdot\nabla Z_j
=\frac{2\pi}{3C_\alpha^2}\delta_{hj}.
\]
The boundary term at infinity vanishes because $Z_j=O(|\xi|^{-1})$
and $\nabla Z_j=O(|\xi|^{-2})$. Therefore the limiting orthogonality
identity is
\[
0=\frac{2\pi}{3C_\alpha^2}\sum_{h=1}^2a_h^2a_h^3,
\]
and $a^2\cdot a^3=0$.
\end{proof}
\subsection{Outer expansion of the translation modes}
\begin{proof}[Proof of the outer expansion \eqref{jm1.7}]
Let $(\tau_n)$ be a sequence satisfying
\begin{equation}\label{jm3.29}
0 < \tau_n = o(1) \quad \text{and} \quad \frac{\varepsilon_n}{\tau_n^7} = o(1) \quad \text{as } n \to +\infty.
\end{equation}
For example, one may take $\tau_n=\varepsilon_n^{1/8}$.
For the fixed radius chosen after Theorem \ref{TA1}, we have
\begin{equation}\label{jm3.30}
\frac{r}{\varepsilon_n} \ge \frac{\tau_n}{\varepsilon_n} \to +\infty \quad \text{as }
 n \to +\infty.
\end{equation}

{\bf Step 1.} We show that
\begin{equation}\label{jm3.31}
\phi_{i,n} = E_n + F_n + o(\varepsilon_n) \quad \mathrm{in} \,\,
 C^1_{\mathrm{loc}}(\overline{\Omega} \setminus \{x_\infty\})\,\,  \mathrm{as} \,\, n \to +\infty,
\end{equation}
where
\begin{align}\label{jm3.32}
 E_n(x)&:= \lambda_{i,n} G(x,x_n)\left( \int_{B_{\tau_n}(x_n)}  (p_n+1) u_{p_n}^{p_n}(y)
  \left( \int_\Omega \frac{u_{p_n}^{p_n}(z)
 \phi_{i,n}(z)}{|z-y|^\alpha} dz \right) dy \right. \nonumber\\
&\quad \left. + \int_{B_{\tau_n}(x_n)}  p_n  u_{p_n}^{p_n-1}(y) \phi_{i,n}(y) \left( \int_\Omega
\frac{u_{p_n}^{p_n+1}(z)}{|z-y|^\alpha} dz \right) dy \right)
\end{align}
and
\begin{align}\label{jm3.33}
F_n(x)&:= \lambda_{i,n} \sum_{j=1}^2 \frac{\partial G}{\partial y_j}(x, x_n)
\left( \int_{B_{\tau_n}(x_n)}  (p_n+1) u_{p_n}^{p_n}(y) \left( \int_\Omega \frac{u_{p_n}^{p_n}(z)
 \phi_{i,n}(z)}{|z-y|^\alpha} dz \right)(y-x_n)_j dy \right. \nonumber\\
&\quad \left. + \int_{B_{\tau_n}(x_n)}  p_n  u_{p_n}^{p_n-1}(y) \phi_{i,n}(y) \left( \int_\Omega
\frac{u_{p_n}^{p_n+1}(z)}{|z-y|^\alpha} dz \right)(y-x_n)_j  dy \right).
\end{align}
Fix $\delta \in (0, r)$ and take $x \in \Omega \setminus B_{2\delta}(x_\infty)$.
By the Green's representation formula, we obtain
\begin{align}\label{jm3.34}
 \phi_{i,n}(x)&= \lambda_{i,n} \left( \int_{\Omega} G(x,y) (p_n+1) u_{p_n}^{p_n}(y)
 \left( \int_\Omega \frac{u_{p_n}^{p_n}(z)
 \phi_{i,n}(z)}{|z-y|^\alpha} dz \right) dy \right. \nonumber\\
&\quad \left. + \int_{\Omega} G(x,y)  p_n  u_{p_n}^{p_n-1}(y) \phi_{i,n}(y) \left( \int_\Omega
\frac{u_{p_n}^{p_n+1}(z)}{|z-y|^\alpha} dz \right) dy \right)\nonumber\\
&= \lambda_{i,n} \left( \int_{\Omega \setminus B_\delta(x_n)} G(x,y) (p_n+1) u_{p_n}^{p_n}(y)
 \left( \int_\Omega \frac{u_{p_n}^{p_n}(z)
 \phi_{i,n}(z)}{|z-y|^\alpha} dz \right) dy \right. \nonumber\\
&\quad \left. + \int_{\Omega \setminus B_\delta(x_n)} G(x,y)  p_n  u_{p_n}^{p_n-1}(y)
 \phi_{i,n}(y) \left( \int_\Omega
\frac{u_{p_n}^{p_n+1}(z)}{|z-y|^\alpha} dz \right) dy \right)\nonumber\\
&\quad + \lambda_{i,n} \left( \int_{B_\delta(x_n) \setminus B_{\tau_n}(x_n)} G(x,y)
 (p_n+1) u_{p_n}^{p_n}(y) \left( \int_\Omega \frac{u_{p_n}^{p_n}(z)
 \phi_{i,n}(z)}{|z-y|^\alpha} dz \right) dy \right. \nonumber\\
&\quad \left. + \int_{B_\delta(x_n) \setminus B_{\tau_n}(x_n)} G(x,y)
 p_n  u_{p_n}^{p_n-1}(y) \phi_{i,n}(y) \left( \int_\Omega
\frac{u_{p_n}^{p_n+1}(z)}{|z-y|^\alpha} dz \right) dy \right)\nonumber\\
&\quad +\lambda_{i,n} \left( \int_{B_{\tau_n}(x_n)} G(x,y) (p_n+1) u_{p_n}^{p_n}(y)
 \left( \int_\Omega \frac{u_{p_n}^{p_n}(z)
 \phi_{i,n}(z)}{|z-y|^\alpha} dz \right) dy \right. \nonumber\\
&\quad \left. + \int_{B_{\tau_n}(x_n)} G(x,y)  p_n  u_{p_n}^{p_n-1}(y)
\phi_{i,n}(y) \left( \int_\Omega
\frac{u_{p_n}^{p_n+1}(z)}{|z-y|^\alpha} dz \right) dy \right)\nonumber\\
&=\widehat A_n(x) +\widehat B_n(x) +\widehat C_n(x).
\end{align}
In the following, we estimate the three terms.
By Proposition \ref{jmp3.1} and the definition of eigenfunctions, we know that
  the eigenvalue $\lambda_{i,n}$ is bounded in $\mathbb{R}$, and the eigenfunction $\phi_{i,n}$ is
uniformly bounded in $\Omega$. Thus, by \eqref{eq:2.14}, Lemma
\ref{lem:offpeak} with $m=1$ and $\gamma_n=1$, and \eqref{eq:2.68}, we get
\begin{align}\label{jm3.35}
|\widehat A_n(x)|&= \Biggl|\lambda_{i,n} \left( \int_{\Omega \setminus B_\delta(x_n)} G(x,y) (p_n+1)
 u_{p_n}^{p_n}(y) \left( \int_\Omega \frac{u_{p_n}^{p_n}(z)
 \phi_{i,n}(z)}{|z-y|^\alpha} dz \right) dy \right. \nonumber\\
&\quad \left. + \int_{\Omega \setminus B_\delta(x_n)} G(x,y)  p_n  u_{p_n}^{p_n-1}(y)
 \phi_{i,n}(y) \left( \int_\Omega
\frac{u_{p_n}^{p_n+1}(z)}{|z-y|^\alpha} dz \right) dy \right)\Biggl|\nonumber\\
&\le  C \left( \int_{\Omega \setminus B_\delta(x_n)}|G(x,y)|
(p_n+1) u_{p_n}^{p_n}(y) \left( \int_\Omega \frac{u_{p_n}^{p_n}(z)
 }{|z-y|^\alpha} dz \right) dy \right. \nonumber\\
&\quad \left. + \int_{\Omega \setminus B_\delta(x_n)} |G(x,y)|
 p_n  u_{p_n}^{p_n-1}(y)  \left( \int_\Omega
\frac{u_{p_n}^{p_n+1}(z)}{|z-y|^\alpha} dz \right) dy \right)\nonumber\\
&\le o(\varepsilon_n)\sup_{x \in \overline{\Omega} \setminus B_{2\delta}(x_\infty)}
 \|G(x, \cdot)\|_{L^1(\Omega)}
=o(\varepsilon_n).
\end{align}
For the term $\widehat B_n$, we apply Lemma
\ref{lem:shifted-potentials} with $\gamma=1/2$ to obtain
\begin{align}\label{jm3.36}
|\widehat B_n(x)|&= \Biggl|\lambda_{i,n} \left( \int_{B_\delta(x_n) \setminus B_{\tau_n}(x_n)} G(x,y)
 (p_n+1) u_{p_n}^{p_n}(y) \left( \int_\Omega \frac{u_{p_n}^{p_n}(z)
 \phi_{i,n}(z)}{|z-y|^\alpha} dz \right) dy \right. \nonumber\\
&\quad \left. + \int_{B_\delta(x_n) \setminus B_{\tau_n}(x_n)} G(x,y)  p_n  u_{p_n}^{p_n-1}(y) \phi_{i,n}(y) \left( \int_\Omega
\frac{u_{p_n}^{p_n+1}(z)}{|z-y|^\alpha} dz \right) dy \right)\Biggl|\nonumber\\
&\le C\int_{B_{\frac{\delta}{\varepsilon_n}}(0) \setminus B_{\frac{\tau_n}{\varepsilon_n}}(0)}
|G(x, x_n + \varepsilon_n \eta)|\left(1+\frac{v_n(\eta)}{p_n}\right)^{p_n}
\left( \int_{\Omega_n} \frac{\left(1+\frac{v_n(w)}
{p_n}\right)^{p_n}}{|\eta-w|^\alpha}dw \right) d\eta \nonumber \\
&\quad+ C\int_{B_{\frac{\delta}{\varepsilon_n}}(0)
 \setminus B_{\frac{\tau_n}{\varepsilon_n}}(0)}
|G(x, x_n + \varepsilon_n \eta)|\left(1+\frac{v_n(\eta)}{p_n}\right)^{p_n-1}
 \left( \int_{\Omega_n} \frac{\left(1+\frac{v_n(w)}{p_n}\right)^{p_n+1} }
 {|\eta-w|^\alpha}dw \right)d\eta\nonumber \\
&\le C\int_{B_{\frac{\delta}{\varepsilon_n}}(0) \setminus B_{\frac{\tau_n}{\varepsilon_n}}(0)}
\frac{|G(x, x_n + \varepsilon_n \eta)|}{1+|\eta|^{\frac{7}{2}-\alpha}}\cdot\frac{1}{(1+|\eta|^2)^
{\frac{\alpha}{2}}}d\eta\nonumber \\
&\leq \varepsilon_n \sqrt{\frac{\varepsilon_n}{\tau_n^7}} C \int_\Omega |G(x, y)| \, dy \nonumber\\
&\leq \varepsilon_n \sqrt{\frac{\varepsilon_n}{\tau_n^7}} C \sup_{x \in \overline{\Omega}
 \setminus B_{2\delta}(x_\infty)} \|G(x, \cdot)\|_{L^1(\Omega)}\nonumber \\
&\stackrel{\eqref{eq:2.68}}{\leq} \varepsilon_n \sqrt{\frac{\varepsilon_n}{\tau_n^7}}
C \stackrel{\eqref{jm3.29}}{=} o(\varepsilon_n),
\end{align}
in  $C^0(\overline{\Omega} \setminus B_{2\delta}(x_\infty)),$  as $ n \to +\infty.$
Since $\tau_n=o(1)$ and $x_n\to x_\infty$, we have
\begin{equation}\label{jm3.37}
B_{\tau_n}(x_n) \subseteq B_\delta(x_\infty) \quad \mathrm{for }\,\, n\,\, \mathrm{ large}.
\end{equation}
Hence, for every $x\in\Omega\setminus B_{2\delta}(x_\infty)$, the map
$y\mapsto G(x,y)$ is smooth in $B_{\tau_n}(x_n)$ for all large $n$.
Taylor's formula gives
\begin{equation}\label{jm3.38}
G(x, y) = G(x, x_n) + \sum_{j=1}^2 \frac{\partial G}{\partial y_j}(x, x_n) (y - x_n)_j
 + \frac{1}{2} \sum_{j,k=1}^2 \frac{\partial^2 G}{\partial y_j
 \partial y_k}(x, \eta_n) (y - x_n)_j (y - x_n)_k,
\end{equation}
where $x \in \Omega \setminus B_{2\delta}(x_\infty)$, $y \in B_{\tau_n}(x_n)$ and $\eta_n$
denotes a point on the line between $y$ and $x_n$. Moreover, we also have
\begin{equation}\label{jm3.39}
\sup_{x \in \overline{\Omega} \setminus B_{2\delta}(x_\infty)} \left\| \frac{\partial^2 G}{\partial y_j \partial y_k}(x, \cdot)
\right\|_{L^\infty(B_\delta(x_\infty))} < +\infty.
\end{equation}
From \eqref{jm3.38}, we can decompose
\begin{align}\label{jm3.40}
\widehat C_n(x)&=\lambda_{i,n} \left( \int_{B_{\tau_n}(x_n)} G(x,y) (p_n+1) u_{p_n}^{p_n}(y)
 \left( \int_\Omega \frac{u_{p_n}^{p_n}(z)
 \phi_{i,n}(z)}{|z-y|^\alpha} dz \right) dy \right. \nonumber\\
&\quad \left. + \int_{B_{\tau_n}(x_n)} G(x,y)  p_n  u_{p_n}^{p_n-1}(y)
\phi_{i,n}(y) \left( \int_\Omega
\frac{u_{p_n}^{p_n+1}(z)}{|z-y|^\alpha} dz \right) dy \right)\nonumber\\
&:=\widehat D_n(x) +E_n(x) + F_n(x),
\end{align}
where $E_n$ and $F_n$ are defined in \eqref{jm3.32} and \eqref{jm3.33},
respectively, and
\begin{align*}
\widehat D_n(x)&:= \frac{\lambda_{i,n}}{2} \left(\int_{B_{\tau_n}(x_n)} \sum_{j,k=1}^2 \frac{\partial^2 G}
{\partial y_j \partial y_k}(x, \eta_n) (y - x_n)_j (y - x_n)_k(p_n+1) u_{p_n}^{p_n}(y)
 \left( \int_\Omega \frac{u_{p_n}^{p_n}(z)
 \phi_{i,n}(z)}{|z-y|^\alpha} dz \right) dy
\right.\nonumber\\
&\quad\left.+\int_{B_{\tau_n}(x_n)} \sum_{j,k=1}^2 \frac{\partial^2 G}
{\partial y_j \partial y_k}(x, \eta_n) (y - x_n)_j (y - x_n)_k p_n  u_{p_n}^{p_n-1}(y)
\phi_{i,n}(y) \left( \int_\Omega
\frac{u_{p_n}^{p_n+1}(z)}{|z-y|^\alpha} dz \right) dy\right).
\end{align*}
Next we prove that $\widehat D_n = o(\varepsilon_n)$ in
$C^0(\Omega \setminus B_{2\delta}(x_\infty))$
 for $n \to +\infty$.
Since $\eta_n\in B_{\tau_n}(x_n)$, equations \eqref{jm3.37} and
\eqref{jm3.39} imply that
\begin{equation*}
\left| \frac{\partial^2 G}{\partial y_j \partial y_k}(x, \eta_n) \right| \leq C,
\end{equation*}
and this together with \eqref{jm3.30} yields
\begin{align}\label{jm3.41}
|\widehat D_n(x)|&\le C \left(\int_{B_{\tau_n}(x_n)} (p_n+1) u_{p_n}^{p_n}(y)
 \left( \int_\Omega \frac{u_{p_n}^{p_n}(z)
 }{|z-y|^\alpha} dz \right)|y-x_n|^2 dy
\right.\nonumber\\
&\quad\left.+\int_{B_{\tau_n}(x_n)}  p_n  u_{p_n}^{p_n-1}(y)
 \left( \int_\Omega
\frac{u_{p_n}^{p_n+1}(z)}{|z-y|^\alpha} dz \right)|y-x_n|^2 dy\right)\nonumber\\
&\le C\tau_n \left(\int_{B_{\tau_n}(x_n)} (p_n+1) u_{p_n}^{p_n}(y)
 \left( \int_\Omega \frac{u_{p_n}^{p_n}(z)
 }{|z-y|^\alpha} dz \right)|y-x_n|dy
\right.\nonumber\\
&\quad\left.+\int_{B_{\tau_n}(x_n)}  p_n  u_{p_n}^{p_n-1}(y)
 \left( \int_\Omega
\frac{u_{p_n}^{p_n+1}(z)}{|z-y|^\alpha} dz \right)|y-x_n| dy\right)\nonumber\\
&\leq C \tau_n \varepsilon_n \left(\int_{B_{\frac{\tau_n}{\varepsilon_n}}(0)}
\left(1 + \frac{v_n(\eta)}{p_n}\right)^{p_n} \left( \int_{\Omega_n}
\frac{\left(1 + \frac{v_n(w)}{p_n}\right)^{p_{n}}}{|\eta- w|^\alpha} dw \right)
 |\eta| d\eta \right.\nonumber\\
&\quad\left. + \int_{B_{\frac{\tau_n}{\varepsilon_n}}(0)} \left(1 + \frac{v_n(\eta)}
{p_n}\right)^{p_n-1} \left( \int_{\Omega_n} \frac{\left(1 + \frac{v_n(w)}{p_n}\right)
^{p_n+1}}{|\eta - w|^\alpha} dw \right) |\eta| d\eta\right) \nonumber\\
&\leq C \tau_n \varepsilon_n \left( \int_{\mathbb{R}^2} \int_{\mathbb{R}^2} \frac{e^{U(\eta)} e^{U(w)}}
{|\eta - w|^\alpha} |\eta| dw d\eta + o(1) \right) = \tau_n \varepsilon_n (C + o(1))
 = o(\varepsilon_n)
\end{align}
in  $C^0(\overline{\Omega} \setminus B_{2\delta}(x_\infty)),$  as  $n \to +\infty.$

Substituting \eqref{jm3.35}-\eqref{jm3.36} and \eqref{jm3.40}-\eqref{jm3.41} into \eqref{jm3.34} gives \eqref{jm3.31} in $C^0_{\mathrm{loc}}(\Omega\setminus\{x_\infty\})$. Differentiating \eqref{jm3.34} in $x$ and using the mixed derivative estimates in the appendix together with
\[
\sup_{x\in\Omega}\int_\Omega|\nabla_xG(x,y)|\,dy<\infty
\]
yields convergence in $C^1_{\mathrm{loc}}(\overline\Omega\setminus\{x_\infty\})$.

{\bf Step 2.} We prove that
\begin{equation}\label{jm3.42}
F_n = \varepsilon_n 2\pi \sum_{j=1}^2 a_j^i \frac{\partial G}{\partial y_j}
(\cdot, x_\infty) + o(\varepsilon_n) \quad \,\,\mathrm{ in }\,\, C^1_{\mathrm{loc}}
(\overline{\Omega} \setminus \{x_\infty\}) \,\text{as } n \to +\infty.
\end{equation}
Letting $\delta \in (0, r)$,
for $x \in \Omega \setminus B_{2\delta}(x_\infty)$, we observe from \eqref{eq:2.12} that $x_n \in B_\delta(x_\infty)$
 for sufficiently large $n$. Using \eqref{jm3.39}, we obtain
\begin{equation*}
\left| \frac{\partial G}{\partial y_j}(x, x_\infty) - \frac{\partial G}{\partial y_j}(x, x_n) \right|
\leq \sup_{x \in \overline{\Omega} \setminus B_{2\delta}(x_\infty)} \left\| \nabla_y\partial_{y_j}G(x, \cdot)
 \right\|_{L^\infty(B_\delta(x_\infty))} \, |x_n - x_\infty| = o(1)
\end{equation*}
uniformly as $n \to +\infty.$
Consequently, for $x \in \Omega \setminus B_{2\delta}(x_\infty)$,
in view of \eqref{jm3.2}, it follows that
\begin{align*}
\frac{F_n(x)}{\varepsilon_n}
&=\lambda_{i,n}\sum_{j=1}^2
\frac{\partial G}{\partial y_j}(x,x_n)I_{j,n},
\end{align*}
where
\begin{align*}
I_{j,n}:={}&\left(1+\frac1{p_n}\right)
\int_{B_{\frac{\tau_n}{\varepsilon_n}}(0)} \left(1+\frac{v_n(\eta)}{p_n}\right)^{p_n}
\left( \int_{\Omega_n} \frac{\left(1+\frac{v_n(w)}{p_n}\right)^{p_n} \widetilde{\phi}_{i,n}(w)}{|\eta-w|^\alpha} dw \right) \eta_j d\eta \\
&+ \int_{B_{\frac{\tau_n}{\varepsilon_n}}(0)} \left(1+\frac{v_n(\eta)}{p_n}\right)^{p_n-1} \widetilde{\phi}_{i,n}(\eta)
\left( \int_{\Omega_n} \frac{\left(1+\frac{v_n(w)}{p_n}\right)^{p_n+1}}{|\eta-w|^\alpha} dw \right) \eta_j d\eta.
\end{align*}
By \eqref{eq:2.16}, \eqref{jm3.27}, \eqref{jm3.30} and Lemma \ref{RE2.9},
\begin{align*}
I_{j,n}
&= \int_{\mathbb{R}^2}e^{U(\eta)}
\Bigg( \int_{\mathbb{R}^2} \frac{e^{U(w)} \sum_{h=1}^2 \frac{a_h^i w_h}
{C_\alpha^2 + |w|^2}}{|\eta - w|^\alpha} dw \Bigg) \eta_j d\eta \\
&\quad + \int_{\mathbb{R}^2} e^{U(\eta)} \sum_{h=1}^2
\frac{a_h^i \eta_h}{C_\alpha^2 + |\eta|^2} \eta_j
\Bigg( \int_{\mathbb{R}^2} \frac{e^{U(w)}}{|\eta - w|^\alpha} dw \Bigg) d\eta + o(1)\\
&=8C_\alpha^2\sum_{h=1}^2 a_h^i
\int_{\mathbb{R}^2}\frac{\eta_h\eta_j}
{(C_\alpha^2+|\eta|^2)^3}d\eta+o(1)
=2\pi a_j^i+o(1).
\end{align*}
Consequently,
\begin{align*}
\frac{F_n(x)}{\varepsilon_n}
=2\pi\sum_{j=1}^2 a_j^i\frac{\partial G}{\partial y_j}(x,x_\infty)+o(1)
\end{align*}
uniformly as $n\to+\infty$.
Here differentiation is with respect to the source variable $y$, as in \eqref{jm3.38}. Applying the same argument to $\nabla_xG$ yields the asserted $C^1_{\mathrm{loc}}$ convergence.

{\bf Step 3.} We claim that
\begin{equation}\label{jm3.43}
E_n = o(\varepsilon_n) \quad \mathrm{in } \,\,C^1_{\mathrm{loc}}(\overline{\Omega}
\setminus \{x_\infty\})\,\, \mathrm{as }\,\, n \to +\infty.
\end{equation}
For fixed $\delta\in(0,r)$ and
$x\in\Omega\setminus B_{2\delta}(x_\infty)$, equations
\eqref{eq:2.12} and \eqref{jm3.2} yield
\begin{align}\label{jm3.44}
 E_n(x)&= \lambda_{i,n} G(x,x_n)\left( \int_{B_{\tau_n}(x_n)}  (p_n+1) u_{p_n}^{p_n}(y)
  \left( \int_\Omega \frac{u_{p_n}^{p_n}(z)
 \phi_{i,n}(z)}{|z-y|^\alpha} dz \right) dy \right. \nonumber\\
&\quad \left. + \int_{B_{\tau_n}(x_n)}  p_n  u_{p_n}^{p_n-1}(y) \phi_{i,n}(y) \left( \int_\Omega
\frac{u_{p_n}^{p_n+1}(z)}{|z-y|^\alpha} dz \right) dy \right)\nonumber\\
&=\gamma_{i,n}\bigl(G(x,x_\infty)+o(1)\bigr)
\end{align}
in $C^1_{\mathrm{loc}}(\overline\Omega\setminus\{x_\infty\})$, with
\begin{align*}
\gamma_{i,n}&:= \left( \int_{B_{\tau_n}(x_n)}  (p_n+1) u_{p_n}^{p_n}(y)
  \left( \int_\Omega \frac{u_{p_n}^{p_n}(z)
 \phi_{i,n}(z)}{|z-y|^\alpha} dz \right) dy \right. \nonumber\\
&\quad \left. + \int_{B_{\tau_n}(x_n)}  p_n  u_{p_n}^{p_n-1}(y) \phi_{i,n}(y) \left( \int_\Omega
\frac{u_{p_n}^{p_n+1}(z)}{|z-y|^\alpha} dz \right) dy \right).\nonumber\\
\end{align*}
By a change of variables, we deduce that
\begin{align*}
\gamma_{i,n}&=\left(1+\frac1{p_n}\right)
\int_{B_{\frac{\tau_n}{\varepsilon_n}}(0)} \left(1+\frac{v_n(\eta)}{p_n}\right)^{\!p_n}
    \left( \int_{\Omega_n} \frac{\left(1+\frac{v_n(w)}{p_n}\right)^{\!p_n} \widetilde{\phi}_{i,n}(w)}{|\eta-w|^\alpha} dw \right) d\eta \\
&\quad + \int_{B_{\frac{\tau_n}{\varepsilon_n}}(0)} \left(1+\frac{v_n(\eta)}{p_n}\right)^{\!p_n-1} \widetilde{\phi}_{i,n}(\eta)
    \left( \int_{\Omega_n} \frac{\left(1+\frac{v_n(w)}{p_n}\right)^{\!p_n+1}}
    {|\eta-w|^\alpha} dw \right)  d\eta \nonumber\\
&=  \int_{\mathbb{R}^2}e^{U(\eta)}
\Bigg( \int_{\mathbb{R}^2} \frac{e^{U(w)} \sum_{h=1}^2 \frac{a_h^i w_h}
{C_\alpha^2 + |w|^2}}{|\eta - w|^\alpha} dw \Bigg)  d\eta + \int_{\mathbb{R}^2} e^{U(\eta)} \sum_{h=1}^2
\frac{a_h^i \eta_h}{C_\alpha^2 + |\eta|^2}
\Bigg( \int_{\mathbb{R}^2} \frac{e^{U(w)}}{|\eta - w|^\alpha} dw \Bigg) d\eta + o(1)\nonumber\\
&\stackrel{(\blacktriangle)}= o(1)
\end{align*}
as $n \to +\infty.$ Here, $(\blacktriangle)$ follows from the fact that
double integrals can be reduced to integrals of radial functions against odd functions.
If
\begin{equation}\label{jm3.45}
\gamma_{i,n} = o(\varepsilon_n) \quad \mathrm{as} \, n \to +\infty,
\end{equation}
combining with \eqref{jm3.44}, we obtain $E_n(x) = o(\varepsilon_n)$
in $C^0(\Omega \setminus B_{2\delta}(x_\infty)),$ as
$n \to +\infty$. Furthermore, differentiating \eqref{jm3.44} in $x$ and using
the corresponding uniform bounds for $\nabla_xG(x,x_n)$ proves the same estimate in
$C^1_{\text{loc}}(\overline\Omega \setminus \{x_\infty\})$. Thus \eqref{jm3.43} holds.

In what follows, we prove \eqref{jm3.45}. Suppose by contradiction that \eqref{jm3.45} is false. Then, up to a subsequence, there exists $c_0>0$ such that $|\gamma_{i,n}|\ge c_0\varepsilon_n$. In particular, $\gamma_{i,n}\ne0$ for large $n$, and, after passing to a further subsequence, we may assume that
\[
\theta_n:=\frac{\varepsilon_n}{\gamma_{i,n}}\to \theta\in\mathbb{R}.
\]
From \eqref{jm3.31}, \eqref{jm3.42} and \eqref{jm3.44}, we have
\begin{equation}\label{jm3.46}
\frac{\phi_{i,n}(x)}{\gamma_{i,n}}
= G(x,x_\infty)
+ 2\pi \theta \sum_{j=1}^2 a_j^i \frac{\partial G}{\partial y_j}(x,x_\infty) + o(1)
\quad \mathrm{in}\ C^1_{\mathrm{loc}}(\overline{\Omega}\setminus\{x_\infty\}).
\end{equation}
We now use the Pohozaev identity \eqref{wen2.6} with $z=x_n$:
\begin{align}\label{jm3.52}
\int_{\partial\Omega} \frac{\partial u_{p_n}}{\partial \nu} \frac{\partial \phi_{i,n}}
{\partial \nu} (x-x_n) \cdot \nu \, d\sigma_x
&= (1-\lambda_{i,n}) \Bigg\{ p_n \int_\Omega u_{p_n}^{p_n-1} \phi_{i,n} \omega_{p_n} \left( \int_\Omega \frac{u_{p_n}^{{p_n}+1}(y)}{|x-y|^\alpha} dy \right) dx \nonumber \\
&\quad + (p_n+1) \int_\Omega\int_\Omega \frac{u_{p_n}^{p_n}(x) \phi_{i,n}(x)u_{p_n}^{p_n}(y) \omega_{p_n}(y)}{|x-y|^\alpha} dx dy \Bigg\},
\end{align}
where $\omega_{p_n}(x)= (x-x_n) \cdot \nabla u_{p_n}(x)$.
Using \eqref{jm3.46}, \eqref{eq:2.8}, Lemma \ref{Rgdengshi}, and
$\nabla R(x_\infty)=0$, we obtain
\begin{align}\label{jm3.53}
\int_{\partial\Omega} (x - x_n) \cdot \nabla u_{p_n} \frac{\partial \phi_{i,n}}{\partial\nu} \, d\sigma_x
&=\frac{\gamma_{i,n}}{p_n}\Bigl((4-\alpha)\sqrt e
+2\pi^2(4-\alpha)\sqrt e\,\theta
\,a^i\cdot\nabla R(x_\infty)+o(1)\Bigr)\nonumber\\
&=\frac{\gamma_{i,n}}{p_n}\left((4-\alpha)\sqrt e+o(1)\right).
\end{align}
On the other hand, using \eqref{jm3.27} and the oddness of the
translation modes, the expression on the right-hand side of
\eqref{jm3.52} is $o(p_n^{-1})$.
Proposition \ref{jmp3.1} and the least-energy lower bound give
$|1-\lambda_{i,n}|\leq C\varepsilon_n^2$. Under the contradictory
assumption $|\gamma_{i,n}|\geq c_0\varepsilon_n$, the absolute value
of the right-hand side of \eqref{jm3.52} is therefore
\[
o\left(\frac{\varepsilon_n^2}{p_n}\right)
=o\left(\frac{|\gamma_{i,n}|}{p_n}\right).
\]
This contradicts \eqref{jm3.53}, whose leading coefficient is nonzero.
Thus \eqref{jm3.45} holds. Combining Steps~1--3 proves \eqref{jm1.7}.
\end{proof}

\subsection{Sharp expansion of the translation eigenvalues}

We now use the translation identity \eqref{wen48} to derive the sharp expansion of
$\lambda_{i,n}-1$.
\begin{proof}[Proof of the eigenvalue expansion \eqref{jm1.8}]
Using \eqref{wen48}, we obtain
\begin{align}\label{jm3.47}
\int_{\partial\Omega} p_n\frac{\partial \phi_{i,n}}{\partial \nu_x} \frac{\partial u_{p_n}}{\partial x_j} d\sigma_x
&= (1-\lambda_{i,n}) \Bigg\{%
{p^2_n} \int_\Omega \phi_{i,n} u_{p_n}^{{p_n}-1} \frac{\partial u_{p_n}}{\partial x_j} \left( \int_\Omega \frac{u_{p_n}^{{p_n}+1}(y)}{|x-y|^\alpha} dy \right) dx \nonumber \\
&\quad + p_n({p_n}+1) \int_\Omega u_{p_n}^{{p_n}} \frac{\partial u_{p_n}}{\partial x_j} \left( \int_\Omega \frac{u_{p_n}^{{p_n}}(y) \phi_{i,n}(y)}{|x-y|^\alpha} dy \right) dx \Bigg\}.
\end{align}
By Lemma \ref{lem:offpeak} and \eqref{eq:2.8}, the contribution of
$\Omega\setminus B_r(x_n)$ to the expression  in
\eqref{jm3.47} is $o(1)$. Rescaling the remaining integrals and using
\eqref{J2.36}, \eqref{J2.37}, and Lemma \ref{RE2.9}, we obtain
\begin{align}\label{jm3.49}
& \int_{\Omega } p_n^2\phi_{i,n} u_{p_n}^{{p_n}-1} \frac{\partial u_{p_n}}
{\partial x_j} \left( \int_\Omega \frac{u_{p_n}^{{p_n}+1}(y)}{|x-y|^\alpha} dy \right) dx  + \int_{\Omega} p_n({p_n}+1) u_{p_n}^{{p_n}} \frac{\partial u_{p_n}}{\partial x_j}
\left( \int_\Omega \frac{u_{p_n}^{{p_n}}(y) \phi_{i,n}(y)}{|x-y|^\alpha} dy \right) dx \nonumber \\
=& \frac{1}{\varepsilon_n} \left( \int_{\mathbb{R}^2} \sqrt{e}\,
\frac{\partial U}{\partial \eta_j} \left( -\Delta \left( \sum_{k=1}^2 a_k^i
\frac{\eta_k}{C_\alpha^2+|\eta|^2} \right) \right) d\eta + o(1) \right)\nonumber \\
\stackrel{\eqref{eq:2.17}}{=}& \frac{1}{\varepsilon_n} \left( -\frac{2\pi(4-\alpha)\sqrt{e}}{3C_\alpha^2}\,
a_j^i + o(1) \right).
\end{align}

On the boundary, \eqref{jm1.7}, \eqref{eq:2.8} and \eqref{eq:2.6}
give
\begin{equation*}
\int_{\partial\Omega} p_n \frac{\partial u_{p_n}}{\partial x_j}
\frac{\partial \phi_{i,n}}{\partial\nu} \, d\sigma_x
= 2\pi^2(4-\alpha) \sqrt{e}\, \varepsilon_n \left( \sum_{k=1}^2 a_k^i
\frac{\partial^2 R}{\partial x_j \partial x_k}(x_\infty) + o(1) \right).
\end{equation*}
Combining this identity with \eqref{jm3.47} and \eqref{jm3.49}, we obtain
\begin{equation}\label{jm3.56}
2\pi^2(4-\alpha) \sqrt{e}\, \varepsilon_n \left( \sum_{k=1}^2 a_k^i
\frac{\partial^2 R}{\partial x_j \partial x_k}(x_\infty) + o(1) \right)= (1-\lambda_{i,n})\frac{1}{\varepsilon_n} \left( -\frac{2\pi(4-\alpha)\sqrt{e}}{3C_\alpha^2}\,
a_j^i + o(1) \right).
\end{equation}
By Proposition \ref{prop:least-energy-morse} and \eqref{jm3.1}, the sequence
\[
\rho_{i,n}:=\frac{\lambda_{i,n}-1}
{3\pi C_\alpha^2\varepsilon_n^2}
\]
is bounded. After passing to a further subsequence, let
$\rho_{i,n}\to\zeta_i$. Dividing \eqref{jm3.56} by $\varepsilon_n$ and
using both component equations gives
\[
D^2R(x_\infty)a^i=\zeta_i a^i.
\]
Since $a^i\ne0$, $\zeta_i$ is an eigenvalue of $D^2R(x_\infty)$, with
$a^i$ as a corresponding eigenvector. Moreover,
\begin{equation}\label{jm3.57}
\frac{\lambda_{i,n} - 1}{\varepsilon_n^2} = 3\pi C_\alpha^2 \zeta_i + o(1),
\end{equation}
along the subsequence.

By Proposition \ref{jmp3.2}, $a^2\perp a^3$; hence $\lambda_{2,n}\le\lambda_{3,n}$ and  \eqref{jm3.57} imply
\[
\zeta_2=\mu_1,\qquad \zeta_3=\mu_2.
\]
These limits are independent of the subsequence, so the eigenvalue expansions hold for the full sequence, while the eigenfunction expansions hold along subsequences. Finally, Lemma \ref{Rdayu} gives
\[
\mu_1+\mu_2=\Delta R(x_\infty)>0,
\]
so $\mu_2>0$ and, by \eqref{jm1.8}, $\lambda_{3,n}>1$ for all sufficiently large $n$. This completes the proof.
\end{proof}
\section{The fourth eigenpair and Morse index}\label{se5}

In this section, we derive the asymptotic estimates of the fourth eigenpair and Morse index.
\begin{proof}[Proof of Theorem \ref{th4}]
\textbf{Step 1: the dilation direction and its projections.}
 Set
\[
A_n(x):=\int_\Omega \frac{u_{p_n}^{p_n+1}(y)}{|x-y|^\alpha}dy
\]
and
\[
\Psi_{4,n}(x):=(x-x_n)\cdot\nabla u_{p_n}(x)+\frac{4-\alpha}{2p_n}u_{p_n}(x).
\]
It follows that
\begin{align}\label{eq:Psi4linear}
-\Delta \Psi_{4,n}
&=p_n A_nu_{p_n}^{p_n-1}\Psi_{4,n}
+(p_n+1)u_{p_n}^{p_n}\int_\Omega \frac{u_{p_n}^{p_n}(y)\Psi_{4,n}(y)}{|x-y|^\alpha}dy.
\end{align}
Thus $\Psi_{4,n}$ solves the linearized equation in $\Omega$ with
$\lambda=1$.  Moreover, by \eqref{J2.34} and \eqref{J2.37},
\begin{equation}\label{eq:Psi4bound}
\|\Psi_{4,n}\|_{L^\infty(\Omega)}\le \frac{C}{p_n},\qquad
p_n\Psi_{4,n}(x_n+\varepsilon_n\eta)\to \sqrt e\left(\eta\cdot\nabla U(\eta)+\frac{4-\alpha}{2}\right)
\end{equation}
locally uniformly in $\mathbb{R}^2$. Since
\[
T(\eta):=\eta\cdot\nabla U(\eta)+\frac{4-\alpha}{2}
=\frac{4-\alpha}{2}
\frac{C_\alpha^2-|\eta|^2}{C_\alpha^2+|\eta|^2},
\]
$T$ is the dilation mode of the limiting linearized problem.
 By the min-max characterization,
\begin{equation}\label{eq:l4minmax}
\lambda_{4,n}:=\inf_{\substack{\phi\in H_0^1(\Omega),\,\phi\not\equiv0\\
\phi\perp\{\phi_{1,n},\phi_{2,n},\phi_{3,n}\}}}
\frac{\int_\Omega |\nabla\phi|^2dx}{p_n\int_\Omega u_{p_n}^{p_n-1}\phi^2 A_n dx+(p_n+1)\int_\Omega\int_\Omega\frac{u_{p_n}^{p_n}(x)\phi(x)u_{p_n}^{p_n}(y)\phi(y)}{|x-y|^\alpha}dxdy}.
\end{equation}
Let $\widehat\chi_n$ be the cutoff function defined by
\begin{equation*}
\widehat{\chi}_n(x):=
\begin{cases}
1, & |x-x_n|\le \varepsilon_n,\\
\dfrac{1}{\log(\varepsilon_n/r)}\log\dfrac{|x-x_n|}{r}, & \varepsilon_n<|x-x_n|\le r,\\
0, & |x-x_n|>r,
\end{cases}
\end{equation*}
where $r>0$ is the fixed radius chosen after Theorem \ref{TA1}. Put
$h_n=\widehat\chi_n\Psi_{4,n}$ and define
\[
\phi=h_n+c_{1,n}\phi_{1,n}+c_{2,n}\phi_{2,n}+c_{3,n}\phi_{3,n},
\qquad c_{i,n}:=-\frac{Y_{i,n}}{Z_{i,n}},
\]
where
\begin{align*}
Y_{i,n}:={}&p_n\int_\Omega u_{p_n}^{p_n-1}\phi_{i,n}\widehat\chi_n\Psi_{4,n}A_n dx
+(p_n+1)\int_\Omega\int_\Omega\frac{u_{p_n}^{p_n}(x)\widehat\chi_n(x)\Psi_{4,n}(x)u_{p_n}^{p_n}(y)\phi_{i,n}(y)}{|x-y|^\alpha}dxdy,\\
Z_{i,n}:={}&p_n\int_\Omega u_{p_n}^{p_n-1}\phi_{i,n}^2A_n dx
+(p_n+1)\int_\Omega\int_\Omega\frac{u_{p_n}^{p_n}(x)\phi_{i,n}(x)u_{p_n}^{p_n}(y)\phi_{i,n}(y)}{|x-y|^\alpha}dxdy.
\end{align*}
Then $\phi\perp\{\phi_{1,n},\phi_{2,n},\phi_{3,n}\}$ in $H_0^1(\Omega)$. Indeed,
the eigenvalue equation gives
$\int_\Omega\nabla(\widehat\chi_n\Psi_{4,n})\cdot\nabla\phi_{i,n}
=\lambda_{i,n}Y_{i,n}$ and
$\int_\Omega|\nabla\phi_{i,n}|^2=\lambda_{i,n}Z_{i,n}$.
We next prove that
\begin{equation}\label{eq:Ysmall4}
Y_{i,n}=o(p_n^{-1}),\qquad i=1,2,3.
\end{equation}
In the nonlocal term defining $Y_{i,n}$, the contribution from
$y\in\Omega\setminus B_r(x_n)$ is $o(p_n^{-1})$ by
 \eqref{eq223}, \eqref{eq224}, and \eqref{eq:Psi4bound}. Notice also that
\[
\widehat\chi_n(x_n+\varepsilon_n\xi)\longrightarrow1
\quad\text{for every fixed }\xi\in\mathbb R^2,
\]
and that, by \eqref{eq:Psi4bound},
\[
\frac{p_n}{u_{p_n}(x_n)}
\widehat\chi_n(x_n+\varepsilon_n\xi)
\Psi_{4,n}(x_n+\varepsilon_n\xi)
\]
is uniformly bounded and converges locally uniformly to $T(\xi)$.
Thus Lemma \ref{RE2.9} applies to the following limits.
For $i=2,3$, put
\[
\Phi_i(\xi):=\sum_{j=1}^2
a_j^i\frac{\xi_j}{C_\alpha^2+|\xi|^2}.
\]
Since $\Phi_i$ solves the limiting linearized equation, we obtain
\begin{align*}
p_nY_{i,n}\longrightarrow{}&
\sqrt e\int_{\mathbb R^2}
e^{U(\eta)}T(\eta)\Phi_i(\eta)
\left(\int_{\mathbb R^2}
\frac{e^{U(w)}}{|\eta-w|^\alpha}\,dw\right)d\eta+\sqrt e\int_{\mathbb R^2}
e^{U(\eta)}T(\eta)
\left(\int_{\mathbb R^2}
\frac{e^{U(w)}\Phi_i(w)}
{|\eta-w|^\alpha}\,dw\right)d\eta\\
={}&\sqrt e\int_{\mathbb R^2}
T(-\Delta\Phi_i)\,d\eta=0.
\end{align*}
The last equality also follows directly from parity, since $T$ is radial
and $\Phi_i$ is odd.
For $i=1$, using
$\phi_{1,n}=u_{p_n}/u_{p_n}(x_n)$, we similarly obtain
\begin{align*}
p_nY_{1,n}\longrightarrow{}&
\sqrt e\int_{\mathbb R^2}
\Bigg[
\left(\int_{\mathbb R^2}
\frac{e^{U(w)}}{|\eta-w|^\alpha}\,dw\right)
e^{U(\eta)}T(\eta)
+e^{U(\eta)}
\left(\int_{\mathbb R^2}
\frac{e^{U(w)}T(w)}
{|\eta-w|^\alpha}\,dw\right)
\Bigg]d\eta\\
={}&\sqrt e\int_{\mathbb R^2}(-\Delta T)\,d\eta=0.
\end{align*}
Here $\nabla T(\eta)=O(|\eta|^{-3})$, so the last integral vanishes.
Consequently, \eqref{eq:Ysmall4} holds.
It remains to record that $Z_{i,n}$ stay away
from zero. For the principal eigenfunction, the equation for $u_{p_n}$
gives the exact identity
\[
Z_{1,n}
=\frac{2p_n+1}{u_{p_n}^2(x_n)}
\int_\Omega|\nabla u_{p_n}|^2\,dx
\longrightarrow4(4-\alpha)\pi.
\]
For $i=2,3$, rescaling and applying Lemma \ref{RE2.9} give
\begin{align*}
Z_{i,n}\longrightarrow z_i:={}&
\int_{\mathbb R^2}e^{U(\xi)}\Phi_i^2(\xi)
\left(\int_{\mathbb R^2}
\frac{e^{U(\eta)}}{|\xi-\eta|^\alpha}\,d\eta\right)d\xi+\int_{\mathbb R^2}\int_{\mathbb R^2}
\frac{e^{U(\xi)}\Phi_i(\xi)e^{U(\eta)}\Phi_i(\eta)}
{|\xi-\eta|^\alpha}\,d\xi d\eta\\
={}&\int_{\mathbb R^2}|\nabla\Phi_i|^2\,d\xi
=\frac{2\pi}{3C_\alpha^2}|a^i|^2>0.
\end{align*}
Setting $z_1=4(4-\alpha)\pi$, we have
$Z_{i,n}=z_i+o(1)$ with $z_i>0$ for $i=1,2,3$, and therefore
\[
c_{i,n}=-\frac{Y_{i,n}}{Z_{i,n}}=o(p_n^{-1}),
\qquad i=1,2,3.
\]

\textbf{Step 2: the Rayleigh quotient.}
The projection has the exact identities
\[
\int_\Omega|\nabla\phi|^2dx
=\int_\Omega|\nabla h_n|^2dx
-\sum_{i=1}^3\lambda_{i,n}\frac{Y_{i,n}^2}{Z_{i,n}},\qquad B_{p_n}(\phi,\phi)
=B_{p_n}(h_n,h_n)-\sum_{i=1}^3\frac{Y_{i,n}^2}{Z_{i,n}}.
\]
Both correction terms are $o(p_n^{-2})$, since $Z_{i,n}\to z_i>0$,
$Y_{i,n}=o(p_n^{-1})$ and $\lambda_{i,n}$ is bounded for $i=1,2,3$.
Testing \eqref{eq:Psi4linear} with
$\widehat\chi_n^2\Psi_{4,n}$ therefore gives
\begin{align}\label{eq:testnum}
\int_\Omega |\nabla\phi|^2dx
&=\int_\Omega |\nabla(\widehat\chi_n\Psi_{4,n})|^2dx+o(p_n^{-2})\nonumber\\
&=\int_\Omega |\nabla\widehat\chi_n|^2\Psi_{4,n}^2dx
+p_n\int_\Omega u_{p_n}^{p_n-1}\widehat\chi_n^2\Psi_{4,n}^2A_n dx\nonumber\\
&\quad +(p_n+1)\int_\Omega\int_\Omega\frac{u_{p_n}^{p_n}(x)\widehat\chi_n^2(x)\Psi_{4,n}(x)u_{p_n}^{p_n}(y)\Psi_{4,n}(y)}{|x-y|^\alpha}dxdy+o(p_n^{-2}).
\end{align}
Similarly,
\begin{align}\label{eq:testden}
& p_n\int_\Omega u_{p_n}^{p_n-1}\phi^2A_n dx
+(p_n+1)\int_\Omega\int_\Omega\frac{u_{p_n}^{p_n}(x)\phi(x)u_{p_n}^{p_n}(y)\phi(y)}{|x-y|^\alpha}dxdy\nonumber\\
=&p_n\int_\Omega u_{p_n}^{p_n-1}\widehat\chi_n^2\Psi_{4,n}^2A_n dx
+(p_n+1)\int_\Omega\int_\Omega\frac{u_{p_n}^{p_n}(x)\widehat\chi_n(x)\Psi_{4,n}(x)u_{p_n}^{p_n}(y)\widehat\chi_n(y)\Psi_{4,n}(y)}{|x-y|^\alpha}dxdy+o(p_n^{-2}).
\end{align}
We also need to compare the nonlocal terms in \eqref{eq:testnum} and
\eqref{eq:testden}. Their difference can equivalently be symmetrized as
\[
\frac{p_n+1}{2}\int_\Omega\int_\Omega
\frac{u_{p_n}^{p_n}(x)\Psi_{4,n}(x)
u_{p_n}^{p_n}(y)\Psi_{4,n}(y)}{|x-y|^\alpha}
\bigl(\widehat\chi_n(x)-\widehat\chi_n(y)\bigr)^2dxdy.
\]
The factors $\Psi_{4,n}$ can change sign; no sign is claimed for this
error. We prove the estimate at the scale of the denominator:
\begin{align}\label{eq:nonlocalcut4}
&(p_n+1)\int_\Omega\int_\Omega
\frac{u_{p_n}^{p_n}(x)\widehat\chi_n(x)\Psi_{4,n}(x)u_{p_n}^{p_n}(y)\Psi_{4,n}(y)
(\widehat\chi_n(x)-\widehat\chi_n(y))}{|x-y|^\alpha}dxdy=o(p_n^{-2}).
\end{align}
Put $M_n=u_{p_n}(x_n)$, and for $\xi\in\Omega_n$, set
\[
F_n(\xi)=\left(1+\frac{v_n(\xi)}{p_n}\right)^{p_n},
\qquad
\Theta_n(\xi)=\frac{p_n}{M_n}
\Psi_{4,n}(x_n+\varepsilon_n\xi),
\]
so that
\[
\Theta_n(\xi)=\xi\cdot\nabla v_n(\xi)
+\frac{4-\alpha}{2}\left(1+\frac{v_n(\xi)}{p_n}\right).
\]
Also set
\[
\chi_n^*(\xi)=\widehat\chi_n(x_n+\varepsilon_n\xi)
\]
and extend all three functions by zero outside $\Omega_n$. By \eqref{J2.37},
$|\Theta_n|\leq C$. Moreover, $\Theta_n\to T$ locally uniformly,
$0\leq\chi_n^*\leq1$, and $\chi_n^*\to1$ pointwise. The proof of Lemma
\ref{RE2.9} therefore gives
\[
F_n\Theta_n,\qquad \chi_n^*F_n\Theta_n,
\qquad (\chi_n^*)^2F_n\Theta_n
\longrightarrow e^UT
\]
strongly in $L^{4/(4-\alpha)}(\mathbb R^2)$. Define
\[
\mathcal I_\alpha(f,g)=\int_{\mathbb R^2}\int_{\mathbb R^2}
\frac{f(\xi)g(\eta)}{|\xi-\eta|^\alpha}\,d\xi d\eta.
\]
After rescaling, the left-hand side of \eqref{eq:nonlocalcut4} equals
\[
\frac{(p_n+1)M_n^2}{p_n^3}
\left\{\mathcal I_\alpha((\chi_n^*)^2F_n\Theta_n,F_n\Theta_n)
-\mathcal I_\alpha(\chi_n^*F_n\Theta_n,
\chi_n^*F_n\Theta_n)\right\}.
\]
Both Riesz forms converge to
$\mathcal I_\alpha(e^UT,e^UT)$ by the HLS inequality. Since $M_n\to\sqrt e$, this proves
\eqref{eq:nonlocalcut4}.
The cutoff term satisfies, by \eqref{eq:Psi4bound},
\begin{align}\label{eq:cuterror4}
\int_\Omega |\nabla\widehat\chi_n|^2\Psi_{4,n}^2dx
&\le \frac{C}{p_n^2\log^2(\varepsilon_n/r)}\int_1^{r/\varepsilon_n}\frac{ds}{s}
=o\left(\frac{1}{p_n^2}\right).
\end{align}
On the other hand, by \eqref{eq:Psi4bound}, Lemma \ref{RE2.9} and the limiting equation for $T$,
\[
B_{p_n}(\phi,\phi)
=\frac{e}{p_n^2}\int_{\mathbb R^2}|\nabla T|^2\,d\eta
+o(p_n^{-2})
=\frac{2\pi e(4-\alpha)^2}{3p_n^2}
+o(p_n^{-2}),
\]
where the last equality follows from Lemma \ref{lem:bubble-integrals}. Thus $\phi\not\equiv0$ for all sufficiently large $n$ and is admissible in \eqref{eq:l4minmax}.
Combining this with \eqref{eq:testnum}--\eqref{eq:cuterror4}, we get
\[
\lambda_{4,n}\le 1+o(1).
\]
Since $\lambda_{3,n}\to1$ and $\lambda_{4,n}\ge \lambda_{3,n}$, it follows that $\lambda_{4,n}\to1$.

\textbf{Step 3: identification of the fourth mode.}
Applying Lemma \ref{jmle2.13}, we obtain
\begin{equation*}
\widetilde\phi_{4,n}(\xi)
= \sum_{j=1}^2  \frac{a_j^4 \xi_j}{C_\alpha^2+|\xi|^2}
+ b^4 \frac{C_\alpha^2-|\xi|^2}{C_\alpha^2+|\xi|^2}+ o(1)
\quad \mathrm{in}\ C^1_{\mathrm{loc}}(\mathbb{R}^2).
\end{equation*}
Put
\[
Z_j(\xi)=\frac{\xi_j}{C_\alpha^2+|\xi|^2},\qquad
D(\xi)=\frac{C_\alpha^2-|\xi|^2}{C_\alpha^2+|\xi|^2},
\]
and set
\[
\Phi_4=\sum_{j=1}^2a_j^4Z_j+b^4D,\qquad
\Phi_k=\sum_{j=1}^2a_j^kZ_j\quad(k=2,3).
\]
Choose the eigenfunctions to be $B_{p_n}$-orthogonal. Rescaling the identities
$B_{p_n}(\phi_{4,n},\phi_{k,n})=0$ and using Lemma \ref{RE2.9}, we obtain
\[
0=\int_{\mathbb R^2}\nabla\Phi_4\cdot\nabla\Phi_k,
\qquad k=2,3.
\]
Direct computations give
\[
\int_{\mathbb R^2}\nabla D\cdot\nabla Z_j=0,
\qquad
\int_{\mathbb R^2}\nabla Z_h\cdot\nabla Z_j
=\frac{2\pi}{3C_\alpha^2}\delta_{hj}.
\]
Since $a^4\perp a^2,a^3$ and $\{a^2,a^3\}$ is a basis of $\mathbb R^2$, we have $a^4=0$. Lemma \ref{jmle2.13} gives $b^4\ne0$; choosing signs, we take $b=b^4>0$. Proposition \ref{jmp2.14} then yields \eqref{jm1.10} and \eqref{jm1.9} along the subsequence.
\end{proof}

\begin{proof}[Proof of Theorem \ref{th3}]
By Proposition \ref{prop:least-energy-morse}, $m(u_{p_n})=1$ and $\lambda_{2,n}\ge1$. Hence \eqref{jm1.8} implies $\mu_1\ge0$, so $m(x_\infty)=0$. Theorem \ref{th2} gives $\mu_2>0$ and $\lambda_{3,n}>1$ for all sufficiently large $n$. If $\mu_1>0$, then $\lambda_{2,n}>1$ and $m_0(u_{p_n})=1$; if $\mu_1=0$, then $m_0(x_\infty)=1$ and $m_0(u_{p_n})\le2$. Thus
\[
1+m(x_\infty)\le m(u_{p_n})=1
\le m_0(u_{p_n})\le1+m_0(x_\infty)\le2.
\]
If $x_\infty$ is nondegenerate, then $D^2R(x_\infty)$ is positive definite. Consequently, $x_\infty$ is a strict local minimum of $R$, and $u_{p_n}$ is nondegenerate with
\[
m(u_{p_n})=m_0(u_{p_n})=1
\]
for all sufficiently large $n$.
\end{proof}

\begin{appendices}
\section{Auxiliary estimates}
This appendix collects the standard results and auxiliary estimates used in the proofs.

\begin{lemma}\cite{Lieb}\label{A1}
Suppose $N \geq 1$, $\alpha \in (0,N)$ and $\theta, r > 1$ with $\frac{1}{\theta} + \frac{1}{r} + \frac{\alpha}{N} = 2$.
Let $f \in L^\theta(\mathbb{R}^N)$ and $g \in L^r(\mathbb{R}^N)$.
Then there exists a constant $C(\theta,r,\alpha,N)>0$ such that
\begin{equation*}
\left|\int_{\mathbb{R}^N} \int_{\mathbb{R}^N} \frac{f(x)g(y)}{|x - y|^\alpha} dx dy\right|
\leq C(\theta,r,\alpha,N) \|f\|_{L^\theta(\mathbb{R}^N)} \|g\|_{L^r(\mathbb{R}^N)}.
\end{equation*}
\end{lemma}
\begin{lemma}\cite[Lemma 3.5]{Gao}\label{fl2}
There exists $C>0$ such that, uniformly for $x\in\Omega_n$,
\[
\int_{\Omega_n}
\frac{\left(1+\frac{v_n(y)}{p_n}\right)^{p_n+1}}{|x-y|^\alpha}\,dy
\leq C,
\]
and
\[
\lim_{R\to\infty}\limsup_{n\to\infty}
\sup_{x\in\Omega_n}
\int_{\Omega_n\setminus B(x,R)}
\frac{\left(1+\frac{v_n(y)}{p_n}\right)^{p_n+1}}{|x-y|^\alpha}\,dy
=0.
\]
\end{lemma}

\begin{lemma}\cite[Lemma 4.13 and Proposition 4.14]{Gao}\label{lemmaA4} For any $\varepsilon > 0$, there exist $R_\varepsilon > 1$, $n_\varepsilon > 1$ and $C_\varepsilon > 0$ such that
\begin{equation*}
v_n(x) \leq \left( \frac{\beta_n}{2\pi} - \varepsilon \right) \log \frac{1}{|x|} + C_\varepsilon,
\end{equation*}
for any $2R_\varepsilon \leq |x| \leq \frac{r}{\varepsilon_n}$, where $r$ is
the fixed radius chosen after Theorem \ref{TA1} and
$n \geq n_\varepsilon$. Moreover, $\beta_n$ is defined by
\begin{equation}\label{MA4.141}
\beta_n := \displaystyle\int_{B_{\frac{r}{\varepsilon_n}}(0)}
\left(\displaystyle\int_{\Omega_n} \frac{\left(1+\frac{v_n(y)}{p_n}\right)^{p_n+1}}
{|x-y|^\alpha} dy \right) \left(1+\frac{v_n(x)}{p_n}\right)^{p_n} dx\longrightarrow 2(4-\alpha)\pi\,\text{ as } n \to +\infty.
\end{equation}
\end{lemma}
We shall use the following standard Green-function estimates; compare the
proof of \cite[Proposition 2.13]{De}.
\begin{lemma}
Let $\delta\in(0,r)$, where $r$ is the fixed radius chosen after Theorem
\ref{TA1}. The Green function satisfies
\begin{equation*}
\sup_{x \in \overline{\Omega} \setminus B_{2\delta}(x_\infty)} \|\nabla G(x,\cdot)\|_{L^\infty(B_\delta(x_\infty))} < +\infty,
\end{equation*}
and for every $x \in \overline{\Omega} \setminus B_{2\delta}(x_\infty)$,
it follows from \eqref{eq:2.12} that, as $n \to +\infty$,
\begin{equation*}
|G(x,x_n) - G(x,x_\infty)| \leq \sup_{x \in \overline{\Omega} \setminus B_{2\delta}(x_\infty)} \|\nabla G(x,\cdot)\|_{L^\infty(B_\delta(x_\infty))} |x_n - x_\infty| = o(1).
\end{equation*}
Moreover,
\begin{equation}\label{eq:2.68}
\sup_{x \in \overline{\Omega} \setminus B_{2\delta}(x_\infty)} \|G(x,\cdot)\|_{L^1(\Omega)} < +\infty,
\end{equation}
and
\[
\sup_{x\in\overline\Omega}
\int_\Omega\left(|G(x,y)|+|\nabla_xG(x,y)|\right)\,dy\leq C.
\]
If $K_1,K_2\subset\overline\Omega$ are compact and
$\operatorname{dist}(K_1,K_2)>0$, then
\[
\sup_{(x,y)\in K_1\times K_2}
|\partial_x^\beta\partial_y^\gamma G(x,y)|\leq C_{K_1,K_2}
\]
for $|\beta|\leq1$ and $|\gamma|\leq2$.
In addition,
\[
\sup_{x \in \overline{\Omega} \setminus B_{2\delta}(x_\infty)} \|G(x,\cdot)\|_{L^\infty(B_\delta(x_\infty))} < +\infty,
\]
so that in particular, for $x \in \overline{\Omega} \setminus B_{2\delta}(x_\infty)$ and $n$ large,
 one has
\begin{equation*}
|G(x,x_n)| \leq \sup_{x \in \overline{\Omega} \setminus B_{2\delta}(x_\infty)} \|G(x,\cdot)\|_{L^\infty(B_\delta(x_\infty))} < +\infty.
\end{equation*}
\end{lemma}
\begin{lemma}
Under the standing assumptions, for every compact set
$K\subset\overline\Omega\setminus\{x_\infty\}$, we have
\begin{equation}\label{wenC4.5}
p_n^{\frac{p_n+1}{2}} \int_K \int_\Omega \frac{u_{p_n}^{p_n+1}(y) u_{p_n}^{p_n+1}(x)}
{|x-y|^\alpha} \, dxdy\longrightarrow0
\quad\text{as }n\to+\infty.
\end{equation}
\end{lemma}
\begin{proof}
This is \eqref{eq225} with $a=b=1$, $m=0$, and
$\gamma_n=(p_n+1)/2$.
\end{proof}
\begin{lemma}\label{lem:bubble-integrals}
Set $q=4-\alpha$, $C=C_\alpha$, and
\[
D(\xi)=\frac{C^2-|\xi|^2}{C^2+|\xi|^2},\qquad
Z_j(\xi)=\frac{\xi_j}{C^2+|\xi|^2},\qquad
T=\frac q2D.
\]
Then
\begin{align*}
\int_{\mathbb R^2}|\nabla D|^2&=\frac{8\pi}{3},&
\int_{\mathbb R^2}|\nabla T|^2&=\frac{2\pi q^2}{3},\\
\int_{\mathbb R^2}\nabla Z_j\cdot\nabla Z_k
&=\frac{2\pi}{3C^2}\delta_{jk},&
\int_{\mathbb R^2}\nabla D\cdot\nabla Z_j&=0,\\
\int_{\mathbb R^2}U D(-\Delta U)&=\frac{\pi q^2}{2},&
\int_{\mathbb R^2}U(-\Delta D)&=2\pi q,\\
\int_{\mathbb R^2}(\xi\cdot\nabla U)(-\Delta D)
&=\frac{4\pi q}{3}.&&
\end{align*}
In addition,
\begin{equation}\label{eq:bubble-riesz-dilation}
\int_{\mathbb R^2}U(\eta)e^{U(\eta)}
\left(\int_{\mathbb R^2}
\frac{e^{U(\xi)}D(\xi)}{|\xi-\eta|^\alpha}\,d\xi\right)d\eta
=\frac{\pi}{2}\alpha(4-\alpha).
\end{equation}
\end{lemma}
\begin{proof}
Direct differentiation gives
\[
-\Delta U=\frac{2qC^2}{(C^2+|\xi|^2)^2},\qquad
-\Delta D=\frac{8C^2(C^2-|\xi|^2)}{(C^2+|\xi|^2)^3},\qquad
-\Delta Z_j=\frac{8C^2\xi_j}{(C^2+|\xi|^2)^3}.
\]
 For the logarithmic integrals, use
$t=|\xi|^2/C^2$ and
\[
\int_0^\infty\frac{(1-t)\log(1+t)}{(1+t)^3}\,dt=-\frac12.
\]
This gives
\[
C^2\int_{\mathbb R^2}
\frac{U(\xi)(C^2-|\xi|^2)}{(C^2+|\xi|^2)^3}\,d\xi
=\frac{\pi q}{4},
\]
and hence the two asserted integrals involving $U$.
Since $\xi\cdot\nabla U=(q/2)(D-1)$ and
$\int_{\mathbb R^2}(-\Delta D)=0$, the remaining integral is
$(q/2)\int|\nabla D|^2=4\pi q/3$.
For clarity, the logarithmic growth of $U$ must be respected when
integrating by parts. Although
$\int_{\partial B_R}U\partial_\nu D\to0$, one has
\[
\int_{\partial B_R}D\partial_\nu U\longrightarrow2\pi q.
\]
Thus moving the Laplacian from $D$ onto $U$ without this boundary term
would give an incorrect value of $\int U(-\Delta D)$.
Finally, the limiting linearized equation for $D$ gives
\[
e^{U(\eta)}\int_{\mathbb R^2}
\frac{e^{U(\xi)}D(\xi)}{|\xi-\eta|^\alpha}\,d\xi
=-\Delta D(\eta)-D(\eta)(-\Delta U(\eta)).
\]
Multiplication by $U$ and the preceding integrals yield
$2\pi q-\pi q^2/2=\pi\alpha q/2$, proving
\eqref{eq:bubble-riesz-dilation}.
\end{proof}
\begin{lemma}\label{FULU}
Suppose that, along a subsequence, $\lambda_{i,n}\to1$ and
\eqref{jm2.56} holds with $b^i\ne0$, where $i\in\mathbb N$. Then,
\begin{align*}
\mathscr A_n
:=p_n(p_n+1)\int_\Omega\int_\Omega
\frac{u_{p_n}^{p_n}(y)u_{p_n}^{p_n}(x)\phi_{i,n}(x)}
{|x-y|^\alpha}\,dxdy=-\frac{\pi}{2}\alpha(4-\alpha)b^i+o(1).
\end{align*}
\end{lemma}
\begin{proof}
Put $M_n=u_{p_n}(x_n)$, $q=4-\alpha$ and
$w_n=1+v_n/p_n$. We use the orthogonality identity
\eqref{app1} and
$u_{p_n}(y)=M_n+(u_{p_n}(y)-M_n)$. This gives
\begin{equation}\label{app2}
\begin{aligned}
M_n\int_\Omega\int_\Omega
\frac{u_{p_n}^{p_n}(y)u_{p_n}^{p_n}(x)\phi_{i,n}(x)}
{|x-y|^\alpha}\,dxdy=-\int_\Omega\int_\Omega
\frac{u_{p_n}^{p_n}(y)u_{p_n}^{p_n}(x)\phi_{i,n}(x)
(u_{p_n}(y)-M_n)}{|x-y|^\alpha}\,dxdy.
\end{aligned}
\end{equation}
Consequently,
\begin{equation}\label{app4}
\mathscr A_n=-\frac{p_n(p_n+1)}{M_n}
\int_\Omega\int_\Omega
\frac{u_{p_n}^{p_n}(y)u_{p_n}^{p_n}(x)\phi_{i,n}(x)
(u_{p_n}(y)-M_n)}{|x-y|^\alpha}\,dxdy.
\end{equation}
Changing variables $x=x_n+\varepsilon_n\xi$ and
$y=x_n+\varepsilon_n\eta$, and using
$\varepsilon_n^{4-\alpha}M_n^{2p_n}=p_n^{-1}$, gives the exact rescaling
\begin{equation}\label{app5}
\mathscr A_n=-\frac{p_n+1}{p_n}
\int_{\Omega_n}\int_{\Omega_n}
\frac{w_n^{p_n}(\eta)v_n(\eta)
w_n^{p_n}(\xi)\widetilde\phi_{i,n}(\xi)}{|\xi-\eta|^\alpha}
\,d\xi d\eta.
\end{equation}
By Lemmas \ref{lem:log-growth}, \ref{RE2.9} and
\eqref{eq:strong-weighted-HLS}, hence
\begin{align*}
\mathscr A_n=-\int_{\mathbb R^2}U(\eta)e^{U(\eta)}
\left(\int_{\mathbb R^2}
\frac{e^{U(\xi)}\Phi_i(\xi)}{|\xi-\eta|^\alpha}\,d\xi\right)d\eta
+o(1),\qquad
\Phi_i=\sum_{j=1}^2a_j^iZ_j+b^iD,
\end{align*}
where $Z_j$ and $D$ are as in Lemma \ref{lem:bubble-integrals}.
 Using the limiting linearized
equation and then Lemma
\ref{lem:bubble-integrals}, we conclude that
\begin{align*}
\mathscr A_n
=-b^i\int_{\mathbb R^2}U(-\Delta D)
+b^i\int_{\mathbb R^2}UD(-\Delta U)+o(1)=-b^i\left(2\pi q-\frac{\pi q^2}{2}\right)+o(1)
=-\frac{\pi}{2}\alpha(4-\alpha)b^i+o(1).
\end{align*}
\end{proof}
\end{appendices}

\noindent{\bf Acknowledgements}

This research was supported by the National Natural Science Foundation of China (No. 12371121, 12671143) and the Graduate Research Innovation Project of Southwest University (No. SWUB25030).

\noindent{\bf Conflict of interest}

On behalf of all authors, the corresponding author states that there is no conflict of interest.

\noindent{\bf Data availability}

Data sharing is not applicable to this article as no data were created or analyzed in this study.

\end{document}